\documentclass[11pt]{article}
\usepackage[english]{babel}
\usepackage{amsfonts}
\usepackage[utf8]{inputenc}
\usepackage{amsthm}
\usepackage{mathrsfs}
\usepackage{amsmath}
\usepackage{mathrsfs}
\usepackage{enumerate}
\usepackage{placeins}
\usepackage{cite}
\usepackage{bbm}
\usepackage{makeidx}
\usepackage{xcolor}
\usepackage{graphicx}
\usepackage{frontespizio}
\usepackage[makeroom]{cancel}
\usepackage[normalem]{ulem}
\usepackage{amssymb}
\usepackage[a4paper, total={7in, 9in}]{geometry}
\usepackage{hyperref}
\usepackage{mathtools}
\usepackage{enumitem}
\usepackage[toc,page]{appendix}
\usepackage{esint}
\usepackage{tcolorbox}

\author{Nicola Gigli, Ivan Yuri Violo}

\newcommand{\eps}{\varepsilon}

\newcommand{\rr}{\mathbb{R}}
\newcommand{\nn}{\mathbb{N}}

\newcommand{\nchi}{{\raise.3ex\hbox{$\chi$}}}
\newcommand{\bigslant}[2]{{\raisebox{.2em}{$#1$}\left/\raisebox{-.2em}{$#2$}\right.}}
\newcommand{\sfd}{{\sf d}}
\newcommand{\Lip}{{\rm Lip}}
\newcommand{\id}{{\sf id}}
\renewcommand{\phi}{\varphi}
\newcommand{\restr}[1]{\lower3pt\hbox{$|_{#1}$}}

\newcommand{\X}{{\rm X}}

\newcommand{\fr}{\penalty-20\null\hfill$\blacksquare$} 
\definecolor{mygray}{gray}{0.9}

\newcommand{\ee}{{\sf e}}

\newcommand{\rcd}{\mathrm{RCD}}

\newcommand{\Z}{{\rm Z}}

\theoremstyle{plain}
\newtheorem{theorem}{Theorem}[section]
\newtheorem{lemma}[theorem]{Lemma}
\newtheorem{prop}[theorem]{Proposition}

\newtheorem{cor}[theorem]{Corollary}

\theoremstyle{definition}
\newtheorem{definition}[theorem]{Definition}
\newtheorem{remark}[theorem]{Remark}
\newtheorem{example}[theorem]{Example}

\numberwithin{equation}{section}

\title{Notes on the Cheeger and Colding  version  of the Reifenberg theorem for metric spaces}

\author{Nicola Gigli\footnote{\href{mailto:ngigli@sissa.it}{ngigli@sissa.it}, SISSA, Via Bonomea 265, 34136 Trieste (TS), Italy},
\,Ivan Yuri Violo\footnote{\href{mailto:ivan.violo@sns.it}{ivan.violo@sns.it}, Centro di Ricerca Matematica Ennio De Giorgi, Scuola Normale Superiore, Piazza dei
Cavalieri 3, 56126 Pisa (PI), Italy.}}

\date{}
\begin{document}

\maketitle

\begin{abstract}
   The classical Reifenberg's theorem says that a set which is sufficiently well approximated by planes  uniformly at all scales is  a topological H\"older manifold. Remarkably,  this generalizes to metric spaces, where the approximation by planes is replaced by the Gromov-Hausdorff distance. This fact was shown by Cheeger and Colding in an appendix of one of their celebrated works on  Ricci limit spaces \cite{Cheegercolding}. 

   Given the recent interest around this statement in the growing field of  analysis in metric spaces, in this note we provide a self contained  and detailed proof of the Cheeger and Colding result. Our presentation substantially expands  the arguments in \cite{Cheegercolding} and makes explicit all the relevant estimates and constructions. As a byproduct we also shows a biLipschitz version of this result  which, even if folklore among experts, was not present in the literature.

   This work is an extract from the doctoral dissertation of the second author.
\end{abstract}

\tableofcontents

\section{Introduction}
The celebrated Reifenberg's theorem \cite{r}  states that if a set in $\rr^d$ is well approximated at every small scale by $n$-dimensional affine planes, then it is (locally) an $n$-dimensional bi-H\"older manifold.  More precisely, for a set $S\subset  \rr^d$ and $n \in \mathbb{N}$, with $n < d$,  we set
\[
\ee(x,r)\coloneqq r^{-1}\inf_{\Gamma}  \sfd_H(S\cap B_{r}(x),\Gamma\cap B_{r}(x)), \quad \text{for every $r>0$ and $x \in S$},\\
\]
where  $ \sfd_H$ is the Hausdorff distance and where the infimum is  taken among all the  $n$-dimensional affine planes $\Gamma$ in $\rr^d$ containing $x$. Then Reifenberg's result reads as follows:
\begin{theorem}[Classical Reifenberg's theorem,\cite{r}]\label{thm:reif intro}
	For every $n,d\in \nn$  with $n<d$  and $\alpha\in(0,1)$ there exists $\delta=\delta(n,d,\alpha)$ such that the following holds. Let $S\subset \rr^d$ be closed, containing the origin and such that $\ee(x,r)<\delta$ for every $x \in S\cap B_1(0)$ and $r \in(0,1)$.
	
	Then there exists an $\alpha$-bi-H\"older homeomorphism $F: \Omega\to S\cap B_{1/2}^{\rr^d}(0)$, where $\Omega$ is an open set in $\rr^n.$
\end{theorem}
A set $S$ which satisfies the hypotheses of the previous result is called\textit{ Reinfenberg-flat} (in $B_1(0)$). For a short proof of the above result and an explanation of its core ideas we refer to \cite{simondisk} (see also  \cite{morrey}).

The original motivation in \cite{r} to prove this result was the regularity of minimal surfaces. However, from its original formulation,  Reifenberg's theorem and more in general the idea behind its proof have found successful generalizations and applications in harmonic analysis, geometric measure theory, rectifiability theory and PDE's (see e.g. \cite{holes,nv,env,tolsa,tolsa2,DS,jones,badger,toro,Naberlectures} and the references therein).

We will be interested in the Reifenberg's theorem for metric spaces. The generalization of Theorem \ref{thm:reif intro} in this setting has been obtained in a celebrated result by Cheeger and Colding \cite[Appendix A]{Cheegercolding}.  To state it we need to define the metric-analogue of the `flatness'-coefficients $\ee(x,r)$ in Reifenberg's theorem. In this case the comparison with planes, which clearly is not available, is replaced with the Gromov-Hausdorff distance to Euclidean balls.  In particular for a metric space $(Z_1,\sfd_1)$ and a fixed $n \in \nn$ we set
\[
\eps(z,r)\coloneqq r^{-1}\sfd_{GH}(B_r(z),B_r^{\rr^n}(0)), \quad \text{for every $r>0$ and $z \in Z$,}
\]
(we refer to Section \ref{sec:preliminaries chp2} for the definition of $\sfd_{GH}$).

\begin{theorem}[{Reifenberg's theorem for metric spaces, \cite[Theorem A.1.2]{Cheegercolding}}]\label{localthmholder}
	For every $n \in \mathbb{N}$ there exist constants $\eps(n)>0, M=M(n)>1$ such that the following is true.  	Let $B_1(z_0)$ be a a ball inside a complete metric space $(Z,\sfd)$. Suppose that  $\eps(z,r) \le \eps(n)$ for every $z \in B_1(z_0)$ and all $r<1-\sfd(z,z_0)$. 
	Then there exists a   bi-H\"older map $F: B^{\rr^n}_1(0) \to Z $  such that $B_{1-M\eps_0}(z_0)\subset F(B_1(0)).$ 
\end{theorem}
The original motivation of Cheeger and Colding to prove the metric-Reifenberg's theorem was to prove manifold-regularity of the regular set of ``non-collapsed" Ricci-limit spaces \cite{Cheegercolding}. However, as a consequence of their construction, they obtained also several improvements of some previous stability results for Riemannian manifolds with Ricci curvature bounded below (see \cite{perelman,cosphere2,coconvergence}). In addition, this result has found recently many different applications in the theory of  metric measure spaces with synthetic Ricci curvature lower bounds (see for instance  \cite{ricci2,mk,boundary}).  It is worth  mentioning that also the more general Reifenberg-type constructions connected  to  rectifiability (as in \cite{nv}) have found successful applications on spaces with Ricci bounds (see \cite{cjn,boundary} using ideas also from \cite{NJ}) and in theory of uniform rectifiability for metric spaces \cite{BHS23}.   
We finally mention  \cite{snow}  where Reinfenberg-flat metric spaces have been further studied.

Returning to the Euclidean case,  it is  classical that if the approximation by planes  improves sufficiently fast as the scale decreases, the bi-H\"older regularity  in Reifenberg's result can be improved to bi-Lipschitz. This  goes back to Toro \cite{toro} and the same idea has been further refined and developed in the context of rectifiability e.g.\ in \cite{DS,nv,env,holes}. The right decay  condition turns out to be the square summability of the numbers $\ee(x,r)$ along dyadic scales.
\begin{theorem}[{Classical Reifenberg's theorem - biLipschitz version, \cite{toro}}]\label{thm:toro}
	For every $\eps>0$, $n,d\in \nn$  with $n<d$, there exists $\delta=\delta(n,d,\eps)>0$ such that the following holds. Let $S\subset \rr^d$ be closed, containing the origin and such that
	\begin{equation}\label{eq:dini1}
		\sum_{i=1}^\infty \sup_{x \in S\cap B_1(0)} \ee(x,2^{-i})^2<\delta.
	\end{equation}
	Then there exists a (1+$\eps$)-bi-Lipschitz homeomorphism $F: \Omega\to S\cap B_{1/2}^{\rr^d}(0)$, where $\Omega$ is an open set in $\rr^n.$
\end{theorem}
The square summability assumption is  sharp as shown by suitable \emph{snowflake-constructions}. 

It is now understood that a similar improvement can be obtain also in metric spaces, only that the correct assumption is the summability of the numbers $\eps(z,r)$ along dyadic scales, rather than their squares. Very roughly speaking the reason boils down to the Pythogorean theorem, which implies that the projection of a (sufficiently flat) set onto a plane, contained in its $\delta$-neighbourhood, distorts distances only up to $\delta^2$. We refer to \cite{snow,Violoflat} for a more detailed discussion on the discrepancy between summability and square summability in the Euclidean and metric setting. 

Let us summarize what we said so far about the possible variants of Reifenberg's theorem:
\begin{itemize}
    \item \textbf{$\rr^n$ vs metric space}: the object of the statement can e either a closed subset of $\rr^n$ (classical Reifenberg's theorem) or a complete metric space (Cheeger and Colding metric version).
    \item \textbf{Local vs global}: the Reifenberg-flatness condition can be required inside a ball or at all locations. In the first case we get that a slightly smaller ball is homeomorphic to an Euclidean ball, in the second that the whole set (or space) is homeomorphic to a Riemannian manifold.
    \item \textbf{biH\"older vs biLipschitz}: if the flatness parameters  ($\ee(x,r)$ or $\eps(z,r)$) are only required to be sufficiently small, then the homeomorphism   is biH\"older, if instead we require square summability (in $\rr^n$) or summability (in metric spaces) we can get a biLipschitz homeomorphism.
\end{itemize}

In the above terminology the classical Reifenberg's theorem is the local and biH\"older version for sets in $\rr^n$ stated in Theorem \ref{thm:reif intro} and proved in \cite{r}. The local biLipschitz version in $\rr^n$ is instead Theorem \ref{thm:toro} and proved by Toro \cite{toro}. The global counterparts of these results in  $\rr^n$ are well known to be consequences of the local ones, but are  rarely found in the literature. 
In the setting of metric spaces the local biH\"older version  is instead contained in Theorem \ref{localthmholder} by Cheeger and Colding \cite{Cheegercolding}. Actually in \cite{Cheegercolding} they prove the global biH\"older version for metric spaces, while the modifications to obtain the local version are left to the reader. Finally the biLipschitz versions  in metric spaces (both local and global) are folklore among the experts, even if they did not appear in the literature so far, and can be deduced using the same arguments in \cite{Cheegercolding}.

In fact our main goal is to give a completely self contained proof of the both the biLipschitz and H\"older  version for metric spaces in the global case (see Theorem \ref{mainthm} which contains both statements). 

We will do so for two reasons:
\begin{itemize}
	\item to get the bi-Lipschitz regularity  of the map it is necessary that the error-order of every estimate in the proof is quantified and therefore we need to make explicit every step,
	\item given the recent growing interest in the metric version of Reifenberg's theorem  we believe it useful to have available an expanded version, also with a somewhat different presentation, of the arguments in \cite{Cheegercolding}.
\end{itemize}

Finally we will give an outline on how to modify the proof of the global case to obtain the local one (see Theorem \ref{localthm} and Section \ref{sec:local proof}).

\medskip

\textbf{Acknowledgements: } 
The authors are grateful to Guy David, Guido De Philippis and Tatiana Toro for several useful conversations on the topic of this note.

\section{Main statements and remarks}\label{sec:bilipremark}
To state the main result we introduce the dyadic Gromov-Hausdorff `flatness-coefficients'.
\begin{definition}\label{ghjonest}
	Let $(Z,\sfd)$ be a metric space. For every $r>0$, $n \in \mathbb{N}$ and  $i\in \mathbb{N}$ define 
	\begin{equation}\label{ghjones}
		\eps_{i}(r,n)\coloneqq\frac{2^i}{r}\, \sup_{z \in Z} \sfd_{GH}(B^Z_{2^{-i}r}(z),B^{\rr^n}_{2^{-i}r}(0)).
	\end{equation}
\end{definition}

\begin{theorem}[Reifenberg's theorem in metric spaces-global version]\label{mainthm}
	For every  $n \in \mathbb{N}$ the following holds.  Let $(Z,\sfd)$ be a connected and complete metric space and fix $r>0.$
\begin{enumerate}[label=\roman*)]
    \item \emph{biH\"older version:} if $\eps_i(r,n)\le \eps(n)$ for all $i\ge 0$ there exists  a smooth $n$-dimensional Riemannian manifold $W$ and a surjective (uniformly) locally  biH\"older homeomorphism $F: W \to Z $. More precisely there exists $\alpha\in (0,1)$ such that, setting $\rho\coloneqq \frac{r}{800}$, for all $w \in W$ and $z\in \Z$ the maps $F|_{B_\rho(w)}$,  $F^{-1}|_{B_{\rho}(z)}$ are $\alpha$-biH\"older with uniform constants.
    \item  \emph{biLipschitz version:}	if \begin{equation}\label{sumscales}
		\sum_{i=0}^\infty \eps_{i}(r,n) < +\infty,
	\end{equation}
	then for every $\eps>0$ there exists  a smooth $n$-dimensional Riemannian manifold $W_\eps$ and surjective (uniformly) locally  $(1+\eps)$-biLipschitz homeomorphism $F_{\eps}: W_{\eps} \to Z $. More precisely there exists $\rho>0$ such that 
	\begin{equation}\label{finalbilip}
		\Lip\ F_{\eps}|_{B_{\rho}(w)}, \,\, \Lip\ {F_{\eps}}^{-1}|_{B_{\rho}(z)} <1+\eps,
	\end{equation}
	for every $w \in W_{\eps}$ and every $z \in Z.$ 
 Finally for every $\eps_1,\eps_2>0$  the manifolds $W_{\eps_1},W_{\eps_2}$ can be chosen to be diffeomorphic to each other. 
\end{enumerate} 

\end{theorem}

\begin{remark}
	Unlike \cite{Cheegercolding} we are not assuming that $(Z,\sfd)$ is separable. Indeed separability is a consequence of the assumptions of Theorem \ref{mainthm}, as we will prove at the beginning of Section \ref{sec:manifolds}. \fr
\end{remark}

\begin{remark}
	Observe that, since $W_{\eps}$ is complete and connected, the map $F_{\eps}$ is always globally $1+\eps$ Lipschitz. Analogously, if $(Z,\sfd)$ is a length space, we also obtain that $F_{\eps}^{-1}$ is globally $1+\eps$ bi-Lipschitz. On the contrary, without further assumptions of $Z$ we can only say that ${F_{\eps}}^{-1}$ is Lipschitz on any bounded set, with a Lipschitz constant that depends only on the diameter of the set. In particular if $Z$ is bounded, then $F_{\eps}$ is always bi-Lipschitz. \fr
\end{remark}

A local version of Theorem \ref{mainthm}, more in the spirit of Theorem \ref{thm:toro}  also holds.  To state it we need the following variants of the numbers $\eps_i$ defined above.

\begin{definition}\label{ghjonesloct}
	Let $(Z,\sfd)$ be a metric space and $z_0 \in Z$. For every   $n \in \mathbb{N}$ and every $i\in \mathbb{N}$ define 
	\begin{equation}\label{ghjonesloc}
		\eps_i(n)\coloneqq 2^i\, \sup_{z \in B_{1-2^{-i}}(z_0)} \sfd_{GH}(B^Z_{2^{-i}}(z),B^{\rr^n}_{2^{-i}}(0)).
	\end{equation}
\end{definition}

\begin{theorem}[Bi-Lipschitz metric-Reifenberg's theorem - local version]\label{localthm}
	For every $n \in \mathbb{N}$ there exist constants $\eps(n)>0, M=M(n)>1$ such that the following is true.  	Let $B_1(z_0)$ be a a ball inside a complete metric space $(Z,\sfd)$. Suppose that  $\eps_i(n) \le \eps(n)$ for every $i \ge 0$. Then there exists a   biH\"older map $F: B^{\rr^n}_1(0) \to Z $  such that $B_{1-M\eps_0}(z_0)\subset F(B_1(0)).$ 
 If moreover
	\begin{equation}\label{sumscales2}
		\sum_{i= 1}^\infty \eps_i(n) < +\infty,
	\end{equation}
	then $F$ can be taken to be biLipschitz.
\end{theorem}
We observe that, contrary to Theorem \ref{mainthm}, in Theorem \ref{localthm} we do not have an arbitrary small bi-Lipschitz constant. This is due to the fact that in Theorem \ref{mainthm} we have the freedom to choose the manifold to which we compare the metric space, while in Theorem \ref{localthm} the manifold is fixed to be the unit Euclidean ball. We will not include the proof of Theorem \ref{localthm}, which in any case is a minor modification of the one that we will present for Theorem \ref{mainthm}. 

It is also worth to point out that we cannot improve  Theorem \ref{localthm} to have $B_{1}(z_0)=F(B_1(0)),$ indeed it might even  happen  that $B_{1}(z_0)$ is not connected, as shown in the following example.

\begin{example}\label{twoballs}
	Fix $\eps>0$ small and $n\in \mathbb{N}$.
	Consider the metric space $(Z,\sfd)$ obtained as  union of two closed Euclidean balls $B_1:=\overline{B^{\rr^n}_{1}(0_1)},B_2:=\overline{B^{\rr^n}_{1}(0_2)}$ by setting $\sfd(0_2,0_1)=1-\eps$. It is straight-forward to verify that $B^Z_1(0_1)$ satisfies the hypotheses of Theorem \ref{localthm} with $\eps_0=3\eps$ and $\eps_i=0$ for every $i \ge 1.$ However $B^Z_s(0_1)$ is disconnected for every $1-\eps<s\le 1$. \fr
\end{example}

\section{A Corollary: Gromov-Hausdorff close  and Reifenberg flat metric spaces are homeomorphic}\label{sec:cor}

From the proof of Theorem \ref{mainthm} it is possible to deduce the following important corollary, originally stated in \cite[Theorem A.1.3]{Cheegercolding}. 
\begin{cor}\label{corollary}
For every $n\in \nn$ there exists $\delta(n)>0$ such that the following holds.
    Let $(Z_1,\sfd_1)$ and $(Z_2,\sfd_2)$ be two metric spaces satisfying the assumptions of Theorem \ref{mainthm}-$i)$ at scale $r>0$. Suppose also that 
    \begin{equation}\label{eq:corollary assumption}
        \sfd_{GH}((Z_1,\sfd_1),(Z_2,\sfd_2))<\delta(n) r.
    \end{equation}
    Then the Riemannian manifold $W$ in the conclusion of Theorem \ref{mainthm}-$i)$ can be taken to be the same for both $Z_1$ and $Z_2.$ In particular $(Z_1,\sfd_1)$ and $(Z_2,\sfd_2)$ are (locally) biH\"older homeomorphic. Moreover if $Z_1$ and $Z_2$ are also smooth compact Riemannian manifolds (endowed with the geodesic distance) than they are actually diffeomorphic.
\end{cor}
This result have been recently applied several times  to deduce topological stability properties of RCD spaces (see \cite{KM21,HM21} and also \cite{HP23,D22}).

The idea behind Corollary \ref{corollary} is that the topology of a Reifenberg-flat metric space $(Z,\sfd)$ (i.e.\ satisfying the assumptions of Theorem \ref{mainthm}-$i)$) can be essentially deduced looking only at scale $r$. More precisely, inspecting its proof, the manifold $W$ appearing in the statement of Theorem \ref{mainthm}-$i)$ is built using only GH-approximation maps between balls at scale $\sim r$ in $Z$ and $\rr^n.$ Therefore, if $Z_1$ and $Z_2$ are sufficiently close at this scale, i.e.\ $\delta$ in assumption \eqref{eq:corollary assumption} is small enough, then $W$ can be taken to be the same for both.

\section{Preliminaries and notations}\label{sec:preliminaries chp2}
In this short preliminary section  we list some notations and definitions that we will need in the sequel, which are mainly linked to the Gromov-Hausdorff distance between metric spaces and related concepts.

\subsection{Scaling invariant \texorpdfstring{$C^k$}{Ck}-norm and technical estimates}\label{sec:ck norm}
For an $n\times n$ real-matrix $A\in\rr^{n\times n}$ we will denote by $\|A\|$ its operator norm, i.e.\ $\|A\|\coloneqq \sup_{|v|\le 1} |A v|$, and by $\{a_{i,j}\}_{i,j\in{1,...,n}}$ its entries.  We recall also the following elementary inequality, which will be used in several proofs: 
\begin{equation}\label{opnorm}
	\frac{1}{\sqrt n}\|A\|\le \max_{i,j}|a_{i,j}|\le \|A\|, \quad \forall A \in \rr^{n \times n}.
\end{equation}
Indeed by definition of operator norm we have 
$$\max_{i,j}|a_{i,j}|\le \max_j|(a_{1,j},\dots,a_{n,j})|\le \|A\|,$$
while for all $v=(v_1,\dots,v_n)$ with $|v|=1$ we have
\[
|A v|^2=\sum_i \big(\sum_j a_{i,j}v_j\big)^2\le     \left(\max_{i,j}|a_{i,j}|^2\right) \sum_i \sum_j |v_j|^2= \max_{i,j}|a_{i,j}|^2 n.
\]
We will also denote by $I_n$ the $n\times n$-identity matrix.

Given  $f\in C^1(U;\rr^n)$  where $U$ is an open set of $\rr^n,$  we will denote with $f_k$ (or $(f)_k$), $k=1,..,n$ its components and with $Df: U \to \rr^{n\times n}$ its differential, which will be always treated as an ($n\times n$ matrix)-valued map.

A central role will be played by the following norms.
\begin{definition}[$C^1/C^2$-scaling invariant norm]
	Let $U\subset \rr^n$ be open. For all $f\in C^1(U;\rr^n)$ and  $t>0$ we set
	\[ \|f\|_{C^1(U),t}\coloneqq \max\left(\sup_{U}  \frac{|f|}t, \sup_{U}  \|D f\|\right), \]
	 where $\|\cdot \|$ denotes the operator norm of the matrix $Df$. Moreover if $f\in C^2(U;\rr^n)$ we also set
  \[
   \|f\|_{C^2(U),t}\coloneqq \max\left(\|f\|_{C^1(U),t}, \max_{i,j,k\in\{1,\dots,n\}} \sup_{U}   t|\partial_{i,j}f_k | \right).
  \]
\end{definition}

\begin{remark}\label{rmk:scaling invariant}
	The reason for us to work with this norm is the following observation: 
	$\|f\|_{C^1(U),t}\le C$ (resp.  $\|f\|_{C^2(U),t}\le C$) if and only if  $\|f_t\|_{C^1(U/t),1}\le C$ (resp.  $\|f_t\|_{C^2(U/t),1}\le C$),  where $U/t\coloneqq \{x/t \ : x \in U\}$ and $f_t(x)\coloneqq\frac{f(tx)}{t}$.\fr
\end{remark}

\begin{lemma}\label{diffeo}
	Let $f\in C^k(\rr^n;\rr^n)$, $k \in \mathbb{N}$, be such that \[\|Df(x)-I_n\|\le 1/2,\]
	\[\{|f(x)-x| \ | \ x \in \rr^n \} \text{ is bounded }.\]
	Then $f$ is a diffeomorphism (from $\rr^n$ to $\rr^n$) with inverse of class  $C^k$.
\end{lemma}
\begin{proof}
	Fix $x \in \rr^n$. Then for every $v \in \rr^n$
	\[|v|\le |v-Df(x)v|+|Df(x)v|\le 1/2 |v|+|Df(x)v|\]
	hence $1/2|v|\le |Df(x)v|$, thus $Df(x)$ is invertible. Then by the Inverse Function Theorem $f$ is a local diffeomorphism and an open map. Now observe that $f- {\sf{id}} $ is $1/2-$Lipschitz, thus for every $x,y \in \rr^n$
	\begin{equation} \label{eq:lip trick}
 |x-y|-|f(x)-f(y)|\le |({\sf{id}}-f)(x)-({\sf{id}}-f)(y)|\le |x-y|/2,
	\end{equation}
	hence $|x-y|\le 2|f(x)-f(y)| $ and in particular $f$ is injective. Moreover $f^{-1}$ is  $C^k$ again thanks to the Inverse Function Theorem. We claim now that $f$ is proper, i.e.\ that $f^{-1}(K)$ is compact for every $K$ compact. Indeed $K$ is closed and bounded, hence $f^{-1}(K)$ is closed and it is also bounded by second hypothesis on $f$, in particular  it is compact. Since $f$ is proper, it a closed map, but we already observed that it is also open, hence it is surjective. This concludes the proof.
\end{proof}

\begin{lemma}\label{inversebound}
	For every $n \in \mathbb{N}$ there exists a constant $C=C(n)>0$ with the following property. Let $m \in \{1,2\}$. Let $H \in C^k(\rr^n;\rr^n)$, $k \ge m,$ be such that $\|H-{\sf{id}}\|_{C^m(\rr^n),t}\le \eps$, for some $\eps\in (0,1/4)$. Then $H$ is a  diffeomorphism  (from $\rr^n$ to $\rr^n$) with $C^k$ inverse and satisfying
	\[\|H^{-1}-{\sf{id}}\|_{C^m(\rr^n),t}\le C\eps.\]
\end{lemma}
\begin{proof}
	The fact that $H$ is a global  diffeomorphism  of $\rr^n$  with $C^k$ inverse follows from Lemma \ref{diffeo}. To prove the required bound we first observe that
	\[ |H^{-1}-{\sf{id}}|=|({\sf{id}}-H)\circ H^{-1}|\le \eps t.\]
	Moreover 
	\[ \|D(H^{-1})-I_n\|=\|[DH(H^{-1})]^{-1}-I_n\|\le \|[DH(H^{-1})]^{-1}\| \|DH(H^{-1})-I_n\|\le 2\eps,\]
	where we have used the fact that $\|D(H^{-1})(x)\|\le 2$ for every $x\in \rr^n$, as a can be easily derived from $\|DH(x)-I_n\|\le 1/2$ for every $x\in \rr^n$ as we did in \eqref{eq:lip trick}. This settles the case $m=1.$
	For $m=2$, exploiting the formula for the derivative of the inverse matrix, we have
	\begin{align*}
		\partial_j D(H^{-1})&=\partial_j[DH(H^{-1})]^{-1} =[DH(H^{-1})]^{-1} \partial_j (DH(H^{-1}))[DH(H^{-1})]^{-1}\\
		&=D(H^{-1}) \partial_j (DH(H^{-1}))D(H^{-1})
	\end{align*}
	where $\partial_j$ is component-wise derivative. Therefore, recalling \eqref{opnorm} and that $\|DH\|,\|D(H^{-1})\|\le 2$, 
	\begin{align*}
		|\partial_{i,j}(H^{-1})_k|&\le \|\partial_j D(H^{-1})\|\le 4\| \partial_j (DH(H^{-1}))\| \\
		&\le\sqrt{n}\max_{l,m} |\partial_j(\partial_lH_m(H^{-1}))|\\
		&\le \sqrt{n} \max_{l,m} \sum_h |\partial_{h,l}H_m(H^{-1})||\partial_{j}H^{-1}_h|\\
		&\le \sqrt{n}\frac{\eps}{t} \|D(H^{-1})\|\le 2 \sqrt{n}\frac{\eps}{t}.
	\end{align*}
\end{proof}

\begin{lemma}\label{iterateclosest}
	For every $n \in \mathbb{N}$ and every $\delta>0$ there exists a positive constant $C(n,\delta)$ such that the following holds. Fix $m\in \{1,2\}$. Let $f_i \in C^m(U_i ;\rr^n)$ with $i=1,2$, with $U_i \subset \rr^n$ open sets, and let $I_1,I_2 : \rr^n \to \rr^n$ be two isometries. Suppose that $\|f_i-I_i\|_{C^m(U_i),t}\le \delta_i\le \delta$, then letting $W$ be any open set where $f_1\circ f_2$ is defined,
	\begin{equation}\label{iterateclose}
		\|f_1\circ f_2 -I_1\circ I_2\|_{C^m(W),t}\le C(n,\delta)(\delta_1+\delta_2).
	\end{equation}
\end{lemma}
\begin{proof}
	The proof is just a straightforward computation.
	\begin{align*}
		|f_1\circ f_2 -I_1\circ I_2|&\le 	|f_1\circ f_2 -I_1\circ f_2|+	|I_1\circ f_2 -I_1\circ I_2|\le \delta_1t +|f_2-I_2|\le \delta_1t+\delta_2t,
	\end{align*}
	\begin{align*}
		\|	D(f_1\circ f_2)  -D(I_1\circ I_2)\|&= 	\|(Df_1)\circ f_2 Df_2 -(DI_1)\circ I_2 DI_2\|\\
		&\le\|(Df_1)\circ f_2 -(DI_1)\circ I_2\|\|Df_2\|+\|DI_1\|\|Df_2-DI_2\|\\
		&\le \delta_1 (1+\delta_2)+\delta_2\le C(\delta)(\delta_1+\delta_2),
	\end{align*}
 where in the last line we used that $DI_1\circ I_2=DI_1\circ f_2$ because $DI_1$ is constant.
	Moreover, recalling \eqref{opnorm} and that $\|Df_i-I_i\|\le \delta_i,$
	\begin{align*}
		|\partial_{ij}(f_1\circ f_2)_k|&=|\sum_{m,h=1}^n \partial_{hm}(f_1)_k\circ f_2\partial_j(f_2)_m\partial_i(f_2)_k+\sum_{h=1}^n\partial_{ij}(f_2)_k\partial_h(f_1)_k\circ f_2|\\
		&\le  (1+\delta_2)^2\sum_{m,h=1}^n |\partial_{hm}(f_1)_k\circ f_2|+ (1+\delta_1)\sum_{h=1}^n|\partial_{ij}(f_2)_k|\le C(n,\delta)\frac{\delta_1+\delta_2}{t}.
	\end{align*}
\end{proof}
Iterating Lemma \ref{iterateclosest} we get the following.
\begin{lemma}\label{iterateclosest2}
	For every $n \in \mathbb{N}$ and every $\delta>0$ there exists a positive constant $C(n,\delta)$ such that the following holds. Fix $m\in \{1,2\}$. Let $f_1,f_2,f_3,f_4: \rr^n \to \rr^n $ be $C^m$ functions and let $I_1,I_2,I_3,I_4 : \rr^n \to \rr^n$ be  isometries. Suppose that $\|f_i-I_i\|_{C^m(\rr^n),t}\le \delta_i<\delta$, then 
	\begin{equation}
		\|f_1\circ f_2 \circ f_3\circ f_4  -I_1\circ I_2\circ I_3\circ I_4\|_{C^m(\rr^n),t}\le C(n,\delta )(\delta_1+\delta_2+\delta_2+\delta_3+\delta_4).
	\end{equation}
\end{lemma}

\begin{lemma}\label{isombound}
	Let $F : \rr^n \to \rr^n$ be an isometry such that $|F-{\sf{id}}|\le r \eps $ on some ball $B_{r}(x)$, then $\|DF-I_n\|\le 4\eps.$ 
\end{lemma}
\begin{proof}
	We may suppose $x=0.$ Recall that $F=DF\cdot x+v$ for some vector $v$ and notice that $|v|=|F(0)-0|\le r \eps$. Hence $|DF \cdot y-y|\le 2 \eps r$ for every $y \in B_r(0)$. In particular if $|y|=r/2$ we have
	\[ \frac{|DF \cdot y-y|}{|y|}\le 4 \eps,\]
from which the conclusion  follows from the definition of operator norm.
\end{proof}

\begin{lemma}\label{B4}
	Let $f:   U \to V$ be a bijective $C^2$ function  with $C^2$ inverse, where $U,V\subset \rr^n$ are open sets.  Let $g: \rr^n \to \rr^n$ be a $C^1$ function satisfying $\|g-{\sf{id}}\|_{C^1(\rr^n),t}\le \delta<1$ for some $t>0$. Moreover suppose that 
	\[\|Df\|,\|Df^{-1}\|,t|\partial_{ij}(f)_k|,t|\partial_{ij}(f^{-1})_k|\le M, \]
	everywhere on $U$ (for $f$) and everywhere on $V$ (for $f^{-1}$).
	
	Then letting $W\subset \rr^n$ be any open set where the function $f\circ g \circ f^{-1}$ is defined, we have
	\[ \|f\circ g \circ f^{-1}-{\sf{id}}\|_{C^1(W),t} \le \delta c, \]
	where $c=c(n,M)>1$ is positive constant that depends only on $M$ and $n$.
\end{lemma}
\begin{proof}
	Observe first that by the mean value theorem
	\begin{align*}
		|f\circ g \circ f^{-1}-{\sf{id}}|&=|f\circ g \circ f^{-1}-f\circ f^{-1}|\\
		&\le (\sup_U \|Df\|)|g(f^{-1})-f^{-1}|\le(\sup_U \|Df\|)\delta t.
	\end{align*}
	Moreover 
	\begin{align*}
		\|D(f\circ g \circ f^{-1})-I_n\|&\le\|Df(g(f^{-1}))Dg(f)Df^{-1}-I_n\|\\
		&\le \|Df(g(f^{-1}))-Df(f^{-1})\|\|Dg\|\|Df^{-1}\|+\|Df(f^{-1})DgDf^{-1}-I_n\|.
	\end{align*}
	To bound the first term we need to exploit the bound on the second derivatives of $f$ as follows. From the mean value theorem and recalling \eqref{opnorm},
	\begin{align*} 
		\|Df(g(f^{-1})-Df(f^{-1})\| &\le \sqrt n\sup_{i,k} |\partial_i f_k(g(f^{-1}(x)))-\partial_i f_k(f^{-1})|\\
		&\le  \sqrt n  |g(f^{-1})-f^{-1}| \sup_{i,j,k}|\partial_{ij}f_k|\\
		&\le  \sqrt n \delta t\sup_{i,j,k}|\partial_{ij}f_k|.
	\end{align*}
	To bound the second term, recall that $Df(f^{-1})Df^{-1}=I_n$ and argue in the following way
	\begin{align*}
		\|Df(f^{-1})DgDf^{-1}-I_n\| = \|Df(f^{-1})(Dg-I_n)Df^{-1}\|\le \|Df\|\|Df^{-1}\|\delta.
	\end{align*}
\end{proof}

\subsection{Gromov-Hausdorff distance}
Here we introduce  the Gromov-Hausdorff distance and  Gromov-Hausdorff approximation maps, for  more details on this topic we refer to \cite{burago} and \cite{gromov}.

\begin{definition}[$\delta$-neighbourhood and $\delta$-dense set]
	Let $(\X,\sfd)$ be a metric space. For every $\delta>0$ and every set $E\subset \X$ we define the \textit{$\delta$-neighbourhood of $E$} as
 \[
 (E)^\delta\coloneqq \{x \in \X \ : \ \sfd(x,E)<\delta\}.
 \]
We will then say that a set $S\subset \X$ is $\delta$-dense in $\X$ if $(S)^\delta=\X.$
\end{definition}

\begin{definition}[$\delta$-isometry]
	Let $(\X_1,\sfd_1),(\X_2,\sfd_2)$ be two metric spaces and fix a number $\delta>0$. We say that a function $f : S \to \X_2$  for some $S \subset \X_1$  is a $\delta$-isometry if
	\[|\sfd_2(f(x),f(y))-\sfd_1(x,y)|< \delta, \quad \forall x,y \in \,S.\]
\end{definition}

\begin{definition}[Hausdorff distance]\label{dh}
	Let $(\X,\sfd)$ be a metric space and let $A,B\subset \X$. We define the Hausdorff distance between $A$ and $B$ as the number
	\[d^{\X}_H(A,B)\coloneqq\inf \{r \ | \  B\subset (A)_r \text{ and } A\subset (B)_r \}.\]
\end{definition}

\begin{definition}[Gromov-Hausdorff distance] \label{GHdef}
	Let $(\X_1,\sfd_1),(\X_2,\sfd_2)$ be two metric spaces, we define the Gromov-Hausdorff distance between $\X_1$ and $\X_2$ as the number
	\[d_{GH}(\X_1,\X_2)\coloneqq\inf_{\substack{{(Z,\sfd)}\\{i_1 : \X_1 \to (Z,\sfd)}\\{i_2 : \X_2 \to (Z,\sfd )}}}  d^{\Z}_H(i_1(\X_1),i_2(\X_2)),\]
	where the infimum is taken among all the triples $(\Z,\sfd),i_1,i_2$ such that  $i_1,i_2$ are isometric embeddings of $\X_1,\X_2$ into  $\Z$. 
\end{definition}
With a slight abuse of notation we will also write 
\[
\sfd_{GH}(A_1,A_2)\coloneqq \sfd_{GH}\big((A_1,\sfd_1\restr {A_1}),(A_2,\sfd_2\restr {A_2})\big), \quad \text{for all subsets $A_1\subset \X_1$, $A_2\subset \X_2$}.
\]

It will be sometimes more convenient to work with an equivalent characterization of the Gromov-Hausdorff distance through Gromov-Hausdorff approximation maps defined as follows.
\begin{definition}[GH-approximations]\label{ghappr}
	Let $(\X_1,\sfd_1),(\X_2,\sfd_2)$ be two metric spaces and fix a number $\delta>0$. We say that a function $f:\X_1 \to \X_2$ is a \emph{$\delta$-Gromov-Hausdorff approximation map} ($\delta$-GH-app.\ for short) if
	\begin{itemize}
		\item $f$ is a $\delta$-isometry,
		\item $f(X)$ is $\delta$-dense in $\X_2.$
	\end{itemize} 
\end{definition}
The Gromov-Hausdorff approximation maps and Gromov-Hausdorff distance are related as follows. 
\begin{theorem}[\hspace{1sp}{\cite[Corollary 7.3.28]{burago}}]\label{ghchar}
	Let $(\X_1,\sfd_1),(\X_2,\sfd_2)$ be two metric spaces. Then
	\begin{enumerate}
		\item if $d_{GH}(\X_1,\X_2)\le \delta$, then there exists a $2\delta$-GH-app.\ $f: \X_1 \to \X_2,$
		\item if there exists a $\delta$-GH-app.\ $f: \X_1 \to \X_2,$ then $d_{GH}(\X_1,\X_2)\le 2\delta$.
	\end{enumerate}
\end{theorem}
Observe that in \cite{burago} $\delta$-isometry maps are what we here call $\delta$-GH-app.\ maps.  

The following result shows that GH-approximations between two metric spaces can be chosen to be almost the inverse of each other.
\begin{prop}\label{apprinverse}
	Let $(\X_1,\sfd_1),(\X_2,\sfd_2)$ be two metric spaces and suppose $f: \X_1 \to \X_2$ is a $\delta$-GH-app.\ then there exists a $3\delta$-GH-app.\ $g: \X_2 \to \X_1$ such that 
	\begin{equation}\label{inv1}
		\sfd_2(f(g(y)),y)<2\delta ,
	\end{equation}
	\begin{equation}\label{inv2}
		\sfd_1(g(f(x)),x)< 2\delta,
	\end{equation}
	for every $x \in \X_1$ and $y \in \X_2.$
\end{prop}
\begin{proof}
	We can construct $g$ explicitly as follows. For every $y \in \X_2$ by definition there exists at least one $x\in \X_1$ satisfying $\sfd_2(f(x),y)<\delta,$ we define $g(y)$ to be one of such points (chosen arbitrarily). Then \eqref{inv1} holds by construction.  We now prove \eqref{inv2}. Pick any $x \in \X_1$, then by definition of $g$ we have that $\sfd_2(f(g(f(x))),f(x))<\delta $ and since $f$ is a $\delta$-isometry we deduce that $\sfd_1(g(f(x)),x)<2\delta.$ Notice now that \eqref{inv2} already proves that $g(\X_2)$ is $2\delta$-dense in $\X_1$. It remains to prove that $g$ is a $3\delta$-isometry. Pick $y_1,y_2 \in \X_2.$ From \eqref{inv1} we have $\sfd_2(f(g(y_i)),y_i)<\delta,$ therefore by the triangle inequality we have
	\[ |\sfd_2(f(g(y_1)),f(g(y_2)))-\sfd_2(y_1,y_2)|<2 \delta.\]
	Recalling that $f$ is a $\delta$-isometry we obtain that
	\[|\sfd_1(g(y_1),g(y_2))-\sfd_2(y_1,y_2)|<3\delta.\]
\end{proof}
\begin{remark}
    Strictly speaking in the proof of the above proposition we used the axiom of choice in the construction of $g$. However it would be sufficient to use the axiom of countable choice if for instance the metric spaces $X_1$ is separable. Indeed we can define $g(y)$ by selecting a point from $f(S)$, with $S$ a dense set in $X_1.$
    \fr
\end{remark}

Given a metric space $(\X,\sfd)$ and a point $p\in \X$ we denote by $B^{\X}_r(p)\coloneqq\{x \in \X \ : \ \sfd(x,p)<r\}$ the ball of radius $r>0$ centered at $p$ and by $\overline B_r(p)$ its topological closure, which in case of a length space coincides with the closed ball $\{x \in \X \ : \ \sfd(x,p)\le r\}.$ 

In general a GH-approximation between two balls in different metric spaces needs not to send the center near the center, e.g.\ there are balls in metric spaces admitting two different centers. However this does hold if one the spaces is the Euclidean space.
\begin{prop}\label{centerappr}
	Let $(\X,\sfd)$ be a metric space and let $B_1(x_0)$ the ball of radius $1$ centered at $x_0 \in \X.$ Suppose $f: B_1(x)\to B^{\rr^k}_1(0)$ is a $\delta$-GH-app., then 
	\[|f(x_0)|\le 7\delta .\]
\end{prop}
\begin{proof}
	Suppose by contradiction that $|f(x_0)|>7\delta$, then there exists a point $y \in B_1(0)$ such that $|y-f(x_0)|>1+6\delta$. Let $g$ be the almost-inverse map given by Proposition \ref{apprinverse}. Then from \eqref{inv2} we have $\sfd(g(f(x_0)),x_0)<2\delta$. Moreover, since $g$ is a $3\delta$-isometry, we get
	\[\sfd(g(y),g(f(x_0)))>1+3\delta. \]
	Hence by the triangle inequality 
	\[\sfd(g(y),x_0)\ge \sfd(g(y)),g(f(x_0))-\sfd(g(f(x_0)),x_0) \ge 1+\delta,\]
	but $g(y)\in B_1(x_0)$ and thus we have a contradiction.
\end{proof}

The following  well known approximation result gives a quantified approximation of an almost-isometry with an isometry and will  be the starting point for the construction in the proof of the bi-Lipschitz metric Reifenberg's theorem.  It can be easily proved via computation in coordinates
(see for example Lemma 7.11 in \cite{snow}).  A non-quantified version of this statement could be more directly derived by compactness, however  for our goal of proving the bi-Lipschitz Reifenberg's theorem, this explicit control will be essential. 
\begin{lemma}\label{AA}
	Suppose that a function $f: {B_{t}(0)} \to  \rr^n$, with $B_t(0)\subset \rr^n$, is a $\delta t$-isometry
	for some $\delta<1.$ Then there exists an  isometry $I: \rr^n \to \rr^n$ with $I(0)=f(0)$, such that 
	\begin{equation}\label{nn}
		|I-f|\le C(n)\delta t, \quad \text{in $B_t(0)$},
	\end{equation}
 where $C(n)>0$ is a constant depending only on $n.$
\end{lemma}

\section{Proof of the metric  Reifenberg's theorem}\label{sec:constructions}
Throughout this section $(\Z,\sfd)$ is a metric space satisfying the assumptions of Theorem \ref{mainthm} with $n \in \mathbb{N}$. In the following subsection we give a short proof of Theorem \ref{mainthm}, \emph{assuming} the existence of suitable objects, that is a sequence of approximating manifolds endowed with suitable maps. The rest of the subsections will be devoted to the construction of such objects.

\subsection{Main argument}\label{sec:main argument}
 For brevity we will  write $\eps_i$ in place of $\eps_i(r,n),$ where  $\eps_i(r,n)$ are the numbers appearing in the statement of Theorem \ref{mainthm}. Up to rescaling we can also assume that $r=200$ (the assumptions will be used in later sections).
 
The rough idea of the proof is to construct for every scale $2^{-i}$  a manifold $(W_i,\sfd_i)$ that approximates the metric space $Z$  in the sense that there exists a map $f_i : W_i \to Z$  that is roughly an $\eps_i2^{-i}$-isometry. Moreover we build maps $h_i : W_i \to W_{i +1}$ that are bi Lipschitz  with  constant $\sim (1+\eps_i)$.  Then the required map $F$ will be obtained as limit of the maps $f_i\circ h_i\circ h_{i-1}\circ \cdots \circ h_1$.

{ {\bf CLAIM}}:	For the proof it is sufficient to construct  a sequence of (complete and connected) $n$-dimensional Riemannian manifolds $(W_i,\sfd_i)$, symmetric maps $\rho_i : W_i \times W_i \to [0,\infty)$ vanishing on the diagonal, surjective diffeomorphisms $h_i : W_i \to W_{i+1}$ and maps $f_i: W_i \to Z$, for $i \in \mathbb{N}$, satisfying the following statements.
\begin{enumerate}[label=\alph*)]
	\item \label{a)} For every $w_1,w_2 \in W_i$  it holds that
	$$\rho_i(w_1,w_2)=\sfd_i(w_1,w_2),$$
	whenever either $\rho_i(w_1,w_2)\le 2^{-i}$ or $\sfd_i(w_1,w_2)\le 2^{-i}$.
	\item \label{b)}$$\frac{1}{1+C(\eps_i+\eps_{i+1})}\le\frac{\rho_{i+1}(h_i(w_1),h_i(w_2))}{\rho_i(w_1,w_2)}\le 1+C(\eps_i+\eps_{i+1}),$$
	for every $w_1,w_2 \in W_i,$
	\item \label{c)}$|\sfd(f_i(w_1),f_i(w_2))-\rho_i(w_1,w_2)|\le C2^{-i}\eps_i$, for every $w_1,w_2 \in W_i,$
	\item \label{d)}$\sfd(f_{i+1}(h_i(w)),f_i(w))\le C 2^{-i}(\eps_i+\eps_{i+1})$, for every $w \in W_i,$
	\item \label{e)}$f_i(W_i)$ is $ 20 \cdot 2^{-i}$-dense in $Z$.
\end{enumerate}

We now briefly comment the statement of the claim. The function $\rho_i$ substitutes the Riemannian distance, coincides with $\sfd_i$ at small scales and it is built to approximate the distance $\sfd$ (see item \ref{c)}). It is needed because we do not have a good control on the relation between $\sfd_i$ and $\sfd$ at large scales. The  map $f_i$ is a kind of GH-approximation between $W_i$ and $Z$. Indeed it is almost surjective (see item \ref{e)}) and it is a quasi isometry (see item \ref{c)}), however not with respect to the Riemannian distance, but with respect to $\rho_i$. Item \ref{b)} is the bi-Lipschitz property of the maps $h_i$. Finally, item \ref{d)} ensures that $h_i$ does not move points too much, from the prospective of the metric space.\\

{\textbf {Proof of the CLAIM}}:
Pick any $\eps>0$. Start by fixing an integer $m\ge 0$ large enough, that will be chosen later. In the case $i)$ of the Theorem we actually choose $m=0.$ Set $\theta_m\coloneqq\sup_{i\ge m} \eps_i$. In particular in case $i)$ we have $\theta_m\le \eps(n)$ by assumption.
For every $i > m$ define the function $F_i:W^n_{m}\to Z$ as $F_i\coloneqq f_i \circ h_{i-1}\circ...\circ h_{m+1}\circ  h_{m}$  and set $F_m\coloneqq f_m.$ From \ref{d)} we get that for every $w \in W_m^n$
\begin{equation}\label{Fbound}
	\sfd(F_{i+1}(w),F_i(w))\le (\eps_i+\eps_{i+1})C2^{-i}\le 2C\theta_m2^{-i},
\end{equation}
for every $i \ge m.$
Hence the sequence $\{F_i(w)\}_{i\ge m}$ is Cauchy in $Z$ and we call $F(w)$ its limit. We will prove that the manifold $W_m^n$	and the map $F: W_m^n \to Z$ satisfy the conclusions of Theorem \ref{mainthm} with $\eps$. Notice first that from \eqref{Fbound} we obtain 
\begin{equation}\label{fif}
	\sfd(F(w),F_i(w))\le  2C2^{-i}\theta_m,
\end{equation} 
for every $i \ge m.$
Consider now any couple of distinct points $w_1,w_2 \in W_m$  and set
\[s_m\coloneqq\rho_m(w_1,w_2)>0.\]
Define also inductively for every $i\ge m$ the points $w_1^i,w_2^i \in W_i$ and the numbers $s_i$ by setting $w_1^m\coloneqq w_1, w_2^m\coloneqq w_2$ and then
\[w_1^{i+1}\coloneqq h_{i}(w_1^i),\quad w_2^{i+1}\coloneqq h_{i}(w_2^i), \quad  s_i\coloneqq\rho_i(w_1^i,w_2^i).\]
Observe that from \ref{b)} 
\begin{equation}\label{eq:silip}
    (1+C(\eps_i+\eps_{i+1}))^{-1}s_i\le s_{i+1}\le s_{i}(1+C(\eps_i+\eps_{i+1})),
\end{equation}
for every $i \ge m$. Moreover from \ref{c)} we deduce that 
\begin{equation}\label{rofi}
	|\sfd(F_i(w_1),F_i(w_2))-s_i|\le C \eps_i2^{-i},
\end{equation}
for every $i\ge  m$. 

We now divide the proof depending on the assumptions $i)$ (i.e.\ $\eps_i\le \eps(n)$) or $ii)$ (i.e.\ $\sum_i \eps_i<+\infty$) of Theorem \ref{mainthm}

\noindent{\sc $F$ is locally BiH\"older (under assumption $i)$):}
Combining \ref{c)} and \ref{d)} we get
\begin{equation}\label{eq:si si+1}
   s_i- 6C\eps(n)2^{-i}\le  s_{i+1}\le s_i+ 6C\eps(n)2^{-i}.
\end{equation}
In particular the limit $s_\infty\coloneqq \lim_{m\to +\infty} s_m$ exists and by \eqref{rofi} 
    $s_\infty=\sfd(F(w_1),F(w_2)).$
Let $j\in \nn$ be arbitrary. Applying repeatedly \eqref{eq:silip} for all $s_0,\dots,s_j$ and then \eqref{eq:si si+1} for the remaining $s_i$'s with $i>j$, we obtain
\begin{equation}\label{eq:holder trick}
   2^{-\alpha j} s_0 - 6C \eps(n)2^{-j}\le  s_\infty\le 2^{\alpha j} s_0 + 6C \eps(n)2^{-j},
\end{equation}
where we set $\alpha\coloneqq \log_2(1+2C \eps(n))>0.$ Note that $\alpha\in(0,1)$ provided $\eps(n)<C/2.$ Assume that $s_0\le 1$ (which by  \ref{a)} implies $s_0=\sfd_0(w_1,w_2)$).
Then there exists a (maximal) $j\in\nn$ so that 
\begin{equation} \label{eq:first j}
    2^{\alpha j}s_0\le 2^{-j}, \quad 2^{\alpha(j+1)}s_0\ge 2^{-(j+1)}.
\end{equation}
From the second in \eqref{eq:first j} we get
\begin{equation}\label{eq:s0 potenziato}
    2s_0^\frac{1}{1+\alpha}\ge 2^{-j}.
\end{equation}
Plugging in the right hand side of \eqref{eq:holder trick} first \eqref{eq:first j} and then \eqref{eq:s0 potenziato} we obtain
\begin{equation}\label{eq:holder1}
     s_\infty\le 2^{\alpha j} s_0 + 6C \eps(n)2^{-j}\le (1+6C \eps(n))2^{-j} \le 2(1+6C \eps(n))s_0^\frac{1}{1+\alpha}
\end{equation}
Similarly, since $\alpha<1,$ we can find $j\in \nn$ such that
\begin{equation}\label{eq:laltro}
    2^{-\alpha j}s_0\ge 2^{-j} \quad \text{ and } \quad 2^{-\alpha (j-1)}s_0\le 2^{-(j-1)}.
\end{equation}
From the second in \eqref{eq:laltro}  we get $s_0^\frac{1}{1-\alpha}\le 2^{-j+1}$.
Plugging this and the first in \eqref{eq:laltro} into \eqref{eq:holder trick} we get
\begin{equation}\label{eq:holder2}
     s_\infty\ge  2^{-j} - 6C \eps(n)2^{-j}\ge (1-6C \eps(n)))\frac{s_0^\frac{1}{1-\alpha}}2,
\end{equation}
provided $\eps(n)<C/6.$
Combining \eqref{eq:holder1} and \eqref{eq:holder2}  we get (provided $\eps(n)<C/12$)
\begin{equation}\label{eq:biholder}
    \frac14 \sfd_0(w_1,w_2)^\frac{1}{1-\alpha}\le \sfd(F(w_1),F(w_2))\le 4\sfd_0(w_1,w_2)^{1-\alpha}, \quad \text{whenever $\sfd_0(w_1,w_2)\le 1$}.
\end{equation}
In particular $F$ is continuous.
Note that by \eqref{fif} and \eqref{rofi} (which were deduced without assuming $s_0\le 1$) we have 
$$|\sfd(F(w_1),F(w_2))-s_0|\le C \eps(n).$$
Hence if $\sfd(F(w_1),F(w_2))<\frac12$, then $s_0<1$, $s_0=\sfd_0(w_1,w_2)$ and \eqref{eq:biholder} holds. In particular $F$ is injective. To conclude it remains to prove that $F$ is surjective (this is checked below). Indeed if this was true the required locally biH\"older condition is given by \eqref{eq:biholder} and the observation we just made (note that for $r=200$, $r/800=\frac14$).

\noindent{\sc $F$ is locally BiLipschitz (under assumption $ii)$): }
 Iterating inequality \eqref{eq:silip} we get
\begin{equation}	\label{lip}
	s_m\prod_{j=m}^i(1+C(\eps_j+\eps_{j+1}))^{-1}\le s_{i}\le s_{m} \prod_{j=m}^i(1+C(\eps_j+\eps_{j+1})),
\end{equation} 
for every $i>m.$ From \eqref{rofi} we infer that $s_i \to \sfd(F(w_1),F(w_2))$ as $i \to +\infty.$ Then passing to the limit in \eqref{lip} we obtain
\[\prod_{j=m}^{+\infty }\frac{1}{1+C(\eps_j+\eps_{j+1})}\le \frac{\sfd(F(w_1),F(w_2))}{s_m}\le \prod_{j=m}^{+\infty }(1+C(\eps_j+\eps_{j+1})).\]
Observe now that, since by hypothesis $\sum_{j\ge 0}\eps_j <+\infty $, we have that $\prod_{j=0}^{+\infty }(1+C(\eps_j+\eps_{j+1}))\in(0,\infty)$. In particular $\prod_{j=k}^{+\infty }(1+C(\eps_j+\eps_{j+1})) \downarrow 1$ and $\prod_{j=k}^{+\infty }(1+C(\eps_j+\eps_{j+1}))^{-1} \uparrow 1$ as $k$ goes to $+\infty.$ Therefore, up to choosing $m$ big enough, we have
\[\frac{1}{1+\eps}\le \frac{\sfd(F(w_1),F(w_2))}{s_m}\le 1+\eps.\]
Recall now that by \ref{a)} we have that $\sfd_m(w_1,w_2)\le 2^{-m}$ implies $s_m=\sfd_m(w_1,w_2),$   hence 
\begin{equation}	\label{lipfinal}
	\frac{1}{1+\eps}\le \frac{\sfd(F(w_1),F(w_2))}{\sfd_m(w_1,w_2)}\le 1+\eps, \quad \text{whenever $\sfd_m(w_1,w_2)\le 2^{-m}.$}
\end{equation}
This already proves that $F$ is continuous and the first part of \eqref{finalbilip}, provided we take $\rho<2^{-m-1}.$ Combining now \eqref{rofi} with \eqref{fif} for $i=m$ and observing that $\theta_m \to 0^+$ as $m$ goes to $+\infty$, we obtain that
\[ |\sfd(F(w_1),F(w_2))-s_m|\le 6C\theta_m2^{-m}<\frac{1}{4}2^{-m},\]
provided $m$ is big enough. Therefore, whenever $\sfd(F(w_1),F(w_2))\le 2^{-m-1}$, it holds that $s_m\le 2^{-m}$ and thus from \ref{a)} also that $d_m(w_1,w_2)\le 2^{-m}.$ Combining this observation with \eqref{lipfinal} we obtain that
\begin{equation}\label{eq:prebilip}
\frac{1}{1+\eps}\le \frac{\sfd(F(w_1),F(w_2))}{\sfd_m(w_1,w_2)}\le 1+\eps, \quad \text{whenever $\sfd(F(w_1),F(w_2))\le 2^{-m-1}.$}
\end{equation}
 This proves that $F$ is injective. To conclude it remains to prove that $F$ is surjective (indeed the second part of \eqref{finalbilip} would follow from \eqref{eq:prebilip} taking $\rho<2^{-m-2}$). 

\noindent{\sc F is surjective:} The proof is the same for both cases $i)$ and $ii)$ of the theorem.	We start proving that $F(W_m^n)$ is dense. To see this consider $z \in Z$ and $\delta>0.$ 
Take now $i\ge m$ such that $20\cdot 2^{-i}\le \delta/8$, then by \ref{e)} and the fact that the maps $h_k$ are surjective, there exists $w \in W_m^n$ such that $\sfd(F_i(w),z)< \delta/2$. Moreover from \eqref{fif} it holds $\sfd(F(w),F_i(w))<  4C2^{-i}\theta_m\le  4C\delta \theta_m\le  \delta/2$, provided $m$ is big enough in case $ii)$ (or in case $i)$, since $\theta_m\le \eps(n),$ provided $\eps(n)$ is small enough). Hence $\sfd(F(w),z)<  \delta$ and thus $F(W_m^n)$ is dense from the arbitrariness of $\delta>0$. Pick now any $z \in Z,$ by density, there exists a sequence $w_k \in W_m$ such that $F(w_k)\to z$. In particular $(F(w_k))$ forms a Cauchy sequence in $(Z,\sfd)$. From \eqref{eq:prebilip} we deduce that also $w_k$ is Cauchy  and by the completeness of $W_m$ it converges to a limit $w\in W_m$. From the continuity of $F$ we have that $F(w)=z$. This proves that $F$ is surjective.  

\medskip

The last part of the theorem follows directly form the fact that the maps $h_i: W_i \to W_{i+1}$ are diffeomorphisms.
This concludes the proof of the {\bf CLAIM}. \\

\subsection{Construction of the approximating manifolds \texorpdfstring{$W_i$}{Wi}}\label{sec:manifolds}

\subsubsection{Notation and choice of constants}\label{sec:notations reif}
We start by fixing once and for all a positive constant $C=C(n)$, that will appear both on the present and on the following sections. Its value will be determined along the proof and may change from line to line, but will in any case remain dependent only on $n$.
Moreover we  also pick  $\eps(n)=\eps(n,C)$ another positive constant that may change along the proof, but will depend only on $C$ and $n$. In particular $\eps(n)$ at the end will depend only on $n$. Since $r$ and $n$ are fixed along all the proof we will also write $\eps_i$ in place of $\eps_i(r,n),$ for every $i \in \mathbb{N}.$ 

We observe that Theorem \ref{mainthm} holds for a metric space $(Z,\sfd)$ if and only if it holds for all the rescaled spaces $(Z,\lambda\sfd),$ with $\lambda>0.$ Moreover, denoted by $\eps_i^{\lambda}(r,n)$ the numbers in Definition \ref{ghjonest} relative to the rescaled space $(Z,\lambda \sfd)$, it easy to verify that 
$$\eps_i^\lambda(\lambda r 2^{-j},n)=\eps_{i+j}(r,n),$$
for every $i,j \in \mathbb{N}$ and every $r>0$.
Therefore, since by \eqref{sumscales} $\eps_i \to 0^+$ as $i \to +\infty$, up to rescaling the metric $\sfd$ we can assume both that
\begin{equation}\label{ass1}
	r=200
\end{equation}
and that
\begin{equation}\label{ass2}
	\eps_i\le \eps(n) \text{ for every } i \ge 0. 
\end{equation}
This assumption will be fixed throughout the rest of Section \ref{sec:constructions}. In particular whenever in the sequel we will say  that  $\eps_i$ is small enough, we will mean that the constant $\eps(n)$ is chosen sufficiently small.

\subsubsection{Construction of the coverings}\label{sec:coverings}
We start by showing that under our assumptions $(Z,\sfd)$ is separable and locally compact. 
To prove this we exploit the following result due to Alexandroff.
\begin{theorem}[\hspace{1sp}\cite{alexandrov}]\label{thm:alex}
	Let $(Z,\sfd)$ be a connected metric space that is locally separable, i.e.\ for every $z \in Z	$ there exists a separable ball  that contains $z$.  Then $Z$ is separable.
\end{theorem}

Therefore to prove that $(Z,\sfd)$ is separable it is enough  to prove the following Lemma.
\begin{lemma}\label{totballs}
	$B_{r}(z_0)$ is totally bounded for every $z_0 \in Z$, where $r$ is the one in the statement of Theorem \ref{mainthm} (that is $r=200$  under the current assumptions).
\end{lemma}
\begin{proof}
	We start with the following claim. For every $z\in Z$ and every $t=r2^{-k}$ with $k\in \mathbb{N}$, there exists a finite set $S_{z,t} \subset B_{t}(z)$ that is $t/2$-dense in $B_{t}(z).$ To prove this first observe there exists a finite subset $E_t$ of $B^{\rr^n}_{t}(0)$ that is $t/8$ dense in $B^{\rr^n}_{r}(0)$. Moreover by hypothesis 	
	$\sfd_{GH}(B^Z_{t}(z),B^{\rr^n}_{t}(0))<\eps(n)t\le t/16,$ therefore there exists a $t/8$-GH approximation $f: B^{\rr^n}_{t}(0) \to B_t(z).$ Define $S_{t,z}\coloneqq f(E_t).$ By definition of GH-approximation, for  every $p \in B_t(z),$ there exists $x \in B^{\rr^n}_{t}(0) $ such that $\sfd(f(x),p)\le t/8$. Moreover there exists $e \in E_t$ with $|e-x|\le t/8,$ hence $\sfd(f(e),p) \le 3/8t<t/2$ and the claim is proved. We start defining sets $S_k$ with $k\in \mathbb{N}$ inductively as follows. Set $S_1\coloneqq S_{z_0,r/2}$ and $$S_{k+1}\coloneqq\bigcup_{z \in S_k} S_{z,\frac{r}{2^{k+1}}}.$$
	Then define 
	$$A_k\coloneqq \bigcup_{z\in S_k} B_{\frac{r}{2^k}}(z).$$
	Observe that $A_k \subset A_{k+1},$ indeed 
	$$A_k=\bigcup_{z\in S_k} B_{\frac{r}{2^k}}(z)\subset \bigcup_{z\in S_k} \bigcup_{\bar z \in S_{z,\frac{r}{2^{k+1}}}} B_{\frac{r}{2^{k+1}}}(\bar z)=\bigcup_{z\in S_{k+1}} B_{\frac{r}{2^{k+1}}}(z)=A_{k+1}.$$
	Notice also that that $B_{r}(z_0)\subset \bigcup_{z\in S_1} B_{\frac{r}{2}}(z)=A_1 \subset A_k$ for every $k$. Moreover $A_k$ is union of a finite number of balls of radius $r/2^{k}$, therefore $B_{r}(z_0)$ is totally bounded.
\end{proof}
\begin{remark}
	Instead of Theorem \ref{thm:alex} we could have used the fact that a metric space is paracompact and then the fact that every connected, locally compact and paracompact tolopogical space is separable (see e.g. \cite[Appendix A]{spivak}).\fr
\end{remark}

\begin{tcolorbox}[colframe=white,colback=mygray]
  \textbf{Important:} From now on we assumed to have fixed for all $i\in \nn$ a set $X_i$ that is $2^{-i}$-dense in $Z$ and we label the elements of $X_i$ as $X_i=\{x_{i,1},x_{i,2},....\}=\{x_{i,j}\}_{j\in J_i}$ where $J_i\coloneqq \{1,2,...,\#|X_i|\}$.
We also choose a partition of $X_i$ into disjoint subsets $Q^i_1,Q^i_2,...,Q^i_{N_i},$ satisfying 
\[
\sfd(x,y)\ge100 \cdot 2^{-i}, \quad \text{for all $x,y \in Q^i_k$ and all $k=1,\dots,N_i$}.
\]
Moreover for all $i\in \nn$ we can take $N_i\le N(n)$, where $N(n)$ is an integer depending only on $n$.
 Finally we partition  the set of indices $J_i$ as $J_i=\bigcup_{k=1}^{N_i}J_i^k$ where $J_i^k\coloneqq  \{ j\in J_i \ | x_{i,j}\in Q_k^i\}$.
\end{tcolorbox}

The validity of the above construction is justified by the following result.
\begin{prop}\label{prop:covering}
     For all $i\in\nn$ and all $\theta \in [1/2,1]$ we can find countable set $X_i \subset Z$ satisfying the following properties:
     \begin{enumerate}[label=\roman*)]
         \item $X_i$ is $\theta 2^{-i}$-dense in $Z,$
         \item for all constants $\lambda\le 200$ the set $X_i$ can be partitioned into disjoint subsets $Q^i_1,Q^i_2,...,Q^i_{N_i},$ with $N(i)\le N(n)$ (an integer depending only on $n$) with the property that $\sfd(x,y)\ge\lambda \cdot 2^{-i}$ for all $x,y \in Q^i_k$ and for every $k=1,\dots,N(i).$
     \end{enumerate}
\end{prop}
\begin{proof}
Consider  for every $i\in \mathbb{N}$ a set $X_i \subset Z$, such that for every $x_1,x_2 \in X_i$ it holds that $\sfd(x_1,x_2)\ge \theta 2^{-i}$ and $X_i$ is maximal with this property with respect to inclusion. In particular $X_i$ is also $\theta 2^{-i}$-dense. Moreover, since the balls of radius $2^{-i-1}$ centered in $X_i$  are  pairwise disjoint and $Z$ is separable, $X_i$ is countable. 
 The required partition $X_i$ into disjoint subsets can be obtained as follows.	Let $Q_1^i \subset X_i$, maximal (with respect to inclusion), with the property that $\sfd(x,y)\ge \lambda\cdot 2^{-i} $ for every $x,y \in Q^i_1.$ Consider also $Q^i_2\subset X_i\setminus Q^i_1$ maximal such that $\sfd(x,y)\ge\lambda \cdot 2^{-i}$ for every $x,y \in Q^i_2.$ Keep defining  inductively (non-empty) sets $Q^i_1,Q^i_2,...,Q^i_{N_i}.$ It remains to show that cardinality of these partitions is uniformly bounded in $i$.
 Suppose that $Q^i_{\bar k}$ is non-empty for some $\bar k>1$ and fix any $x \in Q^i_{\bar k}$. In particular $x \in X_i\setminus (Q_1^i\cup...\cup Q^i_{\bar k-1})$.
	By maximality we have $B_{\lambda \cdot 2^{-i}}(x) \cap Q^i_k \neq \emptyset$ for every $1\le k \le \bar k.$  Set now $S\coloneqq  X_i \cap B_{\lambda \cdot 2^{-i}}(x)$. Then $S\cap  Q^i_k \neq \emptyset$ for every $1\le k \le \bar k$ and in particular $\bar k\le \#|S|.$ 
	
	We now aim to give an upper bound on $\#|S|$.
	Recall that by  hypothesis  there exists a $C\eps_i2^{-i} $-isometry $f : B_{200\cdot 2^{-i}}(x) \to B_{200\cdot 2^{-i}}(0)$. Recall also that by construction the points in $S$ are at distance at least $2^{-i}/2$  from each other. Therefore, assuming  $\eps_i$ small enough, we deduce both $f|_S$ is injective and that the points in $f(S) \subset B_{200\cdot 2^{-i}}(0) $ are at distance at least $2^{-i}/4$  from each other. Thus $\#|S|=\#|f(S)|\le N(n)$, for some integer $N(n)$ depending only on $n$, that in turn implies $\bar k\le N(n).$
\end{proof}

\subsubsection{Construction of the irregular transition maps}\label{sec:irr charts}
From  hypothesis, combining Theorem \ref{ghchar}, Proposition \ref{apprinverse} and Proposition \ref{centerappr} (recalling that we are assuming $r=200$) we have that for every $j \in J_i$ there exist maps $$\alpha_{i,j} : B^{\rr^n}_{200\cdot 2^{-i}}(0) \to B^{Z}_{200\cdot 2^{-i}}(x_{i,j})$$ and $$\beta_{i,j}: B^{Z}_{200\cdot 2^{-i}}(x_{i,j}) \to B^{\rr^n}_{200\cdot 2^{-i}}(0)$$ that are $C \eps_i 2^{-i}$-Gromov-Hausdorff approximations, such that 
\begin{equation}\label{alfabeta}
	d_{\rr^n}(\beta_{i,j}\circ \alpha_{i,j},\id), \, d_Z(\alpha_{i,j}\circ \beta_{i,j},\id)\le C\eps_i 2^{-i}, \quad \text{uniformly}
\end{equation}
and
\begin{eqnarray}\label{centertocenter1}
	\alpha_{i,j}(0)\in B_{C \eps_i 2^{-i}}(x_{i,j}),\,\,\, \beta_{i,j}(x_{i,j})\in B_{C \eps_i 2^{-i}}(0).
\end{eqnarray}
For simplicity and to avoid very heavy notations we will often drop the index $i$ on the symbols $\alpha_{i,j},\beta_{i,j}$ and $x_{i,j}.$ Since for the most part the index $i$ is fixed, this convention should cause no confusion. \\
An immediate consequence of \eqref{centertocenter1} is that, provided $\eps(n) < \frac{1}{200C}$,
\begin{equation}\label{centertocenter}
	\alpha_{j}(B_{s}(0))\subset B_{s+\frac{2^{-i}}{100}}(x_{j}), \,\,\, \beta_{j}(B_{s}(x_{j}))\subset B_{s+\frac{2^{-i}}{100}}(x_{j}),
\end{equation}
for every $s<200 \cdot 2^{-i}.$\\

The maps that we just defined allow us to think of $Z$ as a very irregular manifold. In particular  $\beta_j$ have the role of charts for the metric space $Z$ and  $\beta_{j_2} \circ \alpha_{j_1}$ (when defined) can be viewed as transition maps. However all these functions are very rough and could be not even continuous. Therefore the main idea to build the manifold $W_i$ is to approximate the maps $\beta_{j_2} \circ \alpha_{j_1}$ with diffeomorphisms $\tilde I_{j_2j_1}$ and then build a manifold having   $\tilde I_{j_2j_1}$  as transition maps. The main issue in this procedure is that to be able to actually build a manifold we need these diffeomorphisms to be compatible with respect to each other. This will require the modification procedure that we developed in Section \ref{sec:modification}. 

\begin{remark}\label{rmk:charts}
    If $Z$ is a smooth compact Riemannian manifold then, for $i$ big enough (depending on $Z$), the maps $\beta_{i,j}, \alpha_{i,j}$ satisfying \eqref{alfabeta} and \eqref{centertocenter1}  can in fact be taken to be (restriction of)  suitable charts and their inverses respectively. 
\end{remark}

\subsubsection{Approximation of the irregular transition maps with isometries}
We start approximating  $\beta_{j_2} \circ \alpha_{j_1}$ with isometries of $\rr^n$ as follows. Note that in our notation isometries do not need to fix the origin.
\begin{lemma}\label{Definition of the transition maps}
	 Suppose  that $j_1,j_2 \in J_i$ are such that $\sfd(x_{j_1},x_{j_2})<30\cdot 2^{-i}$. Then there exist two isometries 
	$I_{i,j_2j_1},I_{i,j_1j_2}: \rr^n \to \rr^n$ such that $I_{i,j_2j_1}=I_{i,j_1j_2}^{-1}$ and
	\begin{equation}\label{alminverse}
		|I_{i,j_2j_1}(x)-\beta_{i,j_2} \circ \alpha_{i,j_1}(x) |\le C\eps_i 2^{-i}
	\end{equation}
	\begin{equation}\label{alminverseother}
		|I_{i,j_1j_2}(x)-\beta_{i,j_1} \circ \alpha_{i,j_2}(x) |\le C\eps_i 2^{-i}
	\end{equation}
	for every $x \in B_{45\cdot2^{-i}}(0)$.
\end{lemma}
\begin{proof}
	Observe that by hypothesis  $B_{90 \cdot 2^{-i}}(x_{j_2}) \subset B_{200\cdot 2^{-i}}(x_{j_1})$. Thus (recalling \eqref{centertocenter}) the map  $\beta_{j_1} \circ \alpha_{j_2} : B_{80 \cdot 2^{-i}}(0) \to {B_{200 \cdot 2^{-i}}(0)}$ is well defined and is a $C\eps_i 2^{-i}$ isometry. Therefore, if $\eps_i$ is small enough, we are in position to apply Lemma \ref{AA} and deduce that there exists a global isometry $I_{i,j_2j_1}: \rr^n \to \rr^n$ such that
	\begin{equation}\label{quasifatto}
		|I_{i,j_1j_2}(x)-\beta_{j_1} \circ \alpha_{j_2}(x) |\le C\eps_i 2^{-i}
	\end{equation} 
	for every $x \in B_{80\cdot2^{-i}}(0)$.
	This already proves \eqref{alminverse}. We  define $I_{i,j_2j_1}\coloneqq I_{i,j_1j_2}^{-1}$. Using \eqref{quasifatto} and \eqref{alfabeta}
	\begin{align*}
		|I_{i,j_2j_1}(x)-\beta_{j_2} \circ \alpha_{j_1}(x) |&=|x-I_{i,j_1j_2}(\beta_{j_2} \circ \alpha_{j_1}(x))|\\
	\text{by \eqref{quasifatto} \quad}	&\le|x-\beta_{j_1}(\alpha_{j_2}(\beta_{j_2}( \alpha_{j_1}(x)))) |+C\eps_i 2^{-i}\\
		&\le |x-\beta_{j_1}( \alpha_{j_1}(x)) |+C\eps_i 2^{-i}\\
		&\le C\eps_i2^{-i}
	\end{align*}
	for every $x \in B_{45 \cdot 2^{-i}}(0).$ This proves \eqref{alminverseother}. Observe that to justify the use of \eqref{quasifatto}  above we need to check that
	\begin{equation}\label{checkball}
		\beta_{j_2}(\alpha_{j_1}(B_{45\cdot 2^{-i}}(0)))\subset B_{80\cdot 2^{-i}}(0).
	\end{equation}
	To prove this observe that  from \eqref{centertocenter} and the fact that $\sfd(x_{j_1},x_{j_2})<30\cdot 2^{-i}$, it follows that $\alpha_{j_1}(B_{45\cdot 2^{-i}}(0))\subset B_{46 \cdot 2^{-i}}(x_{j_1})\subset B_{76 \cdot 2^{-i}}(x_{j_2})$.  From this, using again  \eqref{centertocenter}, we obtain \eqref{checkball}.
\end{proof}
\begin{definition}\label{maps}
	For any couple of indices $j_1,j_2 \in J_i$ such that $\sfd(x_{j_1},x_{j_2})<30\cdot 2^{-i}$ we choose once and for all a couple of maps  $I_{i,j_1j_2},I_{i,j_2j_1}$ that satisfy the conclusion of  Lemma \ref{Definition of the transition maps}.
	We then choose once and for all a   sufficiently large  subset $\mathcal I_i$ of these maps, so that
 \begin{equation}\label{eq:def Ii}
   \{I_{i,j_1j_2}\ : j_1,j_2 \in J_i \text{ and }\sfd(x_{j_1},x_{j_2})<29\cdot 2^{-i}\}\subset   \mathcal I_i 
 \end{equation}
\end{definition}
Again for simplicity we will often write $I_{j_2j_1}$ in place of $I_{i,j_2j_1}.$ Moreover from now on we will essentially forget and about the maps $I_{i,j_1j_2}$ that are not in $\mathcal I_i.$

\begin{remark}[On the choice of $\mathcal I_i$] 
    The reason we want some freedom in the definition of $\mathcal I_i$  in Definition \ref{maps} is only because doing so  makes the proof of Theorem .... much easier. Actually this is needed only for $i=0.$ For $i\ge 1$ (and also if $i=0$ in the case the reader was not interested in proving Theorem ...) we can  simply take $\mathcal I_i$ to be \textit{all the maps}, i.e.\ $\mathcal I_i\coloneqq \{I_{i,j_1j_2}\ : j_1,j_2 \in J_i \text{ and }\sfd(x_{j_1},x_{j_2})<30\cdot 2^{-i}\}$.
\end{remark}

\begin{remark}\label{QI}
	For any two maps $I_{j_1j_2}, I_{j_1j_3}\in \mathcal I_i$ with $j_2 \neq j_3$, it holds that $j_1 \in J_i^{a_1},j_2 \in J_i^{a_2},j_3\in J_i^{a_3}$ with $a_1\neq a_2 \neq a_3 \neq a_1$. This follows from the definition of the sets $J_{i}^k$ and the fact that $\sfd(x_{j_1},x_{j_2})< 30\cdot  2^{-i},$ $\sfd(x_{j_1},x_{j_3})< 30\cdot  2^{-i}$ and thus $\sfd(x_{j_2},x_{j_3})< 60\cdot  2^{-i}$. \fr
\end{remark}

In the following key result we observe that isometries that we just defined are almost compatible in a suitable sense. 
\begin{lemma}\label{almcomp}
	 Suppose that for some triple of indices $j_1,j_2,j_3 \in J_i$ the maps $I_{j_3j_2},I_{j_2j_1},I_{j_3j_1}$ are defined.
	Then 
	\[|I_{j_3j_2}(I_{j_2j_1}(x))-I_{j_3j_1}(x)|\le  C\eps_i 2^{-i}\] 
	for every $x \in B_{12\cdot2^{-i}}(0).$
\end{lemma}
\begin{proof} 
	Observe that by construction the existence of the maps  $I_{j_3j_2},I_{j_2j_1},I_{j_3j_1}$  implies that $\sfd(x_{j_1},x_{j_2})<30\cdot 2^{-i}$, $\sfd(x_{j_1},x_{j_3})<30\cdot 2^{-i}$ and $\sfd(x_{j_3},x_{j_2})<30\cdot 2^{-i}$. Then
	\begin{align*}
		|I_{j_3j_2}(I_{j_2j_1}(x))-I_{j_3j_1}(x)|&\le |I_{j_3j_2}(I_{j_2j_1}(x))-\beta_{j_3}(\alpha_{j_1}(x))|+C \eps_i 2^{-i}\\
		&\le|I_{j_3j_2}(\beta_{j_2}(\alpha_{j_1}(x)))-\beta_{j_3}(\alpha_{j_1}(x))|+C \eps_i 2^{-i}\\
		&\le  |\beta_{j_3}(\alpha_{j_2}(\beta_{j_2}(\alpha_{j_1}(x))))-\beta_{j_3}(\alpha_{j_1}(x))|+C \eps_i 2^{-i}\\
		&\le |\beta_{j_3}(\alpha_{j_1}(x))-\beta_{j_3}(\alpha_{j_1}(x))|+C \eps_i 2^{-i}=C \eps_i 2^{-i},
	\end{align*}
	for any $x \in B_{12 \cdot 2^{-i}}(0),$ where we used \eqref{alminverse} in the first three inequalities and in the last inequality we used \eqref{alfabeta}. 
	Observe that, to justify the use of \eqref{alminverse} in the third inequality above we need to check that
	\[\beta_{j_2}\circ \alpha_{j_1}(B_{12\cdot 2^{-i}}(0))\subset B_{45 \cdot 2^{-i}}(0), .\] To see this observe that by \eqref{centertocenter} $\alpha_{j_1} (B_{12\cdot 2^{-i}}(0))\subset B_{13\cdot 2^{-i}}(x_{j_1})$. Then, since $\sfd(x_{j_1},x_{j_2})<30\cdot 2^{-i}$, we have $\alpha_{j_1} (B_{12\cdot 2^{-i}}(0))\subset B_{43\cdot 2^{-i}}(x_{j_2})$. Now applying again \eqref{centertocenter} we conclude.
\end{proof}	

\subsubsection{Modification of the isometries to get compatibility}\label{sec:modification charts}
Our plan is now to construct the manifold $W_i$ by gluing together a number of copies of Euclidean balls in the following way:
\[ W_i\coloneqq  \bigslant{\bigsqcup_{j \in J_i} B^j_{10\cdot 2^{-i}}(0) }{\sim},	 \]
where:
\begin{quote}
    $x \sim y$ for $x \in B^{j_1}_{10\cdot 2^{-i}}(0), y \in B^{j_2}_{10\cdot 2^{-i}}(0) $ if and only if $I_{i,j_2j_1}\in \mathcal I_i$ and  $I_{i,j_2j_1}(x)=y$  (we set  also $x \sim x$ for every $x$).
\end{quote}
  Notice that this is an equivalence relation if and only if the following compatibility relation is true. 
For every couple of maps $I_{i,j_1j_2},I_{i,j_3j_2}\in \mathcal I_i$ for which there exists $x \in B_{10\cdot 2^{-i}}(0)$ such that $I_{i,j_2j_1}(x),I_{i,j_3j_2}(I_{i,j_2j_1}(x))\in B_{10\cdot 2^{-i}}(0)$ we have that  $I_{i,j_3j_1}\in \mathcal I_i$  and 
\[  I_{i,j_3j_2}( I_{i,j_2j_1}(x))=I_{i,j_3j_1}(x).\]
However there is no reason for this to be true in general for the maps that we have defined at the moment. Still it is almost true in a quantitative sense given by  Lemma \eqref{almcomp}. Thanks to this we can perform a small modification of the maps $I_{i,j_1j_2}$ to obtain new maps $\widetilde I_{i,j_1j_2}$ such that the above compatibility relations are satisfied. This modification is quite involved and is actually independent of the metric space $(Z,\sfd)$: instead it takes place entirely in the Euclidean space. For this reason this procedure will be developed separately in Section \ref{sec:modification} and summarized in Theorem \ref{bigmodification2}. Applying that result we can obtain the following (see Section \ref{sec:ck norm} for the definition of the norm $\|\cdot \|_{C^2,t}$).

\begin{lemma}\label{bigmodification}
There exists a constant $\tilde C$ depending only on $n$ such that the following holds.
 For every $I_{i,j_1j_2}\in \mathcal I_i$  there exists a $C^\infty$ diffeomorphism $\widetilde I_{i,j_1j_2}:\rr^n \to \rr^n$ such that 
	\begin{equation}\label{finalb}
		\|I_{i,j_1j_2}-\widetilde I_{i,j_1j_2} \|_{C^2(\rr^n),2^{-i}}\le \tilde C \eps_i,
	\end{equation}
	$\widetilde I_{i,j_2j_1}=\widetilde I_{i,j_1j_2}^{-1}$ and the following compatibility condition holds:
	for every couple of maps $\tilde I_{i,j_1j_2},\tilde I_{i,j_3j_2}$ for which the set  $\{x \in B_{8\cdot 2^{-i}}(0) \ : \ \tilde I_{i,j_2j_1}(x),\tilde I_{i,j_3j_2}(\tilde I_{i,j_2j_1}(x))\in B_{8\cdot 2^{-i}}(0)\}$ is non empty, we have $I_{i,j_3j_1}\in \mathcal I_i$ (and so also $\tilde I_{i,j_3j_1}$ is defined)  and 
	\[  \tilde I_{i,j_3j_2}(\tilde I_{i,j_2j_1}(x))=\tilde I_{i,j_3j_1}(x),\]
	for every $x$ in the above set.
\end{lemma}
\begin{proof}
	It is enough to apply Theorem \ref{bigmodification2} with scale $t=2^{-i}$, $\beta=C \eps_i$, set of indices $J= J_i$ partitioned into disjoint sets $\{J_i^k\}_{k=1}^{N_i}$, couple of indices  $\mathcal{A}=\{(j_1,j_2) \in J_i\times J_i \ : \ I_{j_1j_2} \in \mathcal I_i\}$, where $\mathcal I_i$ is the family of maps in Definition \ref{maps}, and finally with the maps $\{I_{j_1j_2}\}_{(j_1,j_2)\in \mathcal A}=\mathcal I_i$. Indeed observe that thanks to Proposition \ref{prop:covering} we can take  $M=N(n)$ as the upper bound on the cardinality  of  the partitions $J_i$, therefore the constants appearing in Theorem \ref{bigmodification2} will depend in this case only on $n$. Thus  we just need to check the hypotheses of Theorem \ref{bigmodification2}.
	
	Assumption \eqref{asimm} follows by construction and \eqref{apartition} follows by Remark \ref{QI}. Property \ref{A)} is true again by construction. Property \ref{C)} with $\beta=C \eps_i$ is  the content of Lemma \ref{almcomp}.  Note also that $\beta=C \eps_i\le \bar \beta(n,M)$ provided $\eps_i$ is small enough, where $\bar \beta(n,M)$ is the constant in  Theorem \ref{bigmodification2}. It remains to check property \ref{B)}. Suppose  $I_{j_3j_2}( I_{j_2j_1}(x))\in B_{9\cdot 2^{-i}}(0)$ for some $x \in B_{8\cdot 2^{-i}}(0)$. Recalling now \eqref{alminverse}, provided $\eps_i$ is small enough, we deduce that $y\coloneqq \beta_{j_3}\circ \alpha_{j_2}\circ \beta_{j_2}\circ \alpha_{j_1}(x) \in B_{10\cdot 2^{-i}}(0)$. Observe that we can apply \eqref{alminverse} since $\beta_{j_2}\circ \alpha_{j_1}(x) \in B_{45\cdot 2^{-i}}(x_{j_2})$, thanks to \eqref{centertocenter} and the fact that $\sfd(x_{j_1},x_{j_2})<30\cdot 2^{-i}$ which holds since the map $I_{j_1j_2}$ is defined.
	Therefore from \eqref{centertocenter} it holds that $\alpha_{j_1}(x)\in B_{10\cdot 2^{-i}}(x_{j_1})$ and $\alpha_{j_3}(y)\in B_{11\cdot 2^{-i}}(x_{j_3})$. Moreover from \eqref{alfabeta} we have that $\sfd(\alpha_{j_3}(y),\alpha_{j_1}(x))\le 2^{-i} $, if $\eps_i$ is small enough. By the triangle inequality this implies that $\sfd(x_{j_1},x_{j_3})\le (10+11+1)2^{-i}<29\cdot 2^{-i}$, hence the map $I_{j_3j_1}$ exists by our choice of $\mathcal I_i$ (recall \eqref{eq:def Ii}). This proves property \ref{B)} and concludes the proof.
\end{proof}	

\subsubsection{Construction of the differentiable manifolds}\label{sec:build manifold}
We can now construct $W_i$ as we mentioned above. Specifically we will construct it starting from its transition maps and their domains of definition. This kind of construction, even if very natural, needs to be done carefully to ensure that the resulting space is a well defined Hausdorff manifold. In doing this we will refer to the formalization through  \emph{Sets of Gluing Data} described in \cite[Thereom 3.1]{glue}.  We start defining maps $\hat I_{i,j_1j_2}$ that are restriction of the maps $\widetilde I_{i,j_1j_2}$. In particular we set
\begin{equation}\label{eq:def omega charts}
\Omega^i_{j_1j_2}\coloneqq \widetilde I_{i,j_1j_2}^{-1}(B_{8\cdot 2^{-i}}(0))\cap B_{8\cdot 2^{-i}}(0)  \end{equation}
and, whenever $\Omega^i_{j_1j_2}\neq \emptyset$, we set  $\hat I_{i,j_1j_2}\coloneqq \widetilde I_{i,j_1j_2}|_{\Omega^i_{j_1j_2}}$. Observe that the sets $\Omega^i_{j_1j_2}$ are open subsets of $\rr^n.$ Then we define
\[ W_i\coloneqq  \bigslant{\bigsqcup_{j \in J_i} B^j_{8\cdot 2^{-i}}(0) }{\sim},	 \]
where  
\begin{quote}
    $x \sim y$ (with $x \neq y$) if and only if $x \in B^{j_1}_{8\cdot 2^{-i}}(0), y \in B^{j_2}_{8\cdot 2^{-i}}(0) $, $I_{i,j_2j_1}\in \mathcal I_i$  and $\hat I_{i,j_2j_1}(x)=y$. We impose also $x \sim x$ for every $x$.
\end{quote}
 This is now a valid equivalence relation by Lemma \ref{bigmodification} (referring to the construction in \cite{glue}, the compatibility condition ensured by Lemma \ref{bigmodification} implies condition $(3.c)$ in Definition 3.1 of \cite{glue}). The set $W_i$ comes with a natural quotient map
\[ \Pi: \bigsqcup_{j \in J_i} B^j_{8\cdot 2^{-i}}(0)  \to W_i, \]
defined as $\Pi(x)\coloneqq [x]$, where $[x]$ is the equivalence class containing $x$. Define now for every $j \in J_i$ the map 
\[p_{j}\coloneqq \Pi \circ 	{\sf{in}}_j : B_{8\cdot 2^{-i}}(0) \to W_i,\]
where $B_{8\cdot 2^{-i}}(0) \overset{\sf{in}_j}{\xhookrightarrow{}}  \bigsqcup_{j \in J_i} B^j_{8\cdot 2^{-i}}(0) $ is the natural inclusion map on the $j$-th component.
Observe that, by construction, the map $\Pi$ is injective on every $B^{j}_{8\cdot 2^{-i}}(0)$ and thus also $p_{j}$ is injective. We endow $W_i$ with the topology $\tau$ defined by
\[\tau=\{U\subset W_i \ : \ p_j^{-1}(U) \text{ is an open subset of } \rr^n,\, \forall j \in J_i\}.\]
It is easy to check that $\tau$ is indeed a topology. We claim that with this topology, the sets $p_j(B_{8\cdot 2^{-i}}(0))$ are open for every $j \in J_i$ and that the maps $p_j: B_{8\cdot 2^{-i}}(0)\to p_j(B_{8\cdot 2^{-i}}(0))$ are homomorphisms for every $j \in J_i.$ To prove the claim it sufficient to show that  $p_j(V)\in \tau$, for every $j \in J_i$ and every $V\subset B_{8\cdot 2^{-i}}(0)$ open (possibly taking $V=B_{8\cdot 2^{-i}}(0)$). This in turn follows observing that by construction
\[
p_{j_2}^{-1}(p_{j_1}(V))=\Omega^i_{j_1j_2}\cap V, \quad \forall j_1,j_2 \in J_i
\]
and recalling that the sets $\Omega^i_{j_1j_2}$ are open.  

For every $j \in J_i$ we set 
\[
B_j^i\coloneqq p_j(B_{8\cdot 2^{-i}}(0))
\]
 and define the maps 
 \begin{align*}
      \psi_{i,j}\coloneqq p_j^{-1}: B_j^i \to B_{8\cdot 2^{-i}}(0).
 \end{align*}
These maps will be the charts of the manifold (note that they are heomeomorphisms, since so are the $p_j$'s  as observed above). As usual we will write $\psi_{j}$ instead of $\psi_{i,j}$ when there will be no ambiguity in doing so.
We need to show that $(W_i,\tau)$ is Hausdorff. To do this it is enough to observe that 
\[ \widetilde I_{i,j_1j_2} (\partial (\Omega^i_{j_1j_2})\cap B_{8\cdot2^{-i}}(0)) \subset \partial B_{8\cdot2^{-i}}(0) \]
that comes from the fact that $\widetilde I_{i,j_1j_2}$ is a global diffeomorphism and the definition of the sets $\Omega^i_{j_1j_2}$. This fact implies condition $(3.d)$ in Definition 3.1 of \cite{glue}, which is there shown to be enough to ensure that $W_i$ is a Hausdorff space.
It remains to endow $W_i$ with a smooth structure. If follows from the construction that whenever 
$$B_{j_1}^i\cap B_{j_2}^i=p_{j_1}( B_{8\cdot 2^{-i}}(0))\cap p_{j_2}( B_{8\cdot 2^{-i}}(0))\neq \emptyset$$
then the map $\hat I_{i,j_2j_1}$  exists and by the definition of the equivalence relation we have
\begin{equation}\label{eq:correct transitions}
    \psi_{i,j_2} \circ (\psi_{i,j_1})^{-1}= p_{j_2}^{-1} \circ p_{j_1}=\hat I_{i,j_2j_1}.
\end{equation}
This shows precisely that the maps $\hat I_{i,j_2j_1}$ (which are  diffeomorphisms) are the transitions maps for the charts $\psi_{i,j}$.
Moreover we proved above that the sets $B_j^i$ are open in $(W_i,\tau)$. This proves that  $(B^i_j ,\psi_{i,j})_{j \in J_i}$ is a  $C^\infty$-smooth structure for $W_i.$ 
\begin{lemma}\label{connected}
	The manifold $W_i$ is connected.
\end{lemma}	
\begin{proof}
	We argue by contradiction and suppose that $W_i=U\cup V$, where  $U,V$ are two non-empty disjoint open subsets of $W_i$.  Since $W_i\subset \cup_{j \in J_i} B_j^i$ and each $B_j^i$ is non-empty and connected (being homeomorphic to $B_{8\cdot 2^{-i}}(0)$),  we must have that 
	\begin{equation}\label{conncomp}
		U=\bigcup_{j\in S} B_j^i, \quad V=\bigcup_{j\in \tilde S} B_j^i,
	\end{equation}
	where $S,\tilde S$ are two disjoint subsets of the  set $J_i$. In particular $B_{j_1}^i\cap B_{j_2}^i=\emptyset $ for every $j_1 \in S$ and every $j_2\in \tilde S.$ We claim that $\sfd(x_{j_2},x_{j_1})>5\cdot 2^{-i}$ for every $j_1 \in S$ and every $j_2\in \tilde S.$ To see this note that the transition map $\hat I_{j_1j_2}$ must not be defined. By construction this can happen only in two cases. The first case is that the map $\tilde I_{j_1j_2}$ is not defined, the second is that $\tilde I_{j_1j_2}$ is defined but $\tilde I_{j_1j_2}^{-1}(B_{8\cdot 2^{-i}}(0))\cap B_{8\cdot 2^{-i}}(0) =\emptyset$ (recall \eqref{eq:def omega charts}). In the first case by construction (recall Definition \ref{maps}) we must have that $\sfd(x_{j_1},x_{j_2})\ge 29\cdot 2^{-i}.$ In the second case, by \eqref{finalb} and \eqref{alminverse} we have that $|\beta_{j_1} \circ \alpha_{j_2}-\tilde I_{j_2j_1}^{-1}|\le C\eps_i 2^{-i}$  in $B_{8\cdot2^{-i}}(0)$. Hence, if $\eps_i$ is small enough, $\beta_{j_1}( \alpha_{j_2}(0))\notin B_{7\cdot 2^{-i}}(0),$ that combined with \eqref{centertocenter1} and provided $\eps_i$ small enough, gives $\sfd(\alpha_{j_2}(0),x_{j_1})>6\cdot 2^{-i}$. Then again using \eqref{centertocenter1} we deduce that $\sfd(x_{j_2},x_{j_1})>5\cdot 2^{-i}.$ Therefore in both cases $\sfd(x_{j_2},x_{j_1})>5\cdot 2^{-i}$ and the claim is proved. This implies that
	\[\bigg(\bigcup_{j \in S} B^Z_{2^{-i}}(x_{j})\bigg)\cap \bigg(\bigcup_{j \in \tilde S} B^Z_{2^{-i}}(x_{j})\bigg)=\emptyset.\]
	However, since $X_i=\{x_{j}\}_{j\in J_i}$ is $2^{-i}$-dense, the above two sets cover $Z$, which contradicts the connectedness of $Z$.
\end{proof}

\subsubsection{Construction of the Riemannian-metrics}\label{sec:build metric}
It remains to endow the smooth manifolds $W_n^i$ with suitable Riemannian metrics $g_i$.

To construct the metric $g_i$ consider a partition of the unity $\rho_{i,j}$ subordinate to the cover $\{B^i_j\}_{j\in J_i}$ and define
\[ g_i\coloneqq  \sum_{j \in J_i}  \rho_{i,j}\psi_{i,j}^*g   \] 
where $g$ is the standard Euclidean metric on $\rr^n.$ Observe that the covering $\{B^i_j\}_{j\in J_i}$ is locally finite with multiplicity less than $N=N(n)$. This is a consequence of Remark \ref{QI}. We denote with $\sfd_i$ the Riemannian distance function induced by the metric $g_i$ (recall that from Lemma \ref{connected} $W_i$ is connected, hence $\sfd_i$ is finite).
The first key observation is  that, with this metric, the charts $\psi_{i,j}$ turns out to be bi-Lipschitz maps, in particular we have the following.
\begin{lemma}\label{lippsii}
	For every $j\in J_i$ consider $\psi_{j}$ as a smooth function $\psi_{j}: B_j^i\to \rr^n$, then
	\begin{equation}\label{lippsi}
		\|D\psi_{j}\|, \|D\psi_{j}^{-1}\|\le 1+C\eps_i,
	\end{equation}
	where $\|\cdot \|$ denotes the operator norm.
\end{lemma}
\begin{proof}
	Fix $p \in B^i_j$ and pick any $v \in T_p W_i.$ We can write $v$ in local coordinates with respect to the chart $\psi_{j}=(x_1^j,...,x_n^j)$ as $v=a^k\frac{\partial}{\partial x^j_k}.$ Set $ w\coloneqq  (D\psi_{j})_p\cdot v \in \rr^n$ and observe that $w=(a_1,...,a_n).$ We now need to compute $|v|_{T_pW_i}.$ Recall that $g_p= \sum_{k\in J_i} \rho_{i,k}(p) \psi_{k}^*g$, thus 
	\begin{align*}
		|v|^2_{T_pW_i}=\sum_{k\in J_i} \rho_{i,k}(p) |(D\psi_{k})_pv|^2=\sum_{k\in J_i} \rho_{i,k}(p) |A_{k,j}\bar a|^2.
	\end{align*}
	where $A_{k,j}$ is the Jacobian matrix of the transition function $\psi_k \circ (\psi_j)^{-1}$, that by construction coincides with $\hat I_{k,j}.$ Therefore from \eqref{finalb} we have that $ 1-C\eps_i\le \|A_{k,j}\|\le 1+C\eps_i$ (indeed  $I_{i,j_1j_2}$ are isometries). Thus
	\begin{equation*}
		(1 -C\eps_i) |w|\le|v|_{T_pW_i}\le (1 + C\eps_i) |w|,
	\end{equation*}
	which gives 
	\[\frac{1}{1+C\eps_i}  |v|\le|w|\le (1 + 2C\eps_i) |v|,\]
	that is what we wanted.
\end{proof}

\begin{lemma}\label{ballsinballs}
    It  holds that
	$$B^{W_i}_{\frac{s}{1+C\eps_i}}(\psi_{j}^{-1}(x))\subset \psi_{j}^{-1}(B_{s}(x))\subset B^{W_i}_{s(1+C\eps_i)}(\psi_{j}^{-1}(x))$$
	for every ball $B_s(x)\subset B_{8\cdot 2^{-i}}(0)$.
\end{lemma}
\begin{proof}
     It follows  immediately from the fact that $\psi_{j}$ are $1+C\eps_i$ bi-Lipschitz homeomorphisms (from \eqref{lippsi}) and the fact that the distances in $\rr^n$ and $W^i$ are geodesic.
\end{proof}

\begin{remark}\label{rmk:complete}
	The manifolds $W_i$ endowed with the distance $d_i$ are complete, however we will postpone the proof of this fact to the end of the next subsection (see Lemma \ref{lem:complete manifolds}).\fr
\end{remark}
\subsection{Construction of the manifold-to-metric space mappings \texorpdfstring{$f_i$}{fi}}\label{sec:manifold metric mappings}

For every $w\in W_i$ we define
\[f_i(w)\coloneqq \alpha_{j}(\psi_{j}(w)), \]
where $j \in J_i$  is the smallest such that $w\in B_j^i$.
We remark that with this definition $f_i(w)$ will depend on the choice of the index $j$, indeed the map $f_i$ will not be even continuous in general. 
However in the following statement we prove that  $f_i$ is almost unique, in the sense that the a different choice of $j$ in its definition change its value by at most $\eps_i2^{-i}. $
\begin{lemma}\label{welldeffii}
	\begin{equation}\label{welldeffi}
		\sfd(f_i(x),\alpha_{j}(\psi_{j}(x)))\le C\eps_i 2^{-i}, \quad \forall\, j \in J_i,\,\, \forall\, x \in B_j^i.
	\end{equation}
\end{lemma}
\begin{proof}
	By definition $f_i(x)=\alpha_k(\psi_k((x)))$ for some $k \in J_i$ for which $x \in B_k^i.$ If $k=j$ there is nothing to prove, so suppose $k \neq j.$ Then
	\begin{align*}
		\sfd(f_i(x),\alpha_{j}(\psi_j(x)))&=\sfd(\alpha_{k}(\psi_{k}(x)),\alpha_{j}(\psi_j(x)))=\\
		&=\sfd(\alpha_{k}(\psi_{k}(x)),\alpha_{j}(\psi_j(\psi_{k}^{-1}(\psi_{k}(x))))=\\
		&=\sfd(\alpha_{k}(\psi_{k}(x)),\alpha_{j}(\hat I_{j,k}(\psi_{k}(x)))\\
		&\le \sfd(\alpha_{k}(\psi_{k}(x)),\alpha_{j}(\beta_j\circ \alpha_k(\psi_{k}(x)))+C\eps_i2^{-i},
	\end{align*}
	where the last inequality follows combining \eqref{alminverse} and \eqref{finalb}.
	Recalling now \eqref{alfabeta} we obtain \eqref{welldeffi}.
\end{proof}

By construction $f_i$ sends points in $B_i^j$ inside $B_{10\cdot 2^{-i}}(x_j)$. The following statement is the converse of this, meaning that points mapped by $f_i$ near $x_j$ must belong to the coordinate patch $B_i^j$.
\begin{lemma}\label{correspl}
	 Suppose that for some $w \in W_i$ it holds that $f_i(w)\in B_{s\cdot 2^{-i}}(x_{j})$ with $s<6.$ Then, provided $\eps_i$ is small enough, the map $\psi_j$ exists and we have
	\begin{equation}\label{corresp}
		w \in \psi_j^{-1}(B_{(s+C\eps_i)2^{-i}}(0)),
	\end{equation}
	in particular $w \in B_j^i.$
\end{lemma}
\begin{proof}
	By definition $f_i(w)=\alpha_k(\psi_k(w))$ for some $k \in J_i$. By definition of $\psi_k$ we have $\psi_k(w)\in B_{8\cdot 2^{-i}}(0)$ and so by \eqref{centertocenter} it holds $f_i(w)\in B_{9\cdot 2^{-i}}(x_{k})$. In particular,$f_i(w)\in B_{s\cdot 2^{-i}}(x_{j})$ with $s<6$, we obtain $\sfd(x_k,x_j)<15\cdot 2^{-i}.$ In particular $I_{j,k}\in \mathcal I_i$ (recall \eqref{eq:def Ii}). Observe now that from \eqref{centertocenter1} we deduce that $\beta_{j}(\alpha_k(\psi_k(w)))=\beta_{j}(f_i(w)) \in B_{(s+C\eps_i)2^{-i}}(0)$. Moreover combining \eqref{finalb} and \eqref{alminverse} we have $|\tilde  I_{j,k}-\beta_j\circ \alpha_k|\le C\eps_i2^{-i}$ in $B_{45\cdot2^{-i}}(0)$. Hence  $\tilde I_{j,k}(\psi_k(w))\in B_{(s+C\eps_i)2^{-i}}(0)$, up to increasing the constant $C$. Since $s<6$, if $\eps_i$ is small enough we have that $\tilde I_{j,k}(\psi_k(w))\in B_{8\cdot 2^{-i}}(0),$ which implies $\psi_k(w) \in \Omega^i_{j,k}$ (where $\Omega^i_{j,k}$ is as in \eqref{eq:def omega charts}). Since $\Omega^i_{j,k}$ is the domain of $\hat I_{j,k}$ we can write
	$$w=\psi_j^{-1}(\psi_j \circ \psi_k^{-1})(\psi_k(w))=\psi_j^{-1}(\hat I_{j,k}(\psi_k(w))),$$
 where we used \eqref{eq:correct transitions}.
	Since we showed that $\tilde I_{j,k}(\psi_k(w))\in B_{(s+C\eps_i)2^{-i}}(0)$, this proves \eqref{corresp}.
\end{proof}
The following  result tells us that, locally, $f_i$ is a $2^{-i}C\eps_i$-isometry.
\begin{lemma}\label{fidii}
For every $j\in J_i$ it holds
	\begin{equation}\label{fid}
		|\sfd(f_i(w_1),f_i(w_2))-\sfd_i(w_1,w_2)|\le 2^{-i}C\eps_i, \quad \text{for every $w_1,w_2 \in B_j^i$.}
	\end{equation}
\end{lemma}
\begin{proof}
	By \eqref{welldeffi} we can suppose  that  $f_i(w_1)=\alpha_{j}(\psi_j(w_1)),f_i(w_2)=\alpha_{j}(\psi_j(w_2))$. Thus we need to bound $|\sfd(\alpha_{j}(\psi_j(w_1)),\alpha_{j}(\psi_j(w_2)))-\sfd_i(w_1,w_2)|,$ however since $\alpha_{j}$ is a $C\eps_i 2^{-i}$ isometry we have $$|\sfd(\alpha_{j}(\psi_j(w_1)),\alpha_{j}(\psi_j(w_2)))-|\psi_{j}(w_1)-\psi_{j}(w_2)||\le C\eps_i 2^{-i}.$$ Hence we reduced ourselves to estimate $||\psi_{j}(w_1)-\psi_{j}(w_2)|-\sfd_i(w_1,w_2)|.$ To do this first observe that by Lemma \ref{ballsinballs}, if $\eps_i$ is small enough, we have 
	$$ B_j^i=\psi_j^{-1}(B_{8\cdot 2^{-i}}(0))\subset B_{10\cdot 2^{-i}}(\psi_j^{-1}(0)),$$
	hence $\sfd_i(w_1,w_2)\le 20\cdot 2^{-i}.$ 
	Then recalling that $\psi_{j}$ is $1+C\eps_i$ bi-Lipschitz we have 
	\[||\psi_{j}(w_1)-\psi_{j}(w_2)|-\sfd_i(w_1,w_2)|=\sfd_i(w_1,w_2)\left| \frac{|\psi_{j}(w_1)-\psi_{j}(w_2)|}{\sfd_i(w_1,w_2)} -1 \right|\le 20\cdot 2^{-i} C\eps_i,  \]
	that is what we wanted.
	
\end{proof}

Since the balls centered in $X_i$ of radius $2^{-i}$ covers $Z$, it is natural to expect that we can cover $W_i$ with images via $\psi_j^{-1}$ of balls with radius significantly smaller ${10\cdot 2^{-i}}$. This essentially is the content of the following result. 
\begin{lemma}\label{completeness}
	For every $i \in \mathbb{N}_0$ it holds that
	\begin{equation}\label{closedcover}
		W_i\subset \bigcup_{j\in J_i} \psi_{j}^{-1}\left(\overline{B_{2\cdot 2^{-i}}(0)} \right )\subset \bigcup_{j\in J_i} {B^{W_i}_{3\cdot 2^{-i}}(\psi_{j}^{-1}(0))} ) .
	\end{equation} 
\end{lemma}
\begin{proof}
	Let $w \in W_i$ then $f_i(w)\in B_{ 2^{-i}}(x_j)$ for some $j \in J_i$, since $X_i$ is $2^{-i}$-dense in $Z$ by construction. Therefore from Lemma \ref{correspl} we deduce that $w \in \psi_j^{-1}(B_{(1+C\eps_i)2^{-i}}(0))\subset \psi_j^{-1}\left (\overline {B_{(2\cdot 2^{-i}}(0)}\right),$ provided $\eps_i$ is small enough. This proves the first inclusion in \eqref{closedcover}. The second inclusion follows from Lemma \ref{ballsinballs}. 
\end{proof}

\begin{lemma}\label{lem:complete manifolds}
	The manifold $(W_i,\sfd_i)$ is complete.
\end{lemma}
\begin{proof}
	Let $w \in W_{i}^n,$ then from \eqref{closedcover}, there exists $j \in J_i$ such that
	$w \in B^{W_i}_{3\cdot 2^{-i}}(\psi_{j}^{-1}(0)).$ Hence $B^{W_i}_{2^{-i}}(w) \subset B^{W_i}_{4\cdot 2^{-i}}(\psi_{j}^{-1}(0)) $.  From Lemma \ref{ballsinballs} we  deduce, provided $\eps_i$ small enough, that $$B^{W_i}_{2^{-i}}(w)\subset B^{W_i}_{4\cdot 2^{-i}}(\psi_{j}^{-1}(0)) \subset \psi_{j}(B_{5\cdot 2^{-i}}(0))\subset \psi_{j}\left (\overline{B_{5\cdot 2^{-i}}(0)}\right ).$$ Since the last set is compact, we deduce that  $\overline{B^{W_i}_{2^{-i}}(w)}$ is also compact and from the arbitrariness of $w$ we conclude.
\end{proof}

\subsection{Construction of the pseudo-distances \texorpdfstring{$\rho_i$}{rho-i} }\label{sec:pseudo}

For every $w_1,w_2 \in W_i$ we set
\begin{equation}\label{rho}
	\rho_i(w_1,w_2)\coloneqq  \begin{cases}
		\sfd_i(w_1,w_2) & \text{ if } \sfd_i(w_1,w_2)\le 2\cdot 2^{-i},\\
		\sfd(f_i(w_1),f_i(w_2)) & \text{ if } \sfd_i(w_1,w_2)> 2\cdot 2^{-i}.\\
	\end{cases}
\end{equation}

\begin{lemma}\label{rhod}
	It holds that
	$$\rho_i(w_1,w_2)=\sfd_i(w_1,w_2),$$
	whenever $\rho_i(w_1,w_2)\le 2^{-i}$ or $\sfd_i(w_1,w_2)\le 2^{-i}$.
\end{lemma}
\begin{proof}
	If $\sfd_i(w_1,w_2)\le 2^{-i}$ the statement is trivial from the definition of $\rho_i$. So suppose $\rho_i(w_1,w_2)\le 2^{-i}$. If $\rho_i(w_1,w_2)=\sfd_i(w_1,w_2)$ we are done, so suppose $\rho_i(w_1,w_2)=\sfd(f_i(w_1),f_i(w_2)) $. Then the $2^{-i}$-density of $X_i$ implies that $f_i(w_1),f_i(w_2) \in B_{5\cdot 2^{-i}}(x_j)$ for some $j \in J_i.$ Hence thanks to Lemma \ref{correspl} we have that $w_1,w_2 \in B_j^i.$ Applying now \eqref{fid} we obtain $\sfd_i(w_1,w_2)\le 2^{-i}+C\eps_i 2^{-i}\le 2\cdot 2^{-i},$ but then from \eqref{rho} we deduce $\rho_i(w_1,w_2)=\sfd_i(w_1,w_2)$.\\		
\end{proof}

\subsection{Construction of the manifold-to-manifold mappings \texorpdfstring{$h_i$}{hi}}\label{sec:manifoldtomanifold}

This section is devoted to to proof the following Lemma.
\begin{lemma}\label{mainh}
	There exists a $1+C(\eps_i+\eps_{i+1})$-bi-Lipschitz diffeomorphism $h_i : W_i \to W_{i+1}$ such that 
	\begin{equation}\label{point4eq}
		\sfd(f_{i+1}(h_i(w)),f_i(w))\le C 2^{-i}(\eps_i+\eps_{i+1}), 
	\end{equation}
	for every $w \in W_i.$
\end{lemma}

Along this section $(\psi_{i+1,j})_{j\in J_{i+1}}$ denote the charts for the manifold $W_{i+1}$ as defined in Section \ref{sec:manifolds}.

\subsubsection{Construction of the new charts for \texorpdfstring{$W_i$}{Wi}}\label{sec:new charts}
The idea of the construction of the maps $h_i$ is to build a new atlas for $W_i$, where instead of using points in $X_i$, as in its construction, we use the points in $X_{i+1}.$ In particular we will construct (inverses of) charts $(\phi_{i,j})_{j\in J_{i+1}}$ for $W_{i}$ in direct relation with the charts $(\psi_{i+1,j})_{j\in J_{i+1}}$ of $W_{i+1}$. In this way we will be able to build maps 
$$h_{i,j}:=\psi_{i+1,j}^{-1}\circ \phi_{i,j}^{-1},$$
defined locally,  from $W_i$ to $W_{i+1}$. Then we  will patch all these maps using a technical result that will be proved later in Section \ref{sec:cheeger mapping} (see Theorem \ref{patching}).\\

For any $j \in J_{i+1}$ define $k(j)\in J_i$ as the minimal index so that $x_{i+1,j} \in B_{2^{-i}}(x_{i,k(j)})$. The number $k(j)$ is well defined thanks to the fact that $X_i$ is $2^{-i}$-dense. Consider the  maps
\[ \alpha_{i+1,j}: B_{100\cdot 2^{-i}}(0) \to B_{100\cdot 2^{-i}}(x_{i+1,j}) \]
\[ \beta_{i,k(j)}: B_{200\cdot 2^{-i}}(x_{i,k(j)}) \to B_{200\cdot 2^{-i}}(0), \]
which are respectively a $C \eps_{i+1}2^{-i}$-GH approximation and  a $C \eps_i 2^{-i}$-GH approximation. Analogously we consider the maps $\alpha_{i,k(j)},\beta_{i+1,j}$. It is crucial here that all these maps are the same considered in Section \ref{sec:manifolds}.
Since $B_{100\cdot 2^{-i}}(x_{i+1,j}) \subset B_{200\cdot 2^{-i}}(x_{i,k(j)}) $ we have that $\beta_{i,k(j)}\circ \alpha_{i+1,j}$ is well defined on the whole $B_{100\cdot 2^{-i}}(0)$ ad it is a $2^{-i}C(\eps_i+\eps_{i+1})$ isometry. 
\begin{lemma}\label{kjlemma}
	For every $j \in J_{i+1}$ there exists an isometry $K_{i,j}:\rr^n \to \rr^n$ such that 
	\begin{equation}\label{kj}
		|K_{i,j}-\beta_{i,k(j)}\circ \alpha_{i+1,j}|\le C(\eps_i+\eps_{i+1}) 2^{-i},
	\end{equation}
	in $B_{15\cdot 2^{-i}}(0)$ and
	\begin{equation}\label{kjinv} |K_{i,j}^{-1}-\beta_{i+1,j}\circ \alpha_{i,k(j)}|\le C(\eps_i+\eps_{i+1})2^{-i},
	\end{equation}
	in $K_{i,j}(B_{15\cdot 2^{-i}}(0))$.
\end{lemma}
\begin{proof}
	The existence of an isometry $K_{i,j}:\rr^n \to \rr^n$ satisfying 
	\begin{equation}\label{asdasd}
		|K_{i,j}-\beta_{i,k(j)}\circ \alpha_{i+1,j}|\le C(\eps_i+\eps_{i+1})2^{-i}
	\end{equation}
	in $B_{20 \cdot 2^{-i}}(0)$
	is an immediate consequence of Lemma \ref{AA}, provided $\eps_i,\eps_{i+1}$  are small enough. To prove \eqref{kjinv} take any $x \in B_{15 \cdot 2^{-i}}(0)$, then using \eqref{asdasd} 
	\begin{align*}
		|K_{i,j}^{-1}(K_{i,j}(x))-\beta_{i+1,j}(\alpha_{i,k(j)}(K_{i,j}(x))|&= |x-\beta_{i+1,j}(\alpha_{i,k(j)}(K_{i,j}(x))|\\
		&\le|x-\beta_{i+1,j}(\alpha_{i,k(j)}(\beta_{i,k(j)}(\alpha_{i+1,j}(x)))| +C(\eps_i+\eps_{i+1}) 2^{-i}\\
		&\le C(\eps_i+\eps_{i+1}) 2^{-i},
	\end{align*}
	where in the last inequality we used \eqref{alfabeta}.  This would prove \eqref{kjinv}, however these computations needs justification. More precisely we need to check that $\beta_{i+1,j}\circ \alpha_{i,k(j)}$ is defined on $\beta_{i,k(j)}(\alpha_{i+1,j}(B_{15 \cdot 2^{-i}}(0)))$ and on  $K_{i,j}(B_{15 \cdot 2^{-i}}(0))$. Since, as remarked previously $\beta_{i,k(j)}\circ \alpha_{i+1,j}$ is defined on $B_{20 \cdot 2^{-i}}(0)$, from \eqref{alfabeta} and \eqref{centertocenter1} we have that  
	\begin{equation}\label{asdasd2}
		\alpha_{i,k(j)}(\beta_{i,k(j)}(\alpha_{i+1,j}(B_{15 \cdot 2^{-i}}(0))))\subset B_{20\cdot 2^{-i}}(x_{i+1,j}),
	\end{equation}
	which is the domain of $\beta_{i+1,j}$. Thus the first condition is verified. For the second we start proving that $\alpha_{i,k(j)}$  is defined on $K_{i,j}(B_{15\cdot 2^{-i}}(0))$. Indeed from \eqref{asdasd2} and \eqref{alfabeta} we deduce that  $\beta_{i,k(j)}(\alpha_{i+1,j}(B_{15\cdot 2^{-i}}(0))\subset B_{25\cdot 2^{-i}}(0) $. Hence from \eqref{asdasd}, provided $\eps_i,\eps_{i+1}$ are small enough, it follows that $K_{i,j}(B_{15\cdot 2^{-i}}(0))\subset B_{30\cdot 2^{-i}}(0)$, that is in the domain of $\alpha_{i,k(j)}.$ Now from \eqref{asdasd}, \eqref{asdasd2} and assuming $\eps_i,\eps_{i+1}$ small enough we deduce that
	\[ \alpha_{i,k(j)}(K_{i,j}(B_{15 \cdot 2^{-i}}(0)))\subset B_{35\cdot 2^{-i}}(x_{i+1,j}),\]
	which is the domain of $\beta_{i+1,j}$. This proves also the second condition and concludes the proof.
\end{proof} 
Define for every $j \in J_{i+1}$ the map $\phi_{i,j}: B_{4\cdot 2^{-i}}(0) \to W_i$ as
\[ \phi_{i,j}\coloneqq \psi_{i,k(j)}^{-1} \circ K_{i,j} |_{B_{4\cdot 2^{-i}}(0)}.\]
Notice that this is well defined since $K_{i,j}(B_{4\cdot 2^{-i}}(0)) \subset B_{8\cdot 2^{-i}}(0)$, by \eqref{kj}, assuming $\eps_i,\eps_{i+1}$ small enough. For every $j \in J_{i+1}$ define also the map $h_{i,j}: \phi_{i,j}(B_{4 \cdot 2^{-i}}(0))\to W_{i+1}$ as
\begin{equation}\label{eq:def hi} h_{i,j}\coloneqq \psi_{i+1,j}^{-1} \circ  \phi_{i,j}^{-1}|_{\phi_{i,j}(B_{4 \cdot 2^{-i}}(0))}.\end{equation}
\\

\begin{figure}[!htb]	\label{hfigure}
	\centering
	\includegraphics[width=0.8\textwidth, angle=0]{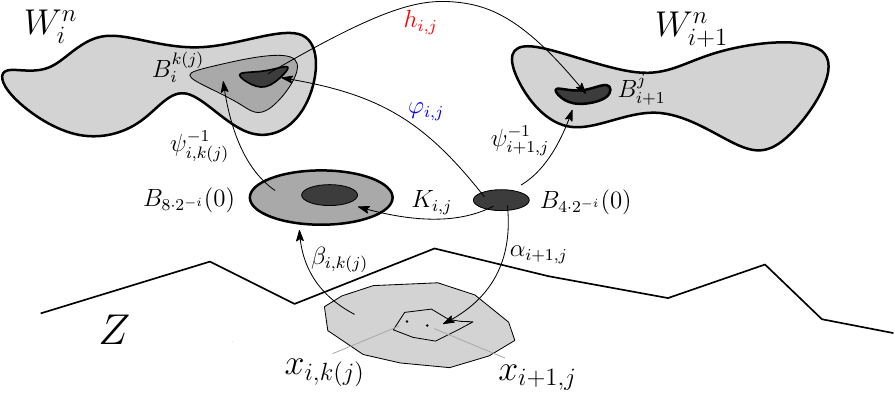}
	\caption{Diagram of definition of the maps $\phi_{i,j}$ and the maps $h_{i,j}$.}
\end{figure}

The following is the analogous to Lemma \ref{correspl} with respect to the new  (inverses of) charts, which also imply that the new charts cover the manifold.

\begin{lemma}\label{wellcovered}
	Suppose $w \in W_i$ and $x_{i+1,j}\in X_{i+1}$ are such that $\sfd(f_i(w),x_{i+1,j})<2^{-i-1}$, then
	\begin{equation}\label{nextchart}
		w \in \phi_{i,j}(B_{\frac{5}{4}\cdot 2^{-i}}(0)).
	\end{equation}
	In particular we have
	\begin{equation}\label{wellcoveredeq}
		W_{i}^n\subset \bigcup_{j\in J_{i+1}} \phi_{i,j}(B_{\frac{5}{4}\cdot 2^{-i}}(0)).
	\end{equation}
\end{lemma}
\begin{proof}
	Observe that  $\sfd(f_i(x),x_{i,k(j)})<3/2\cdot 2^{-i}.$ Then by Lemma \ref{correspl} we deduce that $x \in B_{k(j)}^i$. Applying now \eqref{welldeffi} we have $\sfd(f_i(w),\alpha_{i,k(j)}(\psi_{i,k(j)}(x)))\le C2^{-i}\eps_i,$ that gives $\sfd(\alpha_{i,k(j)}(\psi_{i,k(j)}(w)),x_{i+1,j})<2^{-i}/2+C\eps_i2^{-i}$. Then from \eqref{centertocenter} $\beta_{i+1,j}(\alpha_{i,k(j)}(\psi_{i,k(j)}(x)))\in B_{2^{-i}}(0),$ provided $\eps_i$ is small enough.
	We can now apply \eqref{kjinv} (we will check the needed hypothesis at the end of the proof) to infer that $$K_{i,j}^{-1}(\psi_{i,k(j)}(x))\in B_{\frac{5}{4}\cdot 2^{-i}}(0),$$
	provided $\eps_i,\eps_{i+1}$ are small enough.
	Recalling now the definition of $\phi_{i,j}$, \eqref{nextchart} follows.	Finally \eqref{wellcoveredeq} is a consequence of \eqref{nextchart} and the $2^{-i-1}$-density of $X_{i+1}.$ \\ To justify the use of \eqref{kjinv} above, we need to check that $\psi_{i,k(j)}(w) \in K_{i,j}(B_{15\cdot 2^{-i}}(0))$.  Call $y\coloneqq \beta_{i+1,j}(\alpha_{i,k(j)}(\psi_{i,k(j)}(w)))$, $y\in B_{2^{-i}}(0)$, then from \eqref{kj}  
	\begin{align*}
		|K_{i,j}(y)-\psi_{i,k(j)}(w)|&\le |\beta_{i,k(j)}(\alpha_{i+1,j}(y))-\psi_{i,k(j)}(w)|+C(\eps_i+\eps_{i+1})2^{-i}=\\
		&=|\beta_{i,k(j)}(\alpha_{i+1,j}(\beta_{i+1,j}(\alpha_{i,k(j)}(\psi_{i,k(j)}(w)))))-\psi_{i,k(j)}(x)|+C(\eps_i+\eps_{i+1})2^{-i}\\
		\text{by \eqref{alfabeta} \quad }&{\le} C(\eps_i+\eps_{i+1})2^{-i}.
	\end{align*}
	Now, since $K_{i,j}$ is an isometry and $y \in B_{2^{-i}}(0)$ , if $\eps_i,\eps_{i+1}$ are small enough, we conclude that $\psi_{i,k(j)}(w) \in K_{i,j}(B_{2\cdot 2^{-i}}(0)).$
\end{proof}

\subsubsection{Construction of the diffeomorphism \texorpdfstring{$h_i$}{hi}}\label{sec:hi build}

Before passing the the actual construction of the map $h_i$ we need a technical lemma. Roughly saying it shows the transition functions relative to the charts $\phi_{i,j}$ are close  to the transition functions $\hat I_{i+1,j_1,j_2}$ of the  manifold $W_{i+1}$.
\begin{lemma}\label{charts2}
	 Let $x \in B_{4\cdot 2^{-i}}(0)$ be in the domain of $\phi_{i,j_1}^{-1}\circ \phi_{i,j_2}$ for some $j_1,j_2 \in J_{i+1},$ then $\hat I_{i+1,j_1,j_2}$ exists and
	\begin{equation}\label{phiphi}
		|\phi_{i,j_1}^{-1}\circ \phi_{i,j_2}(x)- \hat I_{i+1,j_1,j_2}(x)|\le C(\eps_i+\eps_{i+1})2^{-i}.
	\end{equation} 
\end{lemma}
\begin{proof}
	Consider a point $x \in B_{4\cdot 2^{-i}}(0)$ and two indices $j_1,j_2\ \in J_{i+1}$ as in the hypotheses. We start by claiming that 
	\begin{equation}\label{j1j2}
		\sfd(x_{i+1,j_1},x_{i+1,j_2})\le 9\cdot 2^{-i}<29\cdot 2^{-(i+1)}. 
	\end{equation}
 In particular this already implies that $I_{i+1,j_1,j_2}\in \mathcal I_{i+1}$ (recall \eqref{eq:def Ii}) and so $\hat I_{i+1,j_1,j_2}$ exists.
	To show the claim note that  $\phi_{i,j_2}(x)=\phi_{i,j_1}(y)$ for some $y \in B_{4\cdot 2^{-i}}(0).$  Then from \eqref{welldeffi} we have
	\[ \sfd\big(\alpha_{i,k(j_2)}\circ \psi_{i,k(j_2)}\circ \phi_{i,j_2}(x),\alpha_{i,k(j_1)}\circ \psi_{i,k(j_1)}\circ \phi_{i,j_1}(y)\big)\le 2^{-i}C\eps_i,\]
	that by definition can be written as
	\[ \sfd(\alpha_{i,k(j_2)}\circ K_{i,j_2}(x)),\alpha_{i,k(j_1)}\circ K_{i,j_1}(y))\le 2^{-i}C\eps_i,\]
	Applying now \eqref{kj} (recalling also \eqref{alfabeta} as usual) we deduce
	\begin{equation}\label{aj1j2}
		\sfd(\alpha_{i+1,j_2}(x),\alpha_{i+1,j_1}(y))\le 2^{-i}C(\eps_i+\eps_{i+1}).
	\end{equation}
	Therefore from \eqref{centertocenter}, if $\eps_i,\eps_{i+1}$ are small enough, we deduce \eqref{j1j2}. We can now proceed with the proof of \eqref{phiphi}.
	From the definitions we have
	$\phi_{i,j_1}^{-1}\circ \phi_{i,j_2}(x)=K_{i,j_1}^{-1}\circ \hat I_{i,k(j_1),k(j_2)}\circ K_{i,j_2}(x).$ 
	From \eqref{finalb}  
	\[|\phi_{i,j_1}^{-1}\circ \phi_{i,j_2}(x)- K_{i,j_1}^{-1}\circ I_{i,k(j_1),k(j_2)}\circ  K_{i,j_2}(x)|\le C\eps_i2^{-i}.\]
	Moreover, since $x \in B_{4\cdot 2^{-i}}(0),$ from \eqref{kj} 
	\[|\phi_{i,j_1}^{-1}\circ \phi_{i,j_2}(x)- K_{i,j_1}^{-1}\circ I_{i,k(j_1),k(j_2)}\circ \beta_{i,k(j_2)}\circ \alpha_{i+1,j_2}(x)|\le C(\eps_i+\eps_{i+1})2^{-i}.\]
	Observe  now that, from \eqref{centertocenter} and the fact that $\sfd(x_{i+1,j_2},x_{i,k(j_2)})\le 2^{-i}$, we have that $\beta_{i,k(j_2)}(\alpha_{i+1,j_2}(x))$ belongs to $B_{10\cdot 2^{-i}}(0)$ (if $\eps_i,\eps_{i+1}$ are small enough), thus we can apply \eqref{alminverse} to obtain
	\[|\phi_{i,j_1}^{-1}\circ \phi_{i,j_2}(x)- K_{i,j_1}^{-1}\circ \beta_{i,k(j_1)}\circ \alpha_{i,k(j_2)}\circ \beta_{i,k(j_2)}\circ \alpha_{i+1,j_2}(x)|\le C(\eps_i+\eps_{i+1})2^{-i}.\]
	Recall now from  \eqref{alfabeta} that  $\alpha_{i,k(j_2)}\circ \beta_{i,k(j_2)}$ is almost the identity, therefore
	\[|\phi_{i,j_1}^{-1}\circ \phi_{i,j_2}(x)-  K_{i,j_1}^{-1}\circ \beta_{i,k(j_1)}\circ \alpha_{i+1,j_2}(x)|\le C(\eps_i+\eps_{i+1})2^{-i}.\]
	To write the above we should check that $\alpha_{i+1,j_2}(x)$ is in the domain of $\beta_{i,k(j_1)}$, this follows from the fact that $\alpha_{i+1,j_2}(x) \in B_{5\cdot 2^{-i}}(x_{i+1,j_2})\subset B_{15 \cdot 2^{-i}}(x_{i,k(j_1)})$,	where the inclusion is  a consequence of \eqref{j1j2} and the fact that $\sfd(x_{i+1,j_1},x_{i,k(j_1)})\le 2^{-i}$. We now apply \eqref{kjinv} to get
	\begin{equation}\label{phiphialm}
		|\phi_{i,j_1}^{-1}\circ \phi_{i,j_2}(x)-  \beta_{i+1,j_1}\circ \alpha_{i,k(j_1)}\circ \beta_{i,k(j_1)}\circ \alpha_{i+1,j_2}(x)|\le C(\eps_i+\eps_{i+1})2^{-i}.
	\end{equation}
	Observe that to use \eqref{kjinv} we need to check that $\beta_{i,k(j_1)}\circ \alpha_{i+1,j_2}(x)\in K_{i,j_1}(B_{15\cdot 2^{-i}}(0))$. To see this, we can use  \eqref{aj1j2} and then \eqref{kj} to deduce
	\begin{align*}
		|\beta_{i,k(j_1)}\circ \alpha_{i+1,j_2}(x)-K_{i,j_1}(y)|&\le |\beta_{i,k(j_1)}\circ \alpha_{i+1,j_1}(y)-K_{i,j_1}(y)|+ 2^{-i}C(\eps_i+\eps_{i+1})\\
		&\le 2^{-i}C(\eps_i+\eps_{i+1}).
	\end{align*}
	From this, since $K_{i,j_1}$ is a global isometry and $y \in B_{4\cdot 2^{-i}}(0)$, we have $\beta_{i,k(j_1)}\circ \alpha_{i+1,j_2}(x)\in K_{i,j_1}(B_{8\cdot 2^{-i}}(0))$, provided $\eps_i,\eps_{i+1}$ small enough, that is what we wanted.
	Using once again \eqref{alfabeta} in \eqref{phiphialm} we obtain that 
	\[|\phi_{i,j_1}^{-1}\circ \phi_{i,j_2}(x)-  \beta_{i+1,j_1}\circ\alpha_{i+1,j_2}(x)|\le C(\eps_i+\eps_{i+1})2^{-i},\]
	that makes sense since $\alpha_{i+1,j_2}(x)\in B_{14 \cdot 2^{-i}}(x_{i+1,j_1}),$ which follows from \eqref{j1j2}.
	Finally combining \eqref{alminverse} and $\eqref{finalb}$ we reach \eqref{phiphi}.
	
\end{proof}

\begin{lemma}\label{hmaps}
	There exists a map $h_i : W_i \to W_{i+1}$ that is surjective, satisfies
	\begin{equation}\label{diffeobound}
	\|Dh_i\|,  \|Dh_i^{-1}\|\le1+C(\eps_i+\eps_{i+1}).
	\end{equation}
	and so that for all $j \in J_{i+1}$ it holds
	\begin{equation}\label{localh}
		h_i|_{\phi_{i,j}(B_{2 \cdot 2^{-i}}(0))}=\psi_{i+1,j}^{-1} \circ H \circ \phi_{i,j}^{-1},
	\end{equation}
	for some diffeomorphism $H: B_{4\cdot 2^{-i}}(0) \to B_{4\cdot 2^{-i}}(0)$  (depending on $j$) such that $\|H-{\sf{id}} \|_{C^1,2^{-i}}\le C(\eps_i+\eps_{i+1}).$
\end{lemma}

\begin{proof}
	We plan to apply Theorem \ref{patching}, with $M=W_i,\overline{M}=W_{i+1},$	$\phi_j=\phi_{i,j},h_j=h_{i,j}$ (recall their definition in \eqref{eq:def hi}),$t=2\cdot 2^{-i}$ and $N=N(n)$ (given in Proposition \ref{prop:covering}). We need to check the hypotheses of the  theorem. First notice that $\phi_{i,j_1}^{-1}\circ \phi_{i,j_2}=K_{i,j_1}^{-1}\circ \hat I_{i,k(j_1),k(j_2)}\circ K_{i,j_2}$, whenever $\phi_{i,j_1}^{-1}\circ \phi_{i,j_2}$ has a non-empty domain of definition. Hence from \eqref{finalb} and the fact that $K_{i,j_1}$ are isometries we deduce that  $\|D(\phi_{i,j_1}^{-1}\circ \phi_{i,j_2})\|\le 1+C\eps_i$ and $|\partial_{ij}(\phi_{i,j_1}^{-1}\circ \phi_{i,j_2})_k|\le C\eps_i/2^{-i}$. Hence  assumption \eqref{patch1} is satisfied. Moreover the required partition of the indices $j\in J_{i+1}$ is naturally induced from the partition $J_{i+1,1},...,J_{i+1,N_{i+1}}$. Indeed if $\phi_{i,j_1}(B_{4\cdot 2^{-i}}(0))\cap \phi_{i,j_2}(B_{4\cdot 2^{-i}}(0))\neq \emptyset $, thanks to Lemma \ref{charts2}, we have that $B_{j_1}^{i+1}\cap B_{j_2}^{i+1}\neq \emptyset$, therefore the transition function $\hat I_{i+1,j_1,j_2}$ exists, and therefore by Remark \ref{QI} $j_1,j_2$ belong to different sets of the partition. Moreover the required condition that
	\begin{equation}\label{covered}
		W_{i}^n\subset \bigcup_{j\in J^{i+1}} \phi_{i,j}(B_{2\cdot 2^{-i}}(0)).
	\end{equation} 
	follows from Lemma \ref{wellcovered}.
	
	Finally we need to prove \eqref{patch2} and \eqref{patch3}. It suffices to prove the following. Define $B_2=B_{2(2-2/N) 2^{-i}}(0)$ and $B_1=B_{2(2-1/N) 2^{-i}}(0)$, then for every $j_1,j_2$  it holds that
	\begin{equation}\label{inside}
		h_{i,j_1}(\phi_{i,j_1}(B_1)\cap\phi_{i,j_2}(B_1))\subset h_{i,j_2}(\phi_{i,j_2}(B_{4 \cdot2^{-i}}(0))
	\end{equation}
	and moreover for such $j_1,j_2$ 
	\begin{equation}\label{close}
		\|\phi_{i,j_2}^{-1}\circ h_{i,j_2}^{-1}\circ h_{i,j_1}\circ \phi_{i,j_2}-{\sf{id}}\|_{C^1,2^{-i}}\le C\eps_i, \quad \text{ on   $\phi_{i,j_2}^{-1}( \phi_{i,j_1}(B_1)\cap\phi_{i,j_2}(B_1))$.	}
	\end{equation}
 If  $ \phi_{i,j_1}(B_2)\cap\phi_{i,j_2}(B_2) =\emptyset $ there is nothing to prove, hence we assume otherwise.
 We start with \eqref{inside}.	
	Suppose $$\phi_{i,j_1}(B_2)\cap\phi_{i,j_2}(B_2)\neq \emptyset$$ and pick $x\in \phi_{i,j_1}(B_1)\cap\phi_{i,j_2}(B_1)$.  Then $x=\phi_{i,j_2}(y)$ for some $y \in B_1$. Thus from \eqref{phiphi}
	\[ |\phi_{i,j_1}^{-1}\circ \phi_{i,j_2}(y)-\hat I_{i+1,j_1,j_2}(y)|\le C(\eps_i+\eps_{i+1})2^{-i}.\]
	Then from the above and since $\psi_{i+1,j_1}^{-1}$ that is  $2$-Lipschitz (if $\eps_{i+1}$ is small enough), we have 
	\[ \sfd_{i+1}(h_{j_1}(x),\psi_{i+1,j_2}^{-1}(y))=\sfd_{i+1}(\psi_{i+1,j_1}^{-1}\circ \phi_{i,j_1}^{-1}\circ \phi_{i,j_2}(y),\psi_{i+1,j_1}^{-1}(\hat I_{i+1,j_1,j_2}(y)))\le C(\eps_i+\eps_{i+1})2^{-i},\]
 where we used both the definition of $h_{j_1}$ and \eqref{eq:correct transitions} in the first identity.
	Therefore, since $y \in B_{2(2-1/N)\cdot 2^{-i}}(0)$, from Lemma \ref{ballsinballs}, assuming $\eps_i,\eps_{i+1}$ are small enough, we have 
	$$h_{j_1}(x)\in \psi_{i+1,j_2}^{-1}(B_{4	\cdot 2^{-i}}(0))=h_{i,j_2}(\phi_{i,j_2}(B_{4	\cdot 2^{-i}}(0)),$$
	that proves \eqref{inside}. It remains only to prove \eqref{close}.
	Notice that from the definitions
	$$\phi_{i,j_2}^{-1}\circ h_{i,j_2}^{-1}\circ h_{i,j_1}\circ \phi_{i,j_2}=\hat I_{i+1,j_2,j_1}\circ \phi_{i,j_1}^{-1}\circ \phi_{i,j_2}.$$ Therefore from \eqref{phiphi} 
	\[|\phi_{i,j_2}^{-1}\circ h_{i,j_2}^{-1}\circ h_{i,j_1}\circ \phi_{i,j_2}-{\sf{id}}|\le C(\eps_i+\eps_{i+1})2^{-i},
 \text{\quad on $\phi_{i,j_2}^{-1}( \phi_{i,j_1}(B_1)\cap\phi_{i,j_2}(B_1))$.}
 \]
	We need now the bound on the first derivatives. To this aim notice that we can also write 
	$$\phi_{i,j_2}^{-1}\circ h_{i,j_2}^{-1}\circ h_{i,j_1}\circ \phi_{i,j_2}=\hat I_{i+1,j_1,j_2}\circ K_{j_1}^{-1}\circ \hat I_{i,k(j_1),k(j_2)}\circ K_{j_2}$$
	and from what we just proved combined with \eqref{finalb} we deduce
	$$|I_{i+1,j_1,j_2}\circ K_{j_1}^{-1}\circ  I_{i,k(j_1),k(j_2)}\circ K_{j_2}-{\sf{id}}|\le C(\eps_i+\eps_{i+1})2^{-i} 
  \text{\quad 	on $\phi_{i,j_2}^{-1}( \phi_{i,j_1}(B_1)\cap\phi_{i,j_2}(B_1)) $.}
 $$
 Exploiting Lemma \ref{ballsinballs} (and the fact   that every map $\phi_{i,j}$ is the composition of $\psi_{i,k(j)}^{-1}$ with an isometry) we can observe that, since $ \phi_{i,j_1}(B_2)\cap\phi_{i,j_2}(B_2) \neq \emptyset $, the (bigger) set $ \phi_{i,j_1}(B_1)\cap\phi_{i,j_2}(B_1)$ contains a metric ball in $W_{i}$ of radius $\frac{2^{-i}}{100N}$, provided $\eps_{i}$ is small enough. Then, again by Lemma \ref{ballsinballs} we deduce that $\phi_{i,j_2}^{-1}( \phi_{i,j_1}(B_1)\cap\phi_{i,j_2}(B_1)) $ contains an Euclidean ball of radius $\frac{2^{-i}}{200N}$, again if $\eps_{i+1}$ is small enough. Therefore, since the above map is an isometry, from Lemma \ref{isombound} we deduce 
	$$\|D(I_{i+1,j_1,j_2}\circ K_{j_1}^{-1}\circ  I_{i,k(j_1),k(j_2)}\circ K_{j_2})-I_n\|\le C(\eps_i+\eps_{i+1}) 
 \text{\quad on $\phi_{i,j_2}^{-1}( \phi_{i,j_1}(B_1)\cap\phi_{i,j_2}(B_1)) $,}
 $$
	 that implies using \eqref{finalb} and Lemma \ref{iterateclosest2}
	$$\|D(\hat I_{i+1,j_1,j_2}\circ K_{j_1}^{-1}\circ  \hat I_{i,k(j_1),k(j_2)}\circ K_{j_2})-I_n\|\le C(\eps_i+\eps_{i+1})
  \text{\quad on $\phi_{i,j_2}^{-1}( \phi_{i,j_1}(B_1)\cap\phi_{i,j_2}(B_1)) $. }
 $$
	This completes the proof of \eqref{close}.
	Therefore if $\eps_i,\eps_{i+1}$ are small enough we can apply Theorem \ref{patching} and get a map $h_i : W_i \to W_{i+1}$ . Thanks to \eqref{g} this map has the property that for every $j \in J_{i+1}$ it holds 
	\begin{equation}
		h_i|_{\phi_{i,j}(B_{2 \cdot 2^{-i}}(0))}=\psi_{i+1,j}^{-1} \circ  \phi_{i,j}^{-1} \circ \phi_{i,j}\circ H \circ \phi_{i,j}^{-1} =\psi_{i+1,j}^{-1} \circ H \circ \phi_{i,j}^{-1},
	\end{equation}
	for some diffeomorphism $H: B_{2\cdot 2^{-i}}(0) \to B_{2\cdot 2^{-i}}(0)$  (depending on $j$) with $\|H-{\sf{id}} \|_{C^1,2^{-i}}\le C(\eps_i+\eps_{i+1}).$ Thus \eqref{localh} is proved. Moreover  from this and \eqref{lippsi} we also get $(1+C(\eps_i+\eps_{i+1}))^{-1}\le \|Dh_i\|\le1+C(\eps_i+\eps_{i+1})$. 
	Finally we observe that $h_i(W_i)=W_{i+1}$ for every $i$. Indeed $h_i|_{\phi_{i,j}(B_{2 \cdot 2^{-i}}(0))}=\psi_{i+1,j}^{-1} \circ H \circ \phi_{i,j}^{-1}$, moreover, since $\|H-{\sf{id}} \|_{C^1,t}\le C(\eps_i+\eps_{i+1}),$ if $\eps_i,\eps_{i+1}$ are small enough we have $B_{\frac{3}{2} 2^{-i}}(0)\subset H(B_{2 \cdot 2^{-i}}(0))$. Thus $h_i(\phi_{i,j}(B_{2 \cdot 2^{-i}}(0)))=\psi_{i+1,j}^{-1} \circ H(B_{2 \cdot 2^{-i}}(0)) \supset \psi_{i+1,j}^{-1}(B_{ \frac{3}{2}2^{-i}}(0)).$ Then we conclude by Lemma \ref{completeness}.
\end{proof}

\begin{lemma}\label{point4}
	It holds that 
	$$\sfd(f_{i+1}(h_i(w)),f_i(w))\le C 2^{-i}(\eps_i+\eps_{i+1}), \quad \forall \, w \in W_i.$$ 
\end{lemma}
\begin{proof}
	
	We need to estimate $\sfd(f_{i+1}(h_i(w)),f_i(w))$. By \eqref{covered}, there exists an index $j \in J_{i+1}$ such that $w=\phi_{i,j}(x)$ for some  $ x \in B_{2\cdot 2^{-i}}(0).$ Since $\phi_{i,j}=\psi_{i,k(j)}\circ K_{i,j}^{-1}$ we deduce that $w \in B_{k(j)}^i$ and applying \eqref{welldeffi} we obtain
	\begin{equation}\label{k(j)}
		\sfd(f_i(w),\alpha_{i,k(j)}\circ \psi_{i,k(j)}(w))\le C\eps_i 2^{-i}.
	\end{equation}
	We claim now that $h_i(w)\in B_j^{i+1}$. Indeed, since $w \in \phi_{i,j}(B_{2 \cdot 2^{-i}}(0))$  from \eqref{localh} and recalling the Lispschitzianity of $\psi_{i+1,j}^{-1}$, we obtain 
	\begin{equation}\label{hij}
		\begin{split}
			\sfd_{i+1}(h_i(w),h_{i,j}(w))&=\sfd_{i+1}(h_i(w),\psi^{-1}_{i+1,j}\circ \phi_{i,j}^{-1}(w))\\
			&\le(1+C\eps_{i+1}) |H( \phi_{i,j}^{-1}(w))-\phi_{i,j}^{-1}(w)|\le  C(\eps_i+\eps_{i+1}) 2^{-i},
		\end{split}
	\end{equation}
	provided $\eps_{i+1}$ is small enough.
	Now  by  definition  $h_{i,j}(w)=\psi_{i+1,j}^{-1}(x)\in \psi_{i+1,j}^{-1}(B_{2 \cdot 2^{-i}}(0)),$ therefore from \eqref{hij} and Lemma \ref{ballsinballs}, if $\eps_i,\eps_{i+1}$ are small enough, we easily deduce that $h_i(w)\in \psi_{i+1,j}^{-1}(B_{4 \cdot 2^{-i}}(0))=B_j^{i+1}$. The claim is proved. In particular we have $h_i(w),h_{i,j}(w)\in B_j^{i+1}$ and combining \eqref{fid} with \eqref{hij} we find that 
	\begin{equation}\label{fhij}
		\sfd(f_{i+1}(h_i(w)),f_{i+1}(h_{i,j}(w)))\le C 2^{-i}(\eps_i+\eps_{i+1}).
	\end{equation}
	Moreover we shall also apply \eqref{welldeffi} and deduce
	\begin{equation}\label{j}
		\sfd(f_{i+1}(h_{i,j}(w)),\alpha_{i+1,j}(x))=\sfd(f_{i+1}(h_{i,j}(w)),\alpha_{i+1,j}\circ \psi_{i+1,j}(h_{i,j}(w)))\le C\eps_{i+1} 2^{-i-1}.
	\end{equation}
	Therefore putting \eqref{j}, \eqref{fhij} and \eqref{k(j)} together we obtain
	\[\sfd(f_{i+1}(h_i(w)),f_i(w))\le C2^{-i}(\eps_{i+1}  +\eps_i )+ \sfd(\alpha_{i,k(j)}\circ \psi_{i,k(j)}(w),\alpha_{i+1,j}(x)). \]
	However $w=\phi_{i,j}(x)=\psi_{i,k(j)}^{-1}(K_{i,j}(x))$, therefore recalling \eqref{kj}  and \eqref{alminverse}
	\begin{align*}
		\sfd(f_{i+1}(h_i(w)),f_i(w))&\le C2^{-i}(\eps_{i+1}  +\eps_i )+ \sfd(\alpha_{i,k(j)}(K_{i,j}(x)),\alpha_{i+1,j}(x))\\
		&\le C2^{-i}(\eps_{i+1}  +\eps_i )+ \sfd(\alpha_{i,k(j)}(\beta_{i,k(j)}(\alpha_{i+1,j}(x))),\alpha_{i+1,j}(x))\\
		&\le C2^{-i}(\eps_{i+1}  +\eps_i ).
	\end{align*}
\end{proof}

We now turn to the proof of the main result of this section.
\begin{proof}[Proof of Lemma \ref{mainh}]
	We already know that $h_i$ is surjective and the estimate in \eqref{diffeobound},  hence we only need to prove that $h_i$ is injective. Suppose by contradiction that exist two distinct points $w_1,w_2 \in W_i$ such that $h_i(w_1)=h_i(w_2).$ First observe that from \eqref{localh} we have that $h_i|_{\phi_{i,j}(B_{2 \cdot 2^{-i}}(0))}$ is injective for every $j \in J_{i+1}.$ Therefore we cannot have $w_1,w_2 \in \phi_{i,j}(B_{2 \cdot 2^{-i}}(0))$ for any $j \in J_{i+1}.$ We claim that this implies
	\begin{equation}\label{distantw}
		\sfd_i(w_1,w_2)>\frac{1}{4}2^{-i}.
	\end{equation}
	Indeed suppose the contrary. From  \eqref{wellcoveredeq} we have that $w_1 \in \phi_{i,j}(B_{\frac{5}{4} \cdot 2^{-i}}(0))$ for some $j \in J_{i+1}.$ From definition we have $\phi_{i,j}=\psi_{i,k(j)}^{-1} \circ K_{i,j}$, where $K_{i,j}$ is an isometry, therefore thanks to Lemma \ref{ballsinballs}  we deduce that
	\[\phi_{i,j}(B_{\frac{5}{4} \cdot 2^{-i}}(0))\subset B^{W_i}_{\frac{3}{2} \cdot 2^{-i}}(\phi_{i,j}(0)). \]
	In particular $w_2 \in B^{W_i}_{\frac{7}{4} \cdot 2^{-i}}(\phi_{i,j}(0))$, however again from Lemma \ref{ballsinballs} we have
	\[B^{W_i}_{\frac{7}{4} \cdot 2^{-i}}(\phi_{i,j}(0))\subset \phi_{i,j}(B_{2 \cdot 2^{-i}}(0)),\]
	but this is a contradiction and \eqref{distantw} is proved. Observe now that, since $h_i(w_1)=h_i(w_2)$, from Lemma \ref{point4}, provided $\eps_i,\eps_{i+1}$ small enough, we  have  
	\begin{equation}\label{closeinz}
		\sfd(f_i(w_1),f_i(w_2))\le \frac{1}{8}2^{-i}.
	\end{equation} This, together with the $2^{-i}$-density of $X_i$ implies that $f_i(w_1),f_i(w_2) \in B_{2\cdot 2^{-i}}(x_{j})$ for some $j \in J_i$. Therefore  Lemma \ref{correspl} gives that $w_1,w_2 \in B_{j_1}^i$, hence we are in position to apply \eqref{fid}, that coupled with \eqref{closeinz} provides
	\[\sfd_i(w_1,w_2)\le  \frac{1}{8}2^{-i}+C\eps_i2^{-i},\]
	which, if $\eps_i$ is sufficiently small, contradicts \eqref{distantw}. 
\end{proof}

\subsection{Final verifications and conclusion}
To prove Theorem \ref{mainthm} it remains to show that all the objects that we built on the previous sections satisfy the requirements \ref{a)}, \ref{b)}, \ref{c)}, \ref{d)} and \ref{e)} stated in Section \ref{sec:main argument}. Almost all the verifications are straightforward consequences of the results already obtained.

\noindent \ref{a)}- It is the content of Lemma \ref{rhod}.\\
\ref{c)}- If $\sfd_i(w_1,w_2)>2\cdot 2^{-i}$ we are done by \eqref{rho}. Suppose  $\sfd_i(w_1,w_2)\le2\cdot 2^{-i}$. Then observe that from \eqref{closedcover}  $w_1 \in B_{3\cdot 2^{-i}}(\psi_{i,j}^{-i}(0))$ for some $j$ and thus $w_1,w_2 \in B_{6\cdot 2^{-i}}(\psi_{i,j}^{-i}(0))\subset B_j^i $ from Lemma \ref{ballsinballs}. Then we can apply  \eqref{fid} and conclude.\\
\ref{d)}- This is \ref{point4eq}.\\
\ref{b)}- Suppose first that $\rho_{i+1}(h_i(w_1),h_i(w_2))\le 2	\cdot 2^{-(i+1)}$ and $\rho_i(w_1,w_2)\le 2	\cdot 2^{-i}$, and observe that by Lemma \ref{rhod}  $\rho_{i+1}(h_i(w_1),h_i(w_2))=\sfd_{i+1}(h_{i}(w_1),h_i(w_2))$ and $\rho_i(w_1,w_2)=d_{i}(w_1,w_2)$. Thus we conclude by the fact that $h_i$ is $1+C(\eps_i+\eps_{i+1}) $ bi-Lipschitz, which comes from Lemma \ref{mainh}. Suppose now that  $\rho_{i+1}(h_i(w_1),h_i(w_2))>2	\cdot 2^{-(i+1)}$. By \ref{c)} and \ref{d)} using triangle inequality we get 
\[ |\rho_{i+1}(h_i(w_1),h_i(w_2))-\rho_i(w_1,w_2)|\le C2^{-i}(\eps_i+\eps_{i+1}).\]
Then dividing by $\rho_{i+1}(h_i(w_1),h_i(w_2))$ follows that
\[ \left |1-\frac{\rho_i(w_1,w_2)}{\rho_{i+1}(h_i(w_1),h_i(w_2))}\right|\le C(\eps_i+\eps_{i+1})\]
that is what we wanted. The case $\rho_i(w_1,w_2)>2^{-i}$ is analogous.\\
\ref{e)}- Take any $x_{i,j}\in X_i$ and take any $w \in B_j^i$, then from \eqref{centertocenter} $\alpha_{i,j}(\psi_{i,j}(w)) \in B_{9\cdot 2^{-i}}(x_{i,j}),$ if $\eps_i$ is sufficiently small. Moreover from \eqref{welldeffi} we have $\sfd(f_i(w),\alpha_{i,j}(\psi_{i,j}(w)))	\le C2^{-i}\eps_i$, therefore, if $\eps_i$ is small enough we conclude from the  $2^{-i}$-density of $ X_i$.

\subsection{Adjustments for the local version}\label{sec:local proof}

The proof of the local version of Reifenberg theorem for metric spaces (see Theorem \ref{localthm})  is almost the same as the proof of global one (Theorem \ref{mainthm}). However some technical difficulties arise which make the argument slightly more complicated. The main point is that we cannot cover in general the ball $B_{1}(z_0)$ with a countable number of balls of a fixed radius contained in it (this can be seen already if $Z=\rr^n$ with $n>1$). Therefore we fix a big constant $M$   and cover only the smaller ball $B_{1-M\eps_0}(z_0)$. Then, as in the global case, for every scale $2^{-i}$ we cover $B_{1-M\eps_0}(z_0)$ with balls of radius $2^{-i}$ and use these covering to build manifolds $W_i$ (only that we need to start from scale $2^{-i_0}\ll M\eps_0$, instead of $i=0$).
As expected, the main issue arises  close to the boundary of the ball $B_{1}(z_0)$. For the sake of the simplicity consider $Z=\rr^n$ and suppose we have covered the $B_{1-M\eps_0}(0)$ with small balls of radius $2^{-i}$. Then the union of these balls is already a good candidate for approximating manifold $W_i$. However, close to the boundary, $W_i$ looks very rough and irregular and the induced metric does not contain any  information about the underlying Euclidean distance.  For this reason most of the construction will happen in some subset of $W_i$ that is a bit far from the boundary. This difficulty mainly reflects on construction of the maps $h_i : W_i \to W_{i+1}$. Indeed the two manifolds $W_i,W_{i+1}$, near the boundary may look very different and in particular we have no reason to hope that $h_i$ is a surjective diffeomorphism. For this reason, $h_i$ can be only defined in a smaller subset $U_i$ of $W_i$.

We will only write down the objects that need to be built in order to perform the main argument as in Section \ref{sec:main argument} for the global version.

\begin{center}
	\underline{INITIAL ASSUMPTIONS} \\
\end{center}

As in the statement of Theorem \ref{localthm} we assume that 
\begin{equation}\label{mainass}
\eps_i(n)\le \eps(n) \text{ for every } i \ge 0,
\end{equation}
where $\eps_i(n)$ are the numbers defined in \eqref{ghjonesloc} relative to  the ball $B_1(z_0)\subset Z$ and $\eps(n)$ is a sufficiently small constant depending only on $n.$ For brevity we will write $\eps_i=\eps_i(n).$ We also fix positive constants $C=C(n)$, $M=M(n)$ depending only on $n$. Moreover $\eps(n)$ is assumed to be small enough so that $\eps(n)M(n)\ll 1.$

\begin{center}
	\underline{MAIN ARGUMENT} \\
\end{center}

 For brevity in the sequel we write $\bar \eps_0\coloneqq M\eps_0.$

\textbf{CLAIM:}	It is sufficient to construct  a sequence of Riemannian manifolds $\{W_i,d_i\}_{i\ge 0}$ (not necessarily connected or complete - note however that condition $B_{1-2\bar \eps_0}(0)\subset U_0$ below  ensures that image of the homeomorphism  covers  $B_{1-M\eps_0}(z_0)$), open sets $U_i \subset W_i$, symmetric maps $\rho_i : W_i \times W_i \to \mathbb{R}_+$,  maps $h_i : U_i \to U_{i+1}$ with $U_{i+1}=h_i(U_i)$ and $f_i: W_i \to Z$,  satisfying the following statements

\begin{enumerate}[label={\alph**)}]
	\item \label{a2)} $(W_0,d_0)=(B_1^{\rr^n}(0),d_{Eucl})$, $\rho_0=d_{Eucl}$ and $B_{1-2\bar \eps_0}(0)\subset U_0,$ 
	\item \label{b2)}$\sfd(f_0(0),z_0)\le \frac{\bar \eps_0}{100}$,
	\item \label{c2)}for every $i \ge 0$
	$$\frac{1}{1+C(\eps_i+\eps_{i+1})}\le\frac{\rho_{i+1}(h_i(w_1),h_i(w_2))}{\rho_i(w_1,w_2)}\le 1+C(\eps_i+\eps_{i+1}),$$
	for every $w_1,w_2 \in U_i$	and
	$$c_1(n)\le\frac{\rho_{1}(h_0(w_1),h_0(w_2))}{\rho_0(w_1,w_2)}\le c_2(n),$$
	for every $w_1,w_2 \in U_0,$ for some positive constants $c_1(n),c_2(n)$ depending only on $n$,
	\item \label{d2)} for $i \ge 1$ 
	$$|d(f_i(w_1),f_i(w_2))-\rho_i(w_1,w_2)|\le C\eps_i2^{-i}\bar \eps_0 ,$$ for every $w_1,w_2 \in U_i$
	and
	$$|d(f_0(x),f_0(y))-|x-y|||\le \frac{\bar \eps_0}{100},$$ for every $x,y \in B_1(0)$,
	\item \label{e2)} for $i \ge 1,$ 
	$$\sfd(f_{i+1}(h_i(w)),f_i(w))\le C(\eps_i+\eps_{i+1})2^{-i} \bar \eps_0 ,$$
	for every $w \in U_i$  and
	$$\sfd(f_{1}(h_0(w)),f_0(w))\le \frac{\bar \eps_0}{100},$$ for every $w \in U_0,$ 
	\item \label{f2)} $f_i(h_i(U_i))$ is $ 2 \cdot 2^{-i}\bar \eps_0 $-dense in $B_{1-2\bar \eps_0}(z_0)$ for every $i \ge 1$.

\end{enumerate}

		{\textbf {Proof of the \textbf{CLAIM}}}:
Define $F_i : U_0 \to Z$ as $F_0\coloneqq f_0$ and $F_i\coloneqq f_i \circ h_{i-1}\circ...\circ h_0$ for every $i \ge 1$,  that is  well defined since $U_{i+1}=h_i(U_i).$  Then from \ref{e2)}
\begin{equation}\label{Fbound2}
d(F_{i+1}(w),F_i(w))=d(f_{i+1}(h_i(h_{i-1}...)),f_i(h_{i-1}...))\le 2\eps(n)C2^{-i}\bar \eps_0,
\end{equation}
for $i \ge 1$ and
\begin{equation}\label{Fbound22}
d(F_{1}(w),F_0(w))=d(f_{1}(h_0(w),f_0(w))\le \frac{\bar \eps_0}{100},
\end{equation}
Hence the sequence $F_i(w)$ is Cauchy in $Z$ and we call $F(w)$ its limit. Notice that from \eqref{Fbound2} we obtain for $i \ge 1$
\begin{equation}\label{fif2}
d(F(w),F_i(w))\le  4\eps(n)C2^{-i}\bar \eps_0.
\end{equation} 
Combining the above with \eqref{Fbound22} it follows that
\begin{equation}\label{fif3}
d(F(w),F_0(w))\le  \bar \eps_0 (4C\eps(n)+100^{-1}).
\end{equation} 
Differently from Theorem \ref{mainthm}, this $F$ is not yet the function required in Theorem \ref{localthm}, but needs to be modified, indeed is not defined in the whole $B_1(0)$. We start proving that $F: U_0 \to B_1(z_0)$ is biH\"older (or biLipschitz) arguing as in the proof of Theorem \ref{mainthm}.  Consider any $w_1,w_2 \in U_0$  and set
\[s_0=\rho_0(w_1,w_2)=|w_1-w_2|\le 2.\]
Define also for every $i$ the points $w_1^i,w_2^i \in W_i$ and the numbers $s_i$ as  $w_1^0\coloneqq w_1, w_1^0\coloneqq w_2$
\[w_1^{i+1}=h_{i}(w_1^i),\,\, w_2^{i+1}=h_{i}(w_2^i),\]
and
\[ s_i\coloneqq \rho_i(w_1^i,w_2^i).\]
Observe that from \ref{c2)} 
\begin{equation}
\frac{1}{1+C(\eps_i+\eps_{i+1})}s_i\le s_{i+1}\le s_{i}(1+C(\eps_i+\eps_{i+1})), \quad \text{for every $i \ge 1$.  }
\end{equation}

\noindent{\sc $F$ is BiH\"older:}
Combining \ref{d2)} and \ref{e2)} we get
\begin{equation}\label{eq:si si+1 2}
   s_i- 6C\eps(n)2^{-i}\le  s_{i+1}\le s_i+ 6C\eps(n)2^{-i}, \quad \text{for all $i\ge 1.$}
\end{equation}
In particular $s_\infty\coloneqq \sfd(F(w_1),F(w_2))= \lim_{m\to +\infty} s_m$.
Let $j\in \nn$ be arbitrary. Applying repeatedly \eqref{eq:silip} for all $s_0,\dots,s_j$ and then \eqref{eq:si si+1 2} for the remaining $s_i$'s with $i>j$, we obtain
\begin{equation}\label{eq:holder trick2}
   2^{-\alpha j} s_0 - 6C \eps(n)2^{-j}\le  s_\infty\le 2^{\alpha j} s_0 + 6C \eps(n)2^{-j},
\end{equation}
where we set $\alpha\coloneqq \log_2(1+2C \eps(n))>0.$ Note that $\alpha\in(0,1)$ provided $\eps(n)<C/2.$ 
Then there exists a (maximal) $j\in\nn$ so that 
\begin{equation} \label{eq:first j2}
    2^{\alpha j}s_0/2\le 2^{-j}, \quad 2^{\alpha(j+1)}s_0/2\ge 2^{-(j+1)}.
\end{equation}
From the second in \eqref{eq:first j2} we get
\begin{equation}\label{eq:s0 potenziato2}
    2s_0^\frac{1}{1+\alpha}\ge 2^{-j}.
\end{equation}
Plugging \eqref{eq:first j2}  in the right hand side of \eqref{eq:holder trick2} and then plugging \eqref{eq:s0 potenziato2} we obtain
\begin{equation}\label{eq:holder12}
     s_\infty\le 2^{\alpha j} s_0 + 6C \eps(n)2^{-j}\le 2(1+6C \eps(n))2^{-j} \le 4(1+6C \eps(n))s_0^\frac{1}{1+\alpha}
\end{equation}
Similarly, since $\alpha<1,$ we can find $j\in \nn$ such that
\begin{equation}\label{eq:laltro2}
    2^{-\alpha j}s_0/2\ge 2^{-j} \quad \text{ and } \quad 2^{-\alpha (j-1)}s_0/2\le 2^{-(j-1)}.
\end{equation}
From the second in \eqref{eq:laltro2}  we get $(s_0/2)^\frac{1}{1-\alpha}\le 2^{-j+1}$.
Plugging this and the first in \eqref{eq:laltro2} into \eqref{eq:holder trick2} we get
\begin{equation}\label{eq:holder2 2}
     s_\infty\ge  2^{-j} - 6C \eps(n)2^{-j}\ge (1-6C \eps(n)))\frac{(s_0/2)^\frac{1}{1-\alpha}}2,
\end{equation}
provided $\eps(n)<C/6.$
Combining \eqref{eq:holder12} and \eqref{eq:holder2 2}  we get the desired byH\"older condition for $F.$

\noindent{\sc $F$ is BiLipschitz (assuming \eqref{sumscales2}):}
Iterating the above and using \ref{c2)} in the case $i=0$ we get
\begin{equation}	\label{lip2}
 c_1(n)\prod_{j=1}^i\frac{1}{1+C(\eps_j+\eps_{j+1})}\le \frac{s_i}{s_0}\le  c_2(n)\prod_{j=1}^i(1+C(\eps_j+\eps_{j+1})).
\end{equation} 
Thanks to \eqref{sumscales2} we have that $\prod_{j=1}^{+\infty}\frac{1}{1+C(\eps_j+\eps_{j+1})}>0$ and $\prod_{j=1}^{+\infty}(1+C(\eps_j+\eps_{j+1}))<+\infty.$ 
Moreover applying \ref{d2)} we can estimate
\begin{equation}\label{rofi2}
|d(F_i(w_1),F_i(w_2))-s_i|=|d(f_i(h_{i-1}(w_1^{i-1})),f_i(h_{i-1}(w_2^{i-1})))-\rho_i(h_{i-1}(w_1^{i-1}),h_{i-1}(w_2^{i-1}))|\le C \eps(n)2^{-i}\bar \eps_0
\end{equation}
for every $i\ge 1$. This implies that $s_i \to \sfd(F(w_1,F(w_2)))$ as $i \to +\infty.$ Therefore passing to the limit in \eqref{lip2} we obtain that 
\[M_1\le \frac{d(F(w_1),F(w_2))}{|w_1-w_2|}\le M_2,\]
for some positive constants $M_1,M_2$.
Therefore $F$ is bi-Lipschitz.
\medskip

We prove now that $F(U_0)$ is dense in $B_{1-3\bar \eps_0}(0)$. Consider $z \in B_{1-3\bar \eps_0}(0)$ and pick any $\delta>0.$ 	
Take now $i$ such that $2\cdot 2^{-i}\bar \eps_0\le \delta/8$, then by \ref{e2)} and the fact that $U_{k+1}=h_k(U_k)$, there exists $w \in W_0^n$ such that $\sfd(F_i(w),z)< \delta/2$ and moreover from \eqref{fif2} $\sfd(F(w),F_i(w))<  \delta/2$, provided $\eps(n)$ is small enough. Hence $\sfd(F(w),z)<  \delta$, thus $F(U_0)$ is dense in $B_{1-3\bar \eps_0}(0).$ We now show that
\begin{equation}\label{triangle}
B_{1-5\bar \eps_0}(z_0)\subset F(B_{1-2\bar \eps_0}(0))
\end{equation}
We first claim that if $F(w)=z \in B_{1-4\bar \eps_0}(z_0)$, then $w \in B_{1-3\bar \eps_0}(0).$ Indeed \eqref{fif3} gives
\[ \sfd(z,f_0(w))\le  \bar \eps_0 (4C\eps(n)+100^{-1}),\]
therefore $f_0(w) \in B_{1-\frac{7}{2}\bar \eps_0}(z_0)$, provided $\eps(n)$ small enough. Thus, by \ref{b2)} and \ref{d2)} for  $i=0$, we have that $w \in  B_{1-3\bar \eps_0}(0)$. The claim is proved. Pick now $z \in B_{1-5\bar \eps_0}(z_0)$, recalling  that $F(U_0)$ is dense in $B_{1-3\bar \eps_0}(0)$ there exists a sequence $w_k \in U_0$ such that $F(w_k)\to z.$ Since $F(w_k) \in B_{1-4\bar \eps_0}(z_0)$ for $k$ big enough, we deduce from the above claim that $w_k \in B_{1-3\bar \eps_0}(0)$ for $k$ big enough. Moreover, since $F^{-1}$ is Lipschitz, the sequence $w_k$ is Cauchy, hence it converges to a point $w \in B_{1-2\bar \eps_0}(0)$ that is contained in $U_0$ from \ref{a2)}. Therefore by the continuity of $F$ follows that $F(w)=z.$ This proves \eqref{triangle}. 
 We define now the map  $\tilde F : B_1(0)\to B_1(z_0)$ as
\[\tilde F(x)\coloneqq  F|_{B_{1-2\bar \eps_0}(0))}((1-2\bar \eps_0)x)\]
that have all the properties required by Theorem \ref{localthm}.\\

\section{Proof of  ``Gromov-Hausdorff close  and Reifenberg flat metric spaces are homeomorphic''}

Here we prove Corollary \ref{corollary}, stated in Section \ref{sec:cor}.
\begin{proof}[Proof of Corollary \ref{corollary}]
Suppose that $(Z_1,\sfd_1)$ and $(Z_2,\sfd_2)$ are metric spaces satisfying the assumption $i)$ of Theorem \ref{mainthm} at scale $r>0$ and with parameters  $\eps_i,\tilde \eps_i\le\eps(n)$ respectively, where $\eps(n)$ is the constant given by $i)$ in Theorem \ref{mainthm}.  Without loss of generality we can assume that $\eps_0=\tilde \eps_0=\eps(n).$ We can also assume (as in the proof of Theorem \ref{mainthm}) that $r=200$. Suppose also that for some $\delta\le\eps(n)$ small enough it holds
\[
\sfd_{GH}((Z_1,\sfd_1),(Z_2,\sfd_2))\le \delta/3.
\]
In particular by Theorem \ref{ghchar} and Proposition \ref{ghappr} we can find two $\delta$-GH-approximations $\Phi_*: Z_2\to Z_1$ and $\Phi^*:Z_1 \to Z_2$ such that
\begin{equation}\label{eq:phi*}
     \sfd_2(\Phi^*\circ \Phi_*(z),z))<\delta, \quad \text{for all $z\in Z$,} \quad   \sfd_1(\Phi_*\circ \Phi^*(y),y))<\delta, \quad \text{for all $y\in Z_2$.}
\end{equation}
We need to show that we can take the same Riemannian manifold $W$ in the conclusion $i)$ of Theorem \ref{mainthm} for both $Z_!$ and $Z_2$. 
Recall that  $W$ (in the case $i)$) can be taken to be $W_0,$ where $\{W_i\}_{i\in \nn}$ are the Riemannian manifolds built in Section \ref{sec:manifolds} (this can be seen immediately in the main argument described in Section \ref{sec:main argument}).

The key point is that $W_0$ is built starting solely from the following objects:
\begin{enumerate}[label=\roman*)]
    \item a countable set of indices $J= J_0$,
    \item a partition of $J_0$ into disjoint subsets $\{J_0^k\}_{k=1}^{N_0}$,
    \item a collection of isometries $\mathcal I_0=\{ I_{j_1j_2}\}_{(j_1,j_2)\in \mathcal A}$, for some set $\mathcal{A}\subset J_0\times J_0$. 
\end{enumerate}
Indeed these are the only elements used in Lemma \ref{bigmodification2} to build the family of maps $\{\hat I_{j_1j_2}\}_{(j_1,j_2)\in \mathcal A}$, which are in turn the only ingredients in the actual construction of the manifold $W_0$ done in Section \ref{sec:build manifold}. We stress that also the indexing $\mathcal A$, not only the family of maps by itself,  is important when building $W_0$. 

Therefore if we can show that the objects in items $i),\,ii),\,iii)$ above can be taken to be the same for both $Z_1$ and $Z_2$ we are done. To show this we proceed as follows: we first make a suitable choice of the required objects for $Z_1$ and then show that these can be taken also for $Z_2.$

As in the proof of Theorem \ref{mainthm}, we will denote by $C$ a constant depending only on $n$, which might change from line to line. 

\textbf{Construction for $Z_1$}:
The starting point of the construction for $Z_1$ (see Section \ref{sec:coverings}) for $i=0$ is a choice of a set $X_0$ that is $1$-dense in $Z_1$ and a labelling the elements of $X_0$ as $X_0=\{x_{0,1},x_{0,2},....\}=\{x_{0,j}\}_{j\in J_0}$ where $J_0\coloneqq \{1,2,...,\#|X_0|\}$. Then this $J_0$ is precisely the set of indices in $a)$.
Then  $X_0$ is partitioned into disjoint subsets $Q^0_1,Q^0_2,...,Q^0_{N_0},$ satisfying 
\begin{equation}\label{eq:base}
    \sfd_1(x,y)\ge100, \quad \text{for all $x,y \in Q^0_k$ and all $k=1,\dots,N_0$}.
\end{equation}
From this the set of indices $J_0$ is also partitioned as $J_0=\bigcup_{k=1}^{N_0}J_0^k$ where $J_0^k\coloneqq  \{ j\in J_0 \ | x_{0,j}\in Q_k^0\}$. This partition is precisely the one in item $b)$ above.  The existence of $X_0$ and a partition $\{Q_k^0\}$ as above is proved in Proposition \ref{prop:covering}. However by that same result we see that we can, and will,  take $X_0$ to be $1/2$-dense in $Z_1$ and have
\begin{equation}\label{eq:improved}
    \sfd_1(x,y)\ge101, \quad \text{for all $x,y \in Q^0_k$ and all $k=1,\dots,N_0$}.
\end{equation}
Next we must choose the maps $\mathcal I_0=\{ I_{j_1j_2}\}_{(j_1,j_2)\in \mathcal A}$ for $Z_1$. We recall here how this is done in the construction. First we need fix some $C\eps(n)$-GH-approximations $\alpha_j,\beta_j$ from $B_{200}^{\rr^n}(0)$ to $B_{200}^{Z_1}(x_{0,j})$ and vice-versa, which are one the inverse of the other up to an error of $C\eps(n)$  (see Section \ref{sec:irr charts}). For the current proof   any such choice will do, the only precaution is that we take $\beta_j$ to be the restriction to $B_{200}^{Z_1}(x_{0,j})$ of a $C\eps(n)$-GH-approximation map  $\beta_j: B^{Z_1}_{200+\eps(n)}(x_{0,j})\to  B_{200+\eps(n)}^{\rr^n}(0)$ (this is clearly possible by the assumption $\sfd_{GH}(B^{Z_1}_{200}(x_{0,j}),B_{200}^{\rr^n}(0))\le 200\eps(n)$ and by Theorem \ref{ghchar}). Then for all indices $j_1,j_2$ such that $\sfd_1(x_{0,j_1},x_{0,j_2})<30$,  an isometry $I_{j_1j_2}$ is chosen with the only requirement that
\begin{equation}\label{eq:isom prop}
    |I_{j_1,j_2}-\beta_{j_1}\circ\alpha_{j_2}|\le C\eps(n), \quad \text{in $B_{45}(0)$,}
\end{equation}
and that $I_{j_1,j_2}^{-1}=I_{j_2,j_1}$ (see in particular Lemma \ref{Definition of the transition maps} and Definition \ref{maps}).
Finally, again as in Definition \ref{maps}, the set $\mathcal I_0$ can  be chosen as any subset of the above maps so that
\begin{equation}\label{eq:lol}
   \{I_{j_1j_2}\ : j_1,j_2 \in J_0 \text{ and }\sfd_1(x_{0,j_1},x_{0,j_2})<29\}\subset   \mathcal I_0.
 \end{equation}
 Here for $Z_1$ we will choose $ \mathcal I_0\coloneqq  \{I_{j_1j_2}\ : j_1,j_2 \in J_0 \text{ and }\sfd_1(x_{j_1},x_{j_2})<30-1/2\}$, which clearly satisfies \eqref{eq:lol}. 

\textbf{Construction for $Z_2$}: We now need to show that the choice of indices $J_0$, the partition $\{J_0^k\}_{k=1}^{N_0}$ and  the class of isometries $\mathcal I_0$ works also for $Z_2$. As we did for $Z_1$ we need first to chose a set $Y_0\subset Z_2$ and a partition of it:
\[
Y_0\coloneqq \Phi^*(X_0), \quad \tilde Q^0_k\coloneqq \Phi^*(Q^0_k), \quad \text{for all $k=1,\dots,N_0$}.
\]
We also label the elements in $Y_0$, as in $Z_1$, using $J_0$, i.e.\ $y_{0,j}\coloneqq \Phi^*(x_{0,j})$ for all $j \in J_0.$
It is then immediate to check, provided $\delta <1/4$, that $\tilde X_0$ that is $1$-dense in $Z_2$ and that, thanks to \eqref{eq:improved}, condition \eqref{eq:base} holds for $\tilde Q^0_k$ in $Z_2$. Clearly this construction induces the same  partition $\{J_0^k\}_{k=1}^{N_0}$  of $J_0.$ This shows that the objects in $i)$ and $ii)$ can be taken to be the same for $Z_2.$

Next we need to choose the isometries for $Z_2$. To do so we  choose the maps
\begin{align*}
    \tilde \alpha_{j} : B^{\rr^n}_{200}(0) \to B^{Z_2}_{200}(y_{0,j}), \quad 
\tilde \beta_{j}: B^{Z_2}_{200}(y_{0,j}) \to B^{\rr^n}_{200}(0)
\end{align*}
as follows:
\[
\tilde \alpha_j(x)\coloneqq \Phi^*\circ \alpha_j((1-\delta)x), \quad \tilde \beta_j\coloneqq\beta_j\circ  \Phi_*.
\]
It is easily checked that they have the correct domains of definition and ranges, provided say $\delta \le \eps(n)/2$. Similarly we can check that
\[
d_{\rr^n}(\tilde \beta_{j}\circ \tilde \alpha_{j},\id), \, d_2(\tilde \alpha_{j}\circ \tilde \beta_{j},\id)\le C\eps(n), \quad \text{uniformly},
\]
which follows from the analogous property of $\alpha_j,\beta_j$ and \eqref{eq:phi*}. In particular  $\tilde \alpha_j$ and $\tilde \beta_j$  are so admissible for the construction for $Z_2.$ The key observation is now that, whenever  $\sfd_1(x_{0,j_1},x_{0,j_2})<30-1/2$, it also holds $\sfd_2(y_{0,j_1},y_{0,j_2})<30$ (if say $\delta<1/2$)  and moreover 
\begin{align*}
    |I_{j_1,j_2}-\tilde \beta_{j_1}\circ\tilde \alpha_{j_2}|\le 
    &|I_{j_1,j_2}-\beta_{j_1}\circ\alpha_{j_2}|+|\beta_{j_1}\circ\alpha_{j_2}-\tilde \beta_{j_1}\circ\tilde \alpha_{j_2}|\le C\eps(n),  \quad \text{in $B_{45}(0)$,}
\end{align*}
where in the last estimate we used both \eqref{eq:isom prop} and \eqref{eq:phi*}, provided $\delta$ is small enough. This shows that for such couple of indices   the same map $I_{j_1,j_2}$ used  in $Z_1$ satisfies again \eqref{eq:isom prop}, but in $Z_2$ (up to enlarging the constant $C$). Hence $I_{j_1,j_2}$ is an admissible choice for the couple $j_1,j_2$ (indeed $I_{j_1,j_2}^{-1}=I_{j_2,j_1}$ is automatically true). For the other couples  we might choose other maps $I'_{j_1,j_2}$ (this is irrelevant since they will be discarded). We then choose precisely $\mathcal I_0$   for the maps in $Z_2$. As recalled above the only requirement for $\mathcal I_0$ is to be a subset of all the isometries that we  constructed for $Z_2$   and that \eqref{eq:lol} must hold. The first condition is met because by construction $ \mathcal I_0\coloneqq  \{I_{j_1j_2}\ : j_1,j_2 \in J_0 \text{ and }\sfd_1(x_{j_1},x_{j_2})<30-1/2\}$ and for such indices $j_2,j_2$ we chose precisely $I_{j_1j_2}$ both in $Z_1$ and in $Z_2.$
On the other hand, since  $\sfd_2(y_{0,j_1},y_{0,j_2})<29$ implies that $\sfd_1(x_{j_1},x_{j_2})<30-1/2$, the requirement \eqref{eq:lol} is also clearly satisfied. Hence $\mathcal I_0$ is admissible also for $Z_2$.

For the second part of Corollary \ref{corollary}, suppose that $Z_1$ and $Z_2$ are compact Riemannian manifolds, we know that for $i$ big enough we can take $\beta_{i,j},\alpha_{i,j}$ and $\tilde \beta_{i,j}, \tilde \alpha_{i,j}$ as charts and their inverses (see Remark \ref{rmk:charts}) respectively for $Z_1$ and $Z_2$. In particular for $i$ big enough we can take $W_i=Z_1$ and $\tilde W_i=Z_2$. However we know that $W_i$ and $W_0$ are diffeomorphic via the map $h_i\circ \dots\circ h_0$ (and the same for $\tilde W_i$ and $W_0$). Therefore $Z_1$ and $Z_2$ are both diffeomorphic to $W_0$. This concludes the proof of Corollary \ref{corollary} is concluded.
\end{proof}

\section{Main tools for the proof}

\subsection{Mappings modification theorem}\label{sec:modification}
This section is devoted to the proof of the most important technical result for the proof of the metric Reifenberg's theorem (in particular in the proof of Lemma \ref{bigmodification}) when building the manifolds approximating the metric space. Roughly saying these manifolds will be built starting from a  family of transition maps, but without the knowledge of charts, which instead need to be constructed by hand. However these transitions maps (which are actually isometries of $\rr^n$) will  not  in general be compatible with each other  and thus need to be suitably modified in order to produce an actual manifold. This modification procedure is precisely the content of this section.

\subsubsection{Cocyclical maps}\label{sec:cocyclical}
A central role in the statement and proof of the mapping modification theorem will be played by the notion of \textit{cocyclical maps}.

\begin{definition}\label{def:cocyclical}
	Let $f,g,h  : \rr^n \to \rr^n$ be bijective maps and fix a radius $r>0.$ Define  the maps $\{I_{ab}\}_{a,b \in \{1,2,3\}}$ as follows: $I_{12}\coloneqq f,I_{23}\coloneqq g,I_{13}\coloneqq h$  and then set $I_{ba}=I_{ab}^{-1}$ for every distinct $a,b \in \{1,2,3\}$. We say that the maps $f,g,h$ are \textit{$r$-cocyclical} if  for any distinct $a,b,c\in \{1,2,3\}$ we have that
	\begin{equation}\label{cocyclical}
		\begin{split}
			&\text{for any point } x \in B_r(0) \text{ such that } I_{ba}(x)\in B_r(0), I_{cb}(I_{ba}(x)) \in B_r(0), \\
			&\text{it holds } I_{ca}(x)=I_{cb}(I_{ba}(x)). 
		\end{split}	
	\end{equation}
\end{definition}
We point out  that the above definition is  independent of the order of the three functions, i.e.\ $f,g,h$ are $r$-cocyclical if and only if $g,f,h$ are $r$-cocyclical and so on.  Moreover it is immediate from the definition that $f,g,h$ are $r$-cocyclical if and only if $f^{-1},g,h$ are $r$-cocyclical. Finally observe that if $f,g,h$ are $r$-cocyclical, then they are also $s$-cocyclical for any $s<r.$
\begin{remark}\label{equivdef}
	It is worth to observe that asking \eqref{cocyclical}  is equivalent to ask that the following binary relation, defined on the the disjoint union  $B_1\sqcup B_2 \sqcup B_3$, of three copies of the Euclidean ball $B^{\rr^n}_r(0)$,  is transitive and symmetric:
	\[x \sim y \text{ with } x \in B_a, \,y \in B_b\Longleftrightarrow I_{ba}(x)=y ,\]
	where $\{I_{ab}\}_{a,b \in \{1,2,3\}}$ are defined as above and $I_{aa}$ is the identity map.\fr
\end{remark}
\begin{remark}\label{permutations}
	We observe that if \eqref{cocyclical} is satisfied for a particular choice $a,b,c$, then it is automatically satisfied also for the choice $c,b,a$.  Therefore it is for example  enough to check it with $(a,b,c)=(1,2,3),\,(3,1,2),\,(1,3,2).$ \fr
\end{remark}

The following simple result will be useful to quickly check the cocyclical condition. Roughly said it tells us that if three maps are almost cocyclical at some radius and \eqref{cocyclical} is verified  for one choice of $a,b,c$ and for that same radius, then they are fully cocyclical at a slightly smaller radius.
\begin{prop}\label{prop:trick}
	Fix $\eps>0$. Let $f,g,h$ and  $\{I_{ab}\}_{a,b \in \{1,2,3\}}$ be as in Definition \ref{def:cocyclical}. Suppose that for every $a,b,c \in \{1,2,3\}$ distinct it holds
 \begin{equation}\label{eq:almost cocycl}
     |I_{cb}\circ I_{ba}-I_{ca}|< \eps , \quad \text{ in } B_r(0)
 \end{equation}
and that
	\begin{equation}\label{cocyclical2}
		\begin{split}
			&\text{for any point } x \in B_r(0) \text{ such that } I_{21}(x)\in B_r(0), I_{32}(I_{21}(x)) \in B_r(0), \\
			&\text{it holds } I_{31}(x)=I_{32}(I_{21}(x)).
		\end{split}	
	\end{equation}
Then the maps $f,g,h$ are $(r-\eps)$-cocyclical.
\end{prop}
\begin{proof}
	Thanks to Remark \ref{permutations} we need to verify \eqref{cocyclical} only for $(a,b,c)$ in the cases $(1,2,3)$, $(3,1,2)$ and $(1,3,2).$ The case $(a,b,c)=(1,2,3)$ is true by hypothesis.
	
We check $(3,1,2).$	Suppose now that $x,\, I_{13}(x), \, I_{21}(I_{13}(x)) \in B_{r-\eps}(0)$. Inequality \eqref{eq:almost cocycl}  with $a=1$, $b=2$, $c=3$ and computed at $I_{13}(x)\in B_r(0)$ reads as $
	| I_{32}(I_{21}(I_{13}(x)))-x|<\eps.$
In particular $I_{32}(I_{21}(y)) \in B_r(0)$ with $y\coloneqq I_{13}(x) \in B_r(0)$. Moreover we are assuming that $I_{21}(y)=I_{21}(I_{13}(x))\in B_r(0) $, therefore from \eqref{cocyclical2} we have $x=I_{31}(y)=I_{32}(I_{21}(y))=I_{32}(I_{21}(I_{13}(x)))$, from which applying the map $I_{23}$ we obtain $I_{23}(x)=I_{21}(I_{13}(x)).$ 

	We now check $(1,3,2).$ Suppose  that $x,\, I_{31}(x), \, I_{23}(I_{31}(x)) \in B_{r-\eps}(0)$. From \eqref{cocyclical2} follows that 
	$
	|  I_{23}(I_{31}(x))-I_{21}(x)|<\eps.
	$
	In particular $I_{21}(x) \in B_r(0)$. Moreover again from \eqref{cocyclical2} $|  I_{32}(I_{21}(x))-I_{31}(x)|<\eps,$ hence also $I_{32}(I_{21}(x)) \in B_r(0)$. Therefore from \eqref{cocyclical2} we have $I_{32}(I_{21}(x))=I_{31}(x)$, from which applying $I_{23}$ we conclude.
\end{proof}

\subsubsection{Statement of the main result}

 We need first to introduce some notation: 
\begin{itemize}
	\item $J$ is a countable set of indices,
	\item  $\{J_{i}\}_{i=1}^N$, $N \ge 3,$ is a partition of $J$ and for every $j \in J$ we denote by $n(j)$  the unique integer such that $j\in J_{n(j)},$
	\item  $\mathcal{A}\subset J\times J$ is a set with the following two properties: 
	\begin{equation}\label{asimm}
		(j_1,j_2)\in \mathcal{A} \Longrightarrow (j_2,j_1)\in \mathcal{A},
	\end{equation}
	\begin{equation}\label{apartition}
		(j_1,j_2),(j_1,j_3)\in \mathcal{A} \Longrightarrow n(j_1)\neq n(j_2)\neq n(j_3)\neq n(j_1). 
	\end{equation}
\end{itemize}

We can now state the main result of this section. See also Section \ref{sec:cocyclical} for the definition of cocyclical maps.
\begin{theorem}[Mappings modification theorem]\label{bigmodification2}
	For every $n,M \in \mathbb{N}$, there exist $C=C(n,M)>0$ and $\bar \beta=\bar \beta(n,M)>0$ such that the following holds. Fix $t>0$  and let $J$, $\{J_{i}\}_{1\le i\le N}$ and $\mathcal{A}$ be as above and such that $N\le M$. Suppose $\{I_{j_1j_2}\}_{(j_1,j_2)\in \mathcal{A}}$ is a family of global isometries of $\rr^n$ with the following properties:
	
	\begin{enumerate}[label=\text{\Alph*)}]
		\item \label{A)} $I_{j_2j_1}=I_{j_1j_2}^{-1},$
		\item \label{B)} $I_{j_3j_2}(I_{j_2j_1}(B_{8t}(0)))\cap B_{9t}(0)\neq \emptyset $ $\Longrightarrow$ $(j_3,j_1)\in \mathcal{A},$
		\item \label{C)} if $(j_1,j_2),(j_2,j_3),(j_3,j_1)\in \mathcal{A}$, then
		\begin{equation}\label{eq:assumption C}
			|I_{j_3j_2}\circ I_{j_2j_1}-I_{j_3j_1}|\le  \beta t, \quad \text{ in $B_{10t}(0),$}
		\end{equation}
		for some $\beta<\bar \beta$. 
	\end{enumerate}
	Then there exists another family $\{\widetilde I_{j_1j_2}\}_{(j_1,j_2)\in \mathcal{A}}$ of $C^\infty$-global diffeomorphisms of $\rr^n$ such that	$\widetilde I_{j_2j_1}=\widetilde I_{j_1j_2}^{-1}$ and satisfying the following compatibility condition.  For every $(j_1,j_2),(j_3,j_2)\in \mathcal{A}$ for which the set  $\{x \in B_{8t}(0) \ : \  \tilde I_{j_2j_1}(x),\tilde I_{j_3j_2}(\tilde I_{j_2j_1}(x))\in B_{8t}(0)\}$ is not empty, we have $(j_3,j_1)\in \mathcal{A}$ and  the maps $\tilde I_{j_2j_1},\tilde I_{j_3j_2},\tilde I_{j_3j_1}$ are $8t$-cocyclical.
	Moreover for every $(j_1,j_2)\in \mathcal{A}$ it holds that
	\begin{equation}\label{finalb2}
		\|I_{j_1j_2}-\widetilde I_{j_1j_2} \|_{C^2(\rr^n),t}\le C \beta.
	\end{equation}
\end{theorem}
We briefly explain the role of the subclasses $J_i$ in which of the family of indices $J$ is partitioned.  The key point is that the triples of maps $I_{j_2j_1}, I_{j_3j_2}, I_{j_3j_1}$ for which we need to obtain the compatibility conditions are such that the indices $j_1,j_2,j_3$ belong to pairwise different  subclasses $J_i$  (by \eqref{apartition}). In particular every map takes part only in at most $N$ of the triples of maps that we need to consider. Indeed in every triple of the form $I_{j_2j_1}, I_{j_3j_2}, I_{j_3j_1}$, the maps pairwise share an index, but
by \eqref{apartition} and the pigeonhole principle we have
\begin{equation}\label{eq:pigeon}
    \# \{j \in J \ : \  (j_1,j)\in \mathcal A\}\le N, \quad \forall j_1\in J.
\end{equation}
This fact will allow to modify every map only a finite (and controlled) number of times, even if we have no control on the total number of maps (which might be even infinite).

\begin{remark}
	Similarly to Remark \ref{equivdef}, we observe that the compatibility condition required in Theorem \ref{bigmodification2} is equivalent to ask  that the following binary relation is  transitive and symmetric. Let $\bigsqcup_{j\in J} B_j$, the disjoint union of  copies of the Euclidean ball $B_{8t}(0)$, indexed by $J$, set
	\[x \sim y \text{ with } x \in B_{j_1}, \,y \in B_{j_2}\Longleftrightarrow \tilde I_{j_2j_1}(x)=y .\]\fr
\end{remark}

\subsubsection{Three-maps modification lemma}\label{sec:basic}
The proof of Theorem \ref{bigmodification2} will follow an algorithm based on the iteration of the following result. Roughly speaking it says that, given three maps which are
both close to isometries \eqref{TML2} and almost cocyclical at a given scale \eqref{TML1}-\eqref{TML2}, we can slightly modify one of them \eqref{bound} to make them cocyclical at a slightly smaller scale \eqref{cocycl}. The crucial part of this result is that this said map is modified only where strictly needed and left unchanged everywhere else (see \eqref{locality} and \eqref{coherence}). This will allow us in the modification algorithm to modify the same map more than once, without disrupting the work done in the previous steps.
\begin{lemma}[Three-maps modification]\label{TML}
	Fix $N,n,k\in \mathbb{N}$, $k\ge 2,$ and also two real numbers $t>0$ and $r\ge 2$. Then there exist constants $C_1=C_1(n,N)$ and $\delta_1(n,N)$ such that the following holds. Suppose we have  a  $C^k$-global  diffeomorphisms of $\rr^n$ $I_{ab}$ for $a,b=1,2,3$ and $a \neq b$ and for which $I_{ab}=I_{ba}^{-1}.$ Suppose we have also some corresponding global isometries $I'_{ab}$ (again $I'_{ab}=(I'_{ba})^{-1}$) for which
	\begin{equation}\label{TML1}
		| I'_{ab}-I'_{ac}\circ  I'_{cb}|\le \eps t, \quad \text{in $B_{rt}(0)$}
	\end{equation}
	for every distinct $a,b,c$ and for some number $\eps <\delta_1(n,N)$.  Suppose finally that
	\begin{equation} \label{TML2}
		\| I_{ab}-I_{ab}'\|_{C^2(\rr^n),t}\le \eps 
	\end{equation}
	for every $a,b$.

	Then there exists a $C^k$-global diffeomorphism $ \hat I_{32}$ of $\rr^n$ such that
	\begin{equation}\label{locality}
		\hat I_{32}=I_{32} \text{ outside } I_{21}(B_{rt}(0)),
	\end{equation} 
	\begin{equation}\label{coherence}
		\hat I_{32}(x)=I_{32}(x) \text{ for any } x \in B_{rt}(0) \text{ such that } I_{32}(I_{21}(x))=I_{31}(x),
	\end{equation}
	\begin{equation}\label{cocycl}
		I_{21},I_{31},\hat I_{32} \text{ are } \big(1-\frac{1}{N}\big )rt\text{-cocyclical},
	\end{equation}
	and finally
	\begin{equation}\label{bound}
		\|I_{32}-\hat I_{32}\|_{C^2(\rr^n),t},\|I_{23}-\hat I_{23}\|_{C^2(\rr^n),t}\le C_1\eps,
	\end{equation} 
	where $\hat I_{23}=\hat I_{32}^{-1}.$
\end{lemma}
For the proof we will need the following elementary technical result.
\begin{lemma}\label{magic}
	For every $n \in \mathbb{N}$ and $\eta >0$, there exists constants $C_2=C_2(n,\eta)>0, \delta_2=\delta_2(n,\eta)\in(0,1)$ with the following property. Let $m \in \{1,2\}$. Let $U_1, U_2$ open bounded subsets of  $\rr^n$ such that $\bar U_1 \subset U_2$ and $\sfd(\bar U_1,U_2^c)\ge \eta t$ with $t>0.$ Suppose $H\in C^k(U_2;\rr^n)$,   with $k \ge 2$, satisfies  
	$$\|H-{\sf{id}}\|_{C^m(U_2),t}\le  \eps$$
	for some $\eps < \delta_2.$
	
	Then there exists a smooth  global diffeomorphism $\hat H : \rr^n \to \rr^n$ such that $H|_{U_1}=\hat H|_{U_1}$, $\hat H|_{U_2^c}=\id$, $\hat H(x)=H(x)$ whenever $H(x)=x$ and
	\begin{equation}\label{H}
		\|\hat H-{\sf{id}}\|_{C^m(\rr^n),t}\le  C_2(n,\eta)\eps.
	\end{equation}
	Moreover $\hat H^{-1}$ is $C^k$ and 
	\begin{equation}\label{invH}
		\|\hat H^{-1}-{\sf{id}}\|_{C^m(\rr^n),t}\le  C_2(n,\eta)\eps.
	\end{equation}
\end{lemma}		
\begin{proof}
		Is is enough to  consider $t=1,$ since the case of a general $t$ follows by scaling observing that the norms $\|\cdot\|_{C^m,t}$ are scaling invariant (recall Remark \ref{rmk:scaling invariant}).
	Let  $\phi \in C_c^{\infty}(\rr^n)$ be such that $\phi=1 $ on $U_1$ and $\phi=0$ on $U_2^c$ and such that $|\partial_{i,j}\phi|,|\partial_i\phi |\le c=c(n,\eta)$ and $|\phi|\le 1$. Define 
	\[\hat H\coloneqq  (H-{\sf{id}})\phi+\id. \]
	Then using \eqref{opnorm}
	$$|\hat H-\id|\le \eps, $$
	$$|\partial_i(\hat H-{\sf{id}})_k|\le |\partial_i (H-{\sf{id}})_k||\phi|+|(H-{\sf{id}})_k| |\partial_i \phi  |\le (1+c) \eps, $$
	\begin{align*}
		|\partial_{i,j}(\hat H-{\sf{id}})_k|&\le  |\partial_{i,j} H_k||\phi|+|\partial_{i}(H-{\sf{id}})_k||\partial_j\phi |+|\partial_{j}(H-{\sf{id}})_k||\partial_j\phi |+|(H-{\sf{id}})_k||\partial_{i,j} \phi|\\
		&\le \tilde c(\eta,n) \eps 
	\end{align*}
	on $\rr^n$, where $\tilde c$ is a constant depending only on $\eta$ and $n$. This proves \eqref{H}. It is also clear that $H|_{U_1}=\hat H|_{U_1}$, $\hat H|_{U_2^c}=id$ and $\hat H(x)=H(x)$ whenever $H(x)=x$. Moreover from the first and second bound above, if $\eps$ is small enough with respect to $\eta$ and $n$, then $\hat H$ is a diffeomorphism and $\hat H^{-1}$ is $C^k$ by Lemma \ref{diffeo}.   Then \eqref{invH} is a direct consequence of Lemma \ref{inversebound}.
\end{proof}

We are now ready to prove the three-maps modification lemma.
\begin{proof}[Proof of Lemma \ref{TML}]
	It is clear from \eqref{TML1} that
	\begin{equation*}
		| I_{32}'( I_{21}'(I_{13}'))-{\sf{id}}|\le\eps t 
	\end{equation*}	
	in $B_{rt}(0)$. Moreover from the above and \eqref{TML2}
	\begin{equation}\label{Hbound1}
		| I_{32}(I_{21}(I_{13}))-{\sf{id}}|\le4\eps t 
	\end{equation}
	in $B_{rt}(0)$. 
	Define now the set $A\coloneqq I_{32}(I_{21}(B_{(1-\frac{1}{2N})rt}(0)))\cap B_{(1-\frac{1}{2N})rt}(0).$ We distinguish two cases:
	
	\noindent {\sc Case 1:} $A=\emptyset$. We simply take $\hat I_{32}=I_{32}.$ Indeed in this case $I_{21},I_{32},I_{31}$ are vacuously  $(1-\frac{1}{N})rt$-cocyclical, i.e.\ the set of points where we need to check \eqref{cocyclical} is empty. To see this denote $B_{(1-\frac{1}{N})rt}(0)$ by $B$. Suppose that there exists $x \in B$ such that $I_{21}(x),I_{32}(I_{21}(x))\in B, $ then we would have $A\neq \emptyset.$ Suppose  instead that there exists $x \in B$ such that $I_{31}(x),I_{23}(I_{31}(x))\in B, $ then applying to $I_{31}(x)$ the map in \eqref{Hbound1}, if  $\eps<1/(100N)$, we deduce  that $I_{32}(I_{21}(x))\in A$ and so $A\neq \emptyset.$ Finally suppose there exists $x \in B$ such that $I_{13}(x),I_{21}(I_{13}(x))\in B $ and set $y\coloneqq I_{13}(x)\in B.$ Then from \eqref{Hbound1}  we deduce $I_{32}(I_{21}(y)) \in B_{(1-\frac{1}{2N})rt}(0)$, therefore again $A\neq \emptyset.$ From Remark \ref{permutations}, this is enough to prove that $I_{21},I_{32},I_{31}$ are $(1-\frac{1}{N})rt$-cocyclical.
	
	\noindent {\sc Case 2:} $A\neq \emptyset$.
	Since $S\coloneqq I_{31}'( I_{12}'(I_{23}'))$ is an isometry, from Lemma \ref{isombound} (recall $r \ge 2$) we get $\|DS-{\sf{id}} \|\le 4\eps$ and thus $\|S-{\sf{id}}\|_{C^2(B_{rt}(0)),t}\le 4\eps$. Therefore from \eqref{TML2}, assuming $\eps <1$ and applying Lemma \ref{iterateclosest2} we obtain
	\begin{equation} \label{Hbound2}
		\|I_{31}( I_{12}(I_{23}))-{\sf{id}}\|_{C^2(B_{rt}(0)),t}\le C\eps,
	\end{equation}
	where $C$ is a constant depending only on $n$.
	Define now the following open sets
	\[ U_2\coloneqq  I_{32}(I_{21}(B_{rt}(0)))\cap B_{rt}(0) \]
	\[ U_1\coloneqq A= I_{32}(I_{21}(B_{(1-\frac{1}{2N})rt}(0)))\cap B_{(1-\frac{1}{2N})rt}(0) \]
	that are non empty, since we assumed $A\neq \emptyset.$
	Clearly $U_1 \subset U_2\subset B_{rt}(0)$ and notice that, by \eqref{TML2} and the fact that $I_{ab}'$ are isometries, if say $\eps \le \frac{1}{100N}$, we have $\sfd(\bar U_1, U_2^c)\ge t/(10N).$ Define now $H\coloneqq  I_{31}\circ  I_{12} \circ I_{23}: U_2 \to \rr^n$.  Thanks to  \eqref{Hbound2} we can now apply Lemma \ref{magic} (provided $\eps \le \delta_2(n,1/(10N))C^{-1}$, where $\delta_2$ is the one given by Lemma \ref{magic}) with $H$, with $U_1$,$U_2$,$t$ and $\eps$ to deduce the existence of  a $C^k$-global diffeomorphism $\hat H$ such that
	\begin{equation}\label{restr}
		\hat H|_{U_1}= I_{31}\circ  I_{12} \circ  I_{23},
	\end{equation}
	\begin{equation}\label{coherence2}
		\hat H(x)=H(x) \text{ whenever }  H(x)=x ,
	\end{equation}
	\begin{equation}\label{u2c}
		\hat H|_{U_2^c} = \id,
	\end{equation}
	\begin{equation}\label{hatHbound}
		\|\hat H^{-1}-{\sf{id}}\|_{C^2(\rr^n),t}, \|\hat H-{\sf{id}}\|_{C^2(\rr^n),t}\le \tilde C\eps,  
	\end{equation}
	where $\tilde C$ is constant depending only on $n$ and $N$.
	We define $ \hat I_{32}=\hat H \circ   I_{32}$ and call $\hat I_{23}$ its inverse. Clearly \eqref{coherence} is satisfied thanks to \eqref{coherence2}. Moreover \eqref{locality} holds from \eqref{u2c}, since $I_{32}(I_{21}(B_{rt}(0)^c))\subset U_2^c$.  We now verify \eqref{bound}. The first bound follows directly from \eqref{hatHbound} and \eqref{TML2}, applying Lemma \ref{iterateclosest}. The bound for the inverse, $\hat I_{23}$, follows from Lemma \ref{inversebound}, provided $\eps$ is small enough. Finally we need to prove the cocyclical condition in \eqref{cocycl}. 
 Suppose first that there exists $x \in B_{(1-\frac{1}{3N})rt}(0)$ such that $I_{21}(x),\hat I_{32}(I_{21}(x))\in B_{(1-\frac{1}{3N})rt}(0).$ Then, if $\eps$ is small enough with respect to $n$ and $N$, from \eqref{bound} we deduce that $I_{32}(I_{21}(x)) \in B_{(1-\frac{1}{2N})rt}(0)$, that implies $I_{32}(I_{21}(x)) \in U_1$, therefore from \eqref{restr}
	\[ \hat I_{32}(I_{21}(x))=\hat H(I_{32}(I_{21}(x)))=I_{31}\circ  I_{12} \circ  I_{23} \circ I_{32}(I_{21}(x))=I_{31}(x). \]
	\eqref{cocycl} then follows from Proposition \ref{prop:trick}, indeed notice that its hypotheses are satisfied thanks to \eqref{TML1} and \eqref{bound}, provided $\eps$ is small enough.
\end{proof}

\subsubsection{Proof of the mapping modification theorem}

\begin{proof}[Proof of Theorem \ref{bigmodification2}]
  Let $\{I_{j_1j_2}\}_{(j_1,j_2)\in \mathcal{A}}$ be a family of global isometries of $\rr^n$ satisfying assumptions \ref{A)}, \ref{B)}, \ref{C)} of the statement. Recall also that, by assumption, $\mathcal A\subset J\times J$ satisfying \eqref{asimm}, \eqref{apartition}, where  $J$ is a countable set of indices partitioned into sets  $\{J_{i}\}_{i=1}^N$, $3\le N\le M$. Recall also that for every $j \in J$ we denote by $n(j)$  the unique integer such that $j\in J_{n(j)}.$
The goal is to construct  new maps $\{\widetilde I_{j_1j_2}\}_{(j_1,j_2)\in \mathcal{A}}$ satisfying suitable compatibility properties as in the conclusion of the theorem.

	This will be achieved by modifying slightly the original maps and by iterating Lemma \ref{TML}. 
	The proof is divided in the following parts: first we describe the iterative algorithm that we use to modify the maps  (\textbf{Part 0} and \textbf{Part 1}); after this we see the effect of this modification  (\textbf{Part 2});  then we prove that the algorithm is applicable, by showing that the hypotheses of Lemma \ref{TML} are satisfied at every step  (\textbf{Part 3}); finally we prove that the maps that we obtain satisfy the compatibility conditions required by Theorem \ref{bigmodification}  (\textbf{Part 4}).

	\medskip
	
\noindent	{\textbf{Part 0, preparation:}}\\
	We choose $\beta \le \frac{\delta_1}{100M^3(C_1+1)^{M^3}}$ where  $\delta_1=\delta_1(n,M^3),C_1=C_1(n,M^3)$ are given in Lemma \ref{TML}. Define also the numbers $\beta_k\coloneqq (C_1+1)^k\beta$ for $k=0,1,\dots,N^3$.
	
	In the modification of the maps we will need to proceed in a precise  ordered fashion. To formalize such procedure we need to introduce the following notation.
	 
	Set $N_3\coloneqq \binom{N}{3} (=\frac{N(N-2)(N-1)}{6})$. Define the set $$\mathscr{T}\coloneqq \{(a,b,c) \ | \ 1\le c<b<a\le N\} $$ and consider the enumeration of the elements of $\mathscr{T}=\{T_1,...,T_{N_3}\}$ defined as follows. Set $T_1\coloneqq (3,2,1)$. If $T_k=(a,b,c)$ then set:
	\begin{itemize}
		\item $T_{k+1}=(a,b,c+1)$ if $c<b-1$,
		\item $T_{k+1}=(a,b+1,1)$ if $c=b-1$ and $b<a-1$,
		\item $T_{k+1}=(a+1,2,1)$ if $T_k=(a,a-1,a-2).$
	\end{itemize}  
	In other words we choose the enumeration so that the sequence $\#T_k$ is increasing, where $\#T_k$ is the 3-digit numbers formed by the  entries of $T_k$ (from right to left): e.g.\  $T_1=(3,2,1), T_2=(4,2,1),T_3=(4,3,1), T_4=(4,3,2), T_5=(5,2,1)...$ and so on.
	
	\medskip
\noindent	{\textbf{Part 1, modification procedure:}}
	
	We divide the modification in a finite number of steps $k=1,...,N_3$. At every step we produce for every map $I_{j_1j_2}, (j_1,j_2)\in \mathcal{A},$ a modified map that will be called $I^k_{j_1j_2}.$  \\
	We start by setting $ I^0_{j_1j_2}\coloneqq I_{j_1j_2}$ for every $ (j_1,j_2)\in \mathcal{A}.$
	\begin{itemize}
		\item For every  $k=1,...,{N_3}$ we  do the following:
	\end{itemize}

\begin{tcolorbox}[colframe=white,colback=mygray]
	{\bf Step $k$}: Consider $T_k=(a_3,a_2,a_1)$ and for every triple of maps  $I^{k-1}_{j_{2}j_{1}},I^{k-1}_{j_{3}j_{1}},I^{k-1}_{j_{3}j_{2}}$ with $n(j_1)=a_1 , n(j_2)=a_2 ,n(j_3)=a_3$, 
	apply Lemma \ref{TML} (see the next part for the verification of the hypotheses) with 
	\[  I_{21}=I^{k-1}_{j_{2}j_{1}},\quad  I_{31}= I^{k-1}_{j_{3}j_{1}},\quad I_{32}= I^{k-1}_{j_{3}j_{2}},  \]
	\[  I_{21}'=I^{0}_{j_{2}j_{1}},\quad I_{31}'= I^{0}_{j_{3}j_{1}},\quad I_{32}'= I^{0}_{j_{3}j_{2}},  \]
	\[t=2^{-i}\]
	\[N=M^2\]
	\[r=10-\frac{k-1}{N^3}\]
	\[\eps=\beta_k=(C_1+1)^k\beta\]
	to produce a modified map $I^{k}_{j_{3}j_{2}}$ (that is the map  $\hat I_{32}$ given by the Lemma) and then set $I^{k}_{j_{2}j_{3}}\coloneqq (I^{k}_{j_{3}j_{2}})^{-1}$. Moreover   set $I^{k}_{j_1j_2}\coloneqq I^{k-1}_{j_1j_2},I^{k}_{j_1j_3}\coloneqq I^{k-1}_{j_1j_3}$ and the same for their inverses. 

	Finally for every map $I_{j,\bar j}^{k-1}$ that does not belong to any triple considered above and neither does its inverse,  we simply set $I_{j, \bar j}^{k}\coloneqq I_{j,\bar j}^{k-1},$ $I_{\bar j,  j}^{k}\coloneqq I_{\bar j, j}^{k-1}.$  
\end{tcolorbox}

 At the end of the iteration define $ \widetilde I_{j_1j_2}\coloneqq I^{{N_3}}_{j_1j_2}$.

\noindent {\emph{Important remark:}} Note that assumption \eqref{apartition} (see also \eqref{eq:pigeon})  ensures that every map $I^{k-1}_{j,\bar j}$ belongs to at most one of  the triples considered at \textbf{Step $k$}, hence the above procedure makes sense as we are trying to modify the same map more than once in the same step.

\medskip
\noindent	{\textbf{Part 2, effect of the modification:}}
	
	We gather here the properties of the new maps produced by the modification. From Lemma \ref{TML} (given that we can apply it) we have that 
	\begin{equation}\label{kbound}
		\| I^k_{j_{3}j_{2}}- I^{k-1}_{j_{3}j_{2}}\|_{C^2(\rr^n), t}\le \beta_k C_1,\quad  \| I^k_{j_{2}j_{3}}- I^{k-1}_{j_{2}j_{3}}\|_{C^2(\rr^n), t}\le \beta_k C_1,
	\end{equation}
	\begin{equation}\label{local}
		I^{k}_{j_{3}j_{2}}=I^{k-1}_{j_{3}j_{2}}\quad  \text{  in  } \quad I^{k-1}_{j_{2},j_{1}}(B_{(10-\frac{k-1}{N^3})t}(0))^c.
	\end{equation} 
	\begin{equation}\label{coherence3}
		I^k_{j_3j_2}(x)=I^{k-1}_{j_3j_2}(x) \text{ for any } x \in B_{(10-\frac{k-1}{N^3})t}(0) \text{ such that } I^{k-1}_{j_3j_2}(I^{k-1}_{j_2j_1}(x))=I^{k-1}_{j_3j_1}(x),
	\end{equation}
	and
	\begin{equation}\label{cocycl3}
		I^{k}_{j_{2}j_{1}},I^{k}_{j_{3}j_{1}},I^{k}_{j_{3}j_{2}} \text{ are } \left(10-\frac{k}{N^3}\right)t\text{-cocyclical}.
	\end{equation}
	Observe that \eqref{cocycl3} holds because $N\le M$ and being $r$-cocyclical implies being $s$-cocyclical for any $s<r$.
	
	\medskip 
	
\noindent	{\textbf{Part 3, verification of the hypotheses needed to apply Lemma \ref{TML}:}}\\
	The fact that $I_{ab}=I_{ba}^{-1}$ and $I'_{ab}=I_{ba}'^{-1}$ is granted by \eqref{asimm} and the fact that at the end every modification step, we set $I^{k}_{j_{2}j_{3}}\coloneqq (I^{k}_{j_{3}j_{2}})^{-1}$.  Our initial assumption on $ \beta$ implies that $\beta_k\le \delta_1(n,N)$ for every $k$, therefore  we only need to prove that \eqref{TML1} and \eqref{TML2} are satisfied for some $\eps\le \beta_k$. \eqref{TML1} is always satisfied by assumption \ref{C)} with $\eps=\beta\le \beta_k$. Therefore we only need to prove that
	\begin{equation}\label{k0}
		\|I^{k-1}_{j_1j_2}-I^0_{j_1j_2}\|_{C^2(\rr^n), t}\le\beta_k.
	\end{equation}
	We can prove this by induction. It is trivial if $k=1$. So suppose it is true for $k$, in particular we can perform the above modification at least up to Step $k$. Then  from \eqref{kbound} (that we are assuming to hold at Step $k$) we have
	\[\|I^{k}_{j_1j_2}-I^0_{j_1j_2}\|_{C^2(\rr^n), t}\le\|I^{k-1}_{j_1j_2}-I^0_{j_1j_2}\|_{C^2(\rr^n), t}+\|I^{k-1}_{j_1j_2}-I^k_{j_1j_2}\|_{C^2(\rr^n), t}\le C\beta_k+\beta_k =\beta_{k+1},\]
	that proves \eqref{k0}. Notice also that from \eqref{k0}, \eqref{finalb2} already follows, indeed
	\[\|\tilde I_{j_1j_2}-I_{j_1j_2}\|_{C^2(\rr^n), t}=\|I^{{N_3}}_{j_1j_2}-I^0_{j_1j_2}\|_{C^2(\rr^n), t}\le \beta_{N_3}= \beta (C_1+1)^{N_3}	\le\beta (C_1+1)^{M^3} ,\]
	hence it is sufficient to take $C\ge (C_1+1)^{M^3}$, which depends only on $n,M$, since $C_1$ depends only on $n$.	
	
	\medskip
\noindent	{\textbf{Part 4, proof of compatibility conditions:}}
	
	To prove Theorem \ref{bigmodification2} it remains only to prove the compatibility conditions. We report them here for the convenience of the reader:
	
	\noindent\emph{Compatiblity conditions:} for every $(j_1,j_2),(j_3,j_2)\in \mathcal{A}$ for which the set  $$\{x \in B_{8t}(0) \ : \  \tilde I_{j_2j_1}(x),\tilde I_{j_3j_2}(\tilde I_{j_2j_1}(x))\in B_{8t}(0)\}$$ is not empty, we have $(j_3,j_1)\in \mathcal{A}$ and that the maps $\tilde I_{j_2j_1},\tilde I_{j_3j_2},\tilde I_{j_3j_1}$ are $8t$-cocyclical.

	\medskip
	
	We first describe the  idea of the argument.
	
	\noindent\emph{Idea:}  After  Step $k$ of the procedure we clearly have, thanks to \eqref{cocycl3},  that the maps relative to the triple $T_k$ are cocyclical at scale $(10-k/N^3)t$. Therefore what we need to do is check that the modification at Step $k$ does not destroy the compatibility conditions created at the previous steps. For this it turns out to be crucial the fact that we are decreasing the scale at every step and   that we are modifying the maps only where is strictly needed (see in particular \eqref{local} and \eqref{coherence3}).
	
	\medskip
	We pass now to the rigorous part. We claim that to prove the above compatibility conditions is enough to show that the following statement, denoted by {\bf{S($k$)}}, is true for every $k=1,...,{N_3}.$\\
	
	{\bf{S($k$):}} For every $m \le k$ consider $T_m=(a_3,a_2,a_1)$. Let $I^k_{j_1j_2},I^k_{j_1j_3},I^k_{j_2j_3}$ be such that $n(j_1)=a_1 , n(j_2)=a_2 ,n(j_3)=a_3$. These are $(10-\frac{k}{N^3})t$-cocyclical.\\

	To see that this would be sufficient to conclude, suppose there exists two maps $\tilde I_{j_2j_1},\tilde I_{j_3j_2}$ and a point $x \in B_{8t}(0)$ such that $\tilde I_{j_2j_1}(x)$, $\tilde I_{j_3j_2}(\tilde I_{j_2j_1}(x))\in B_{8 t}(0)$. Then from \eqref{finalb2} and by how we chose  $\beta$ at the beginning, we have $I_{j_3j_2}( I_{j_2j_1}(x))\in B_{9t}(0)$. Then, thanks to assumption \ref{B)} we have that $(j_3,j_1)\in \mathcal{A}$. Moreover from \eqref{apartition} we must have that $n(j_1)\neq n(j_2)\neq n(j_3)\neq n(j_1)$, therefore $(n(j_1), n(j_2), n(j_3))=T_m$, for some $m\le {N_3}$. Therefore  {\bf{S($N_3$)}} implies that $\tilde I_{j_2j_1},\tilde I_{j_3j_2}, \tilde I_{j_3j_1}$ are $8t$-cocyclical (indeed $N_3\le N^3/6$).
	
	Observe that we actually used only statement {\bf{S($N_3$)}}, however to prove it we will need to argue by induction and prove every 	{\bf{S($k$)}}.
	
	\medskip
	
	\noindent {\bf Proof of {\bf{S($k$):}}}\\
	We prove it by induction on $k.$
	First we observe that after the step $k$ is completed in the modification procedure, any triple of maps $I^k_{j_1j_2},I^k_{j_1j_3},I^k_{j_2j_3}$ such that  $T_k=(n(j_1), n(j_2), n(j_3))$, is $(10-\frac{k}{N^3})t$-cocyclical by  \eqref{cocycl3}.
	Hence {\bf{S($1$)}} is clearly true. \\
	Suppose now that {\bf{S($k$)}} is true for $k$. Consider $T_{k+1}=(b_3,b_2,b_1)$. Since  $I^{k+1}_{j\bar j}\neq I^k_{j\bar j}$ only if $j \in J_{b_3}$, $\bar j \in J_{b_2}$ (or the opposite), we only need to check {\bf{S($k+1$)}} for $T_m=(b_3,b_2,a_1)$ with $a_1\le b_1.$ Indeed the other cases are true by induction hypothesis.  The case $a_1=b_1$ is immediately verified from the initial observation. Let now $T_m=(b_3,b_2,a_1)$ with $a_1<b_1,$ then $m\le k.$ We need to show that $I^{k+1}_{j_1j_2},I^{k+1}_{j_1j_3},I^{k+1}_{j_2j_3}$ are $(10-\frac{k+1}{N^3})t$-cocyclical. To this aim set  $B_k\coloneqq B_{(10-\frac{k}{N^3})t}(0),$ $B_{k+1/2}\coloneqq B_{(10-\frac{k+1/2}{N^3})t}(0)$ and $B_{k+1}\coloneqq B_{(10-\frac{k+1}{N^3})t}(0)$ so that $B_{k+1}\subset B_{k+1/2}\subset B_k.$ We claim that is sufficient to show that:
	\begin{equation}\label{cocyclicalfinal}\tag{$\spadesuit$}
		\begin{split}
			&\text{for any point } x \in B_{k+1/2} \text{ such that } I^{k+1}_{j_3j_2}(I^{k+1}_{j_2j_1}(x)),I^{k+1}_{j_2j_1}(x) \in B_{k+1/2}, \\
			&\text{it holds } I^{k+1}_{j_3j_2}(I^{k+1}_{j_2j_1}(x))=I^{k+1}_{j_3j_1}(x).
		\end{split}	
	\end{equation}
	Indeed the full cocyclical condition on the smaller ball $B_{k+1}$ would then follow from Proposition \ref{prop:trick}, whose hypotheses are satisfied thanks to assumption \ref{C)}, \eqref{finalb2} and our initial choice of $\beta$.

	 \noindent \emph {Proof of \eqref{cocyclicalfinal}:}  Let $x \in B_{k+1/2}$ be such that $I^{k+1}_{j_3j_2}(I^{k+1}_{j_2j_1}(x)),I^{k+1}_{j_2j_1}(x) \in B_{k+1/2}$. 
	 Notice first that from \eqref{kbound} and how we chose  $\beta$, we have that $I^{k}_{j_3j_2}(I^{k}_{j_2j_1}(x)),I^{k}_{j_2j_1}(x) \in B_{k}$, therefore by \textbf{S($k$)} and induction hypothesis $I^{k}_{j_3j_2}(I^{k}_{j_2j_1}(x))=I^{k}_{j_3j_1}(x). $ Hence we need to show
	\begin{equation}\label{k1k}
		I^{k}_{j_3j_2}(y)=I^{k+1}_{j_3j_2}(y)
	\end{equation}
	where $y=I^{k}_{j_2j_1}(x).$  If the map $	I^{k}_{j_3j_2}$ was not modified at the step $k+1$, i.e.\ $I^{k}_{j_3j_2}=I^{k+1}_{j_3j_2}$, there is nothing to prove. Hence
	we can assume that $I^{k+1}_{j_3j_2}$ has been modified  at step $k+1$ of the modification procedure by applying Lemma \ref{TML} to the maps $I^{k}_{j_3j_2},I^{k}_{j_2,j_0},I^{k}_{j_3,j_0}$ for some $j_0 \in J_{b_1}$. We divide two cases.
	
	\noindent {\sc Case 1:} There is not any $z \in B_k$ such that $I^k_{j_2,j_0}(z)=y$. In this case  \eqref{k1k} follows immediately from \eqref{local}.
	
	\noindent {\sc Case 2:} There exists $z \in B_{k}$ such that $I^k_{j_2,j_0}(z)=y.$ The idea is that in this case the map $I^{k}_{j_3j_2}$ was already correct and needed not to be modified.   Observe first that $x=I^k_{j_1j_2}(I^k_{j_2,j_0}(z))$, $I^k_{j_2,j_0}(z) \in B_k$. Moreover by \eqref{kbound} and by how we chose  $\beta$ at the beginning, we have $I_{j_1j_2}(I_{j_2,j_0}(z)) \in B_{10t}(0)$, hence thanks to assumption \ref{B)} we have that $(j_1,j_0)\in \mathcal{A}$. Therefore by induction hypothesis, since $T_l=(b_2,b_1,a_1)$ with $l\le k$, we have $I^k_{j_1,j_0}(z)=x.$ From this we infer that $I^k_{j_3j_1}(I^k_{j_1,j_0}(z))=I^k_{j_3j_1}(x) \in B_k.$ Therefore again by induction hypothesis since $T_h=(b_3,b_1,a_1)$ with $h\le k$, we have 
	$$I^k_{j_3,j_0}(z)=I^k_{j_3j_1}(I^k_{j_1,j_0}(z))=I^k_{j_3j_1}(x)=I^k_{j_3j_2}(y)=I^k_{j_3j_2}(I^k_{j_2,j_0}(z)).$$
	Hence from \eqref{coherence3} we deduce \eqref{k1k}.  This concludes the proof of Theorem \ref{bigmodification2}.
		\end{proof}
	\begin{figure}[ht!]\label{pallette2}
		\centering
		\includegraphics[width=0.7\textwidth, angle=0]{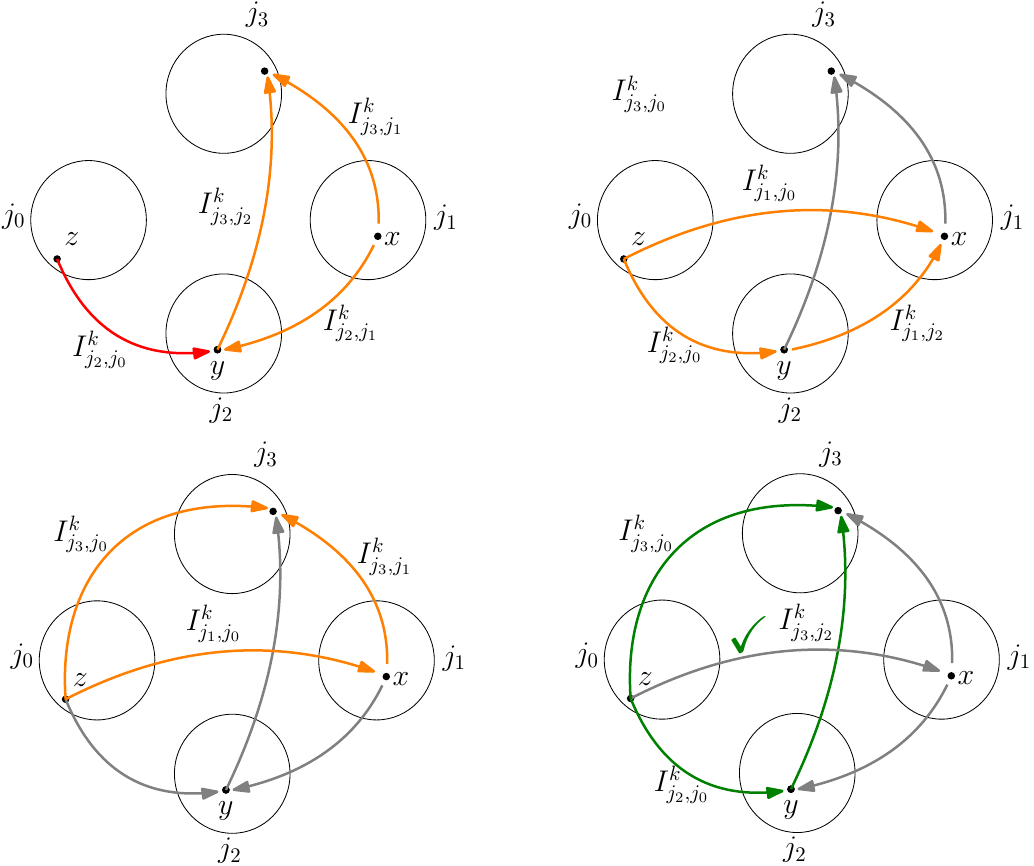}
		\caption{Scheme for the proof of \eqref{cocyclicalfinal}.}
		\label{bk}
	\end{figure}
\FloatBarrier

\subsection{Gluing locally defined manifold-to-manifold immersions}\label{sec:cheeger mapping}
In this section we prove another technical tool, Theorem \ref{patching} stated below, independent of the rest of the note.
This is a variation of a result due to Cheeger \cite[Lemma 3.6 and Lemma 4.1]{Cheeger} and provides a criterion to show that two Riemannian manifolds are diffeomorphic. 
More in details it says that given a manifold $M$, a family of coordinate-charts and  for each chart a smooth embedding from that chart to another manifold,  given that they don't differ too much on the intersection of the charts (condition \eqref{patch3}), we can  modify them by a small amount and glue them together to obtain a smooth immersion map defined in the whole $M$.  This result will be used in the proof of the metric Reifenberg's theorem to build the manifold-to-manifold maps  (see Section \ref{sec:manifoldtomanifold}). 

\begin{theorem}\label{patching}
	For every $N,n \in \mathbb{N}$ and $L\ge 1$ there exist $C_3=C_3(n,N,L)>0$ and $\eps_3(n,N,L)$ with the following property. Let $q \in \mathbb{N}$ with $q \ge 2$ and $t>0.$ Let $M,{\overline M}$ be smooth $n$-dimensional $C^q$ manifolds and let $\phi_j: B_{2t}(0)\subset \rr^n \to M$, $j=1,...,m$ ($m$ possibly $+\infty$), be $C^q$ embeddings and set for every $j=1,\ldots,m$ and $i=0,\dots,N$  $B_i^j\coloneqq \phi_j(B_{(2-i/{N})t}(0))$. Suppose that $M \subset \cup_j \phi_j(B_{t}(0))$ and that  for  every $j_1, j_2=1,...,m$
	\begin{equation}\label{patch1}
		\|D(\phi_{j_1}^{-1}\circ\phi_{j_2})\|,|\partial_{ij}(\phi_{j_1}^{-1}\circ\phi_{j_2})_k|t\le 
		L,
	\end{equation}
	on  the domain of definition of $\phi_{j_1}^{-1}\circ\phi_{j_2}$ (if non-empty). Moreover suppose that we can partition the set of indices $\{1,...,m\}$ into sets $I_1,I_2...,I_{N}$ such that for every $j,\bar j\in I_k$ we have $B^j \cap B^{\bar j}_0 = \emptyset$. Finally suppose that there exists a family of $C^q$ embeddings $h_j: B_0^j \to {\overline M}$ for $j=1,\dots,m$ such that for every couple of indexes $j_1 \in I_h,j_2 \in I_l$  with $1\le  h<l$ for which $B_{l-1}^{j_1} \cap B_{l-1}^{j_2}\neq \emptyset$  we have
	\begin{equation}\label{patch2}
		h_{j_1}(B_{l-2}^{j_1} \cap B_{l-2}^{j_2})\subset h_{j_2}(B_0^{j_2}),
	\end{equation}
	\begin{equation}\label{patch3}
		\| H_{j_1j_2}-{\sf{id}}\|_{C^1\left (\phi_{j_2}^{-1}(B_{l-2}^{j_1} \cap B_{l-2}^{j_2})\right ),t}\le \eps.
	\end{equation}
	for some $\eps \le \eps_3$, where $H_{j_1j_2}=\phi_{j_2}^{-1}\circ h_{j_2}^{-1}\circ h_{j_1}\circ \phi_{j_2}$ (that is well defined by \eqref{patch2}). 
	
	Then there exists a $C^q$-immersion  $h : M \to {\overline M}$ such that $h$ is obtained by modifying slightly the maps $h_j$ in the following sense
	\begin{equation}\label{g}
		h=h_j \circ \phi_j \circ H_j \circ \phi_j^{-1}
	\end{equation}
	in $\phi_j(B_{t}(0))$ for every $j \in I_k$ where $H_j: B_{2t}(0) \to B_{2t}(0)$ is a diffeomorphism of class $C^q$ such that $\|H_j-{\sf{id}}\|_{C^1(B_{2t}(0)),t}\le C_3\eps.$
\end{theorem}
\begin{remark}
	The above theorem does not appear in \cite{Cheeger} as it is stated here. The main difference is that here we have an infinite number of maps and that we have also an explicit $C^1$ control on the error of the modification, which will be crucial in our application in Section \ref{sec:manifoldtomanifold}. This forces us  also  to assume a $C^2$ control on the charts (see \eqref{patch1}), which is not present in \cite{Cheeger}. 
	For these reasons, even if the proof is analogous to the one in \cite{Cheeger}, we will include it here for completeness.\fr
\end{remark}

We can now move to the proof  which builds upon the technical Lemma \ref{magic} proved in the previous section.
\begin{proof}[Proof of Theorem \ref{patching}]
	Before starting, we need to define constants $\eta_i$ for $i=1,...,N$ in the following way. Let $C_2=C_2(n,\frac{1}{LN})$, $\delta_2=\delta_2(n,\frac{1}{LN})<1$ be the constants given in Lemma \ref{magic}. Moreover let $c=C(n,L)>1$ be the constant given in Lemma \ref{B4}. Define $\eta_1\coloneqq \eps$ and inductively $\eta_{k+1}\coloneqq \eta_k D$ for some constant $D=D(L,C_2,N,n)>1$, big enough, to be determined later. Moreover we take $\eps_3=\eps_3(n,N,L)$ small enough to satisfy 
	\begin{equation}\label{etak}
		\eps_3 \le \frac{\delta_2}{cLNC_2D^N}.
	\end{equation}
	Before moving to the main body of the proof, we state a preliminary technical claim, which elementary proof will be given at the end.
	
	\noindent\textbf{Claim:} Suppose that $B^j_r\cap B^{\bar j}_r\neq \emptyset$ for some $r \ge 2$ and some indices $j, \bar j$. Define the set $\Omega_{r}\coloneqq \phi_j^{-1}(B^j_r\cap B^{\bar j}_r)$. Then the set $\{x \in B_{2t}(0) \ | \ \sfd(x,\Omega_r)<\frac{t}{LN}\}$ is contained in the domain of $\phi_{\bar j}^{-1}\circ \phi_{ j}$.\\

	We can now pass to the core of the argument. We will construct by induction $C^q$ maps $$\hat h_k : \bigcup_{\substack{{j \in I_i}\\{i=1,..,k}}} B^j_{k}\to {\overline M}$$ with the following properties. For every $h < k$
	\begin{equation}\label{ind1}
		\hat h_k=\hat h_h, \quad \text{ in } \bigcup_{\substack{{j \in I_i}\\{i=1,..,h}}} B^j_{k}.
	\end{equation}
	 Moreover for every $j_1 \in I_k, j_2 \in I_l$ with $l>k$ such that $B_{l-1}^{j_1} \cap B_{l-1}^{j_2}\neq \emptyset$ then
	\begin{equation}\label{ind2}
		\hat h_k\left(B_{l-1}^{j_1} \cap B_{l-1}^{j_2} \right)\subset h_{j_2}(B_0^{j_2}).
	\end{equation}
	Notice that \eqref{ind2} implies that the map $\hat H_{j_1j_2}\coloneqq \phi_{j_2}^{-1}\circ h_{j_2}^{-1}\circ \hat h_{k} \circ \phi_{j_2}$ is a well defined map $ \hat H_{j_1j_2}:\phi_{j_2}^{-1}\left(B_{l-1}^{j_1} \cap B_{l-1}^{j_2} \right) \to B_{2t}(0) $, then we also require
	\begin{equation}\label{ind3}
		\| \hat H_{j_1j_2}-{\sf{id}}\|_{C^1,t}\le \eta_k
	\end{equation}
	on $\phi_{j_2}^{-1}\left(B_{l-1}^{j_1} \cap B_{l-1}^{j_2} \right)$.
	
	To start the induction for $k=1$ we set $\hat h_1\coloneqq h_{j_1}$ in $B_1^{j_1}$ for every $j_1 \in I_1$. Recall that $B_1^{j_1}$, with $j_1 \in I_1$, are all disjoint and hence $\hat h_1$ is well defined. Then we need only to check \eqref{ind2}, \eqref{ind3}, but these are clearly satisfied thanks to \eqref{patch2} and \eqref{patch3}, since $\eta_1=\eps$. 
	
	Suppose now we have constructed maps $\hat h_1,...,\hat h_k$ and consider any $j_{k+1} \in I_{k+1}$. By induction hypothesis for every $j_h \in I_h$ with $h \le k$ by  the map  $\hat H_{j_h,j_{k+1}}\coloneqq \phi_{j_{k+1}}^{-1}\circ h_{j_{k+1}}^{-1}\circ \hat h_{h} \circ \phi_{j_{k+1}}$ is well defined and  satisfies $\| \hat H_{j_h,j_{k+1}}-{\sf{id}}\|_{C^1,t}\le \eta_h\le \eta_k$ on on  $\phi_{j_{k+1}}^{-1}\left(B_{k}^{j_h} \cap B_{k}^{j_{k+1}} \right)$. Moreover from \eqref{ind1} $\hat h_h=\hat h_k$ on $B^{j_h}_k$ for every $j_h$ as above. Thus we can patch the maps $\hat H_{j_h,j_{k+1}}$ together to get a map 
	$\hat H_{j_{k+1}}\coloneqq \phi_{j_{k+1}}^{-1}\circ h_{j_{k+1}}^{-1}\circ \hat h_{k} \circ \phi_{j_{k+1}}$ defined on the whole set $U_2\coloneqq \phi_{j_{k+1}}^{-1}\left((\bigcup_{i=1}^k \bigcup_{j\in I_i} B^j_k)\cap B_k^{j_{k+1}}\right)$ and satisfying $\|\hat H_{j_{k+1}}-{\sf{id}} \|_{C^1(U_2),t}\le \eta_k\le \delta_2$ (by \eqref{etak}). Set now $U_1\coloneqq \phi_{j_{k+1}}^{-1}\left((\bigcup_{i=1}^k \bigcup_{j\in I_i} B^j_{k+1})\cap B_{k+1}^{j_{k+1}}\right)$. Clearly $U_1 \subset U_2$ and we also claim that 
	\begin{equation}\label{eq:U1U2}
		\sfd_{\rr^n}(\bar U_1,U_2^c)\ge t\frac{1}{LN}
	\end{equation} 
To prove this is sufficient to show that $$\sfd_{\rr^n}\left(\phi_{j_{k+1}}^{-1}(B_{k+1}^j\cap B_k^{j_{k+1}}),(\phi_{j_{k+1}}^{-1}((B_{k}^j)^c\cap B_k^{j_{k+1}})\right)\ge t\frac{1}{LN}, 
	$$ 
	for every $j \in J_i$, and $i=1,...,k$. This can be  seen using \eqref{patch1}. Indeed suppose the above is false, then there exist $x \in B_{(2-(k+1)/N)t}(0)$ and $y\in \phi_{j_{k+1}}^{-1}((B_{k}^j)^c\cap B_k^{j_{k+1}})$ such that $\sfd_{\rr^n}(\phi_{j_{k+1}}^{-1}(\phi_j(x)),y)<t\frac{1}{LN}.$ Then the {\bf Claim} above implies that the whole segment joining $\phi_{j_{k+1}}^{-1}(\phi_j(x))$ and $y$ is in the domain of $\phi_j^{-1}\circ \phi_{j_{k+1}}$. Hence by \eqref{patch1} we must have that $\sfd_{\rr^n}(x,\phi_j^{-1}(\phi_{j_{k+1}}(y)))<t/N$, however by construction $\phi_j^{-1}(\phi_{j_{k+1}}(y)) \in B_{(2-k/N)t}(0)^c$ which is a contradiction. This prove \eqref{eq:U1U2}. Therefore we can apply Lemma \ref{magic} with $U_1,U_2, H=\hat H_{j_{k+1}},\eps= \eta_{k},\eta=\frac{1}{LN}, t=t $. Thus we obtain, after an obvious restriction, a $C^q$-diffeomorphism  $\tilde H_{j_{k+1}} : B_{2t}(0)\to B_{2t}(0)$ such that $\tilde H_{j_{k+1}}|_{U_1}=\hat H_{j_{k+1}}|_{U_1}$ and 
	\begin{equation}\label{ind4}
		\|\tilde H_{j_{k+1}}-{\sf{id}} \|_{C^1(B_{2t}(0)),t}\le C_2\eta_{k}<\frac{1}{NLc},
	\end{equation}
	where the last inequality follows by \eqref{etak}.
	We now define the function
	\begin{equation}\label{modification}
		\tilde h_{j_{k+1}}\coloneqq h_{j_{k+1}} \circ \phi_{j_{k+1}}\circ \tilde H_{j_{k+1}} \circ \phi_{j_{k+1}}^{-1},
	\end{equation}
	on $B_0^{j_{k+1}}$.
	Clearly  $\tilde h_{j_{k+1}}=\hat h_k$ on $\phi_{j_{k+1}}(U_1)=(\bigcup_{i=1}^k \bigcup_{j\in I_i} B^j_{k+1})\cap B_{k+1}^{j_{k+1}}.$ Repeat now the above construction and define maps $\tilde h_{j_{k+1}}$ for every $j_{k+1}\in I_{k+1}$. From the previous observation and the fact that $B_0^{j_{k+1}}\cap B_0^{\overline{j_{k+1}}} =\emptyset $ for every $j_{k+1},\overline{j_{k+1}} \in I_{k+1}$ the map
	\begin{equation}\label{ind5}
		\hat h_{k+1}:=
		\begin{cases}
			\hat h_k  &  \text{ on } \bigcup_{i=1}^k \bigcup_{j\in I_i} B^j_{k+1},\\
			\tilde h_{j_{k+1}}  &  \text{ on } B_{k+1}^{j_{k+1}} \, \text{ for } j_{k+1} \in I_{k+1},
		\end{cases}
	\end{equation}
	is a well defined $C^q$ map $\hat h_{k+1}: \bigcup_{i=1}^{k+1} \bigcup_{j\in I_i} B^j_{k+1} \to \overline{M} $. Moreover it is clear from the definition and from the induction hypothesis that \eqref{ind1} is verified for $\hat h_{k+1}.$ We need now to verify \eqref{ind2} and \eqref{ind3}. Observe that it is enough to check these two conditions for $j_1 \in I_{k+1}, j_2 \in I_{l}$ with $l>k+1$, since in the other cases they are true by \eqref{ind5} and induction hypothesis.

	For \eqref{ind2}  take $j_1 \in I_{k+1}, j_2 \in I_{l}$ with $l>k+1$, such that $B_{l-1}^{j_1} \cap B_{l-1}^{j_2}\neq \emptyset$ and pick $x$ in such set. In particular $x=\phi_{j_1}(y)=\phi_{j_2}(z)$ for some $y,z \in B_{(2-\frac{l-1}{N})t}(0)$. Since $B_{l-1}^{j_1} \subset B_{k+1}^{j_1}$, $\hat h_{k+1}(x)=\tilde h_{j_1}(x)$. Thus from \eqref{ind4} we have that
	\begin{equation}\label{b91}
		\tilde h_{j_1}(x)=h_{j_{1}} \circ \phi_{j_{1}}\circ \tilde H_{j_{1}}(y) \in h_{j_{1}}(\phi_{j_{1}}(B_{(2-\frac{l-2}{N})t}(0)))= h_{j_{1}} (B^{j_1}_{l-2}).
	\end{equation}
	On the other hand we have that
	\[\tilde h_{j_1}(x)= h_{j_{1}} \circ \phi_{j_{1}}\circ \tilde H_{j_{1}} \circ \phi_{j_{1}}^{-1}(\phi_{j_2}(z)),\]
	moreover from the \textbf{Claim} and \eqref{ind4}, we deduce that $\tilde H_{j_{1}}(\phi_{j_{1}}^{-1}(x))$ is in the domain of $\phi_{j_{1}} \circ \phi_{j_{2}}^{-1}$. Therefore we can write 
	\[\tilde h_{j_1}(x)= h_{j_{1}}  \circ \phi_{j_{2}} \circ (\phi_{j_{2}}^{-1} \circ \phi_{j_{1}})\circ \tilde H_{j_{1}} \circ (\phi_{j_{1}}^{-1} \circ \phi_{j_2})(z).
	\]
	From the above, using \eqref{ind4} and Lemma \ref{B4} we deduce that
	\begin{equation}\label{b92}
		\tilde h_{j_1}(x)\in h_{j_{1}}(\phi_{j_{2}}(B_{(2-\frac{l-2}{N})t}(0)))= h_{j_{1}} (B^{j_2}_{l-2}).
	\end{equation}
	Then \eqref{ind2} follows combining \eqref{b91} and \eqref{b92}, using \eqref{patch2}. It remains to show \eqref{ind3} for $j_1 \in I_{k+1}, j_2 \in I_{l}$ for $l> k+1.$ Observe that by \eqref{ind2} that we just proved we can define the map  $\hat H_{j_1j_2}$ on $\phi_{j_2}^{-1}(B_{l-1}^{j_1} \cap B_{l-1}^{j_2} )$ as above. Moreover since $\phi_{j_2}^{-1}(B_{l-1}^{j_1} \cap B_{l-1}^{j_2} )\subset \phi_{j_2}^{-1}(B_{k+1}^{j_1})$ we deduce from \eqref{ind5}
	\begin{align*}
		&\hat H_{j_1j_2}=\phi_{j_2}^{-1}\circ h_{j_2}^{-1}\circ \hat h_{k} \circ \phi_{j_2}=\phi_{j_2}^{-1}\circ h_{j_2}^{-1}\circ \tilde h_{j_1} \circ \phi_{j_2}=\\
		&=\phi_{j_2}^{-1}\circ h_{j_2}^{-1}\circ ( h_{j_1}\circ \phi_{j_1} \circ \tilde H_{j_1} \circ \phi_{j_1}^{-1})  \circ \phi_{j_2}=\\
		&=(\phi_{j_2}^{-1}\circ h_{j_2}^{-1}\circ  h_{j_1} \circ \phi_{j_2}) \circ \phi_{j_2}^{-1} \circ \phi_{j_1}\circ \tilde H_{j_1} \circ \phi_{j_1}^{-1} \circ \phi_{j_2}=\\
		&=H_{j_1j_2} \circ (\phi_{j_2}^{-1} \circ \phi_{j_1})\circ  \tilde H_{j_1} \circ (\phi_{j_1}^{-1} \circ \phi_{j_2}),
	\end{align*}
	on $\phi_{j_2}^{-1}(B_{l-1}^{j_1} \cap B_{l-1}^{j_2} ),$ where $H_{j_1j_2}$ is the map in the hypothesis of the theorem. (Observe that in order to plug in the term $\phi_{j_2} \circ \phi_{j_2}^{-1}$ in the third line, we used the \textbf{Claim} as above). Set now $f\coloneqq \phi_{j_2}^{-1} \circ \phi_{j_1}$  and $g\coloneqq f \circ \tilde H_{j_1} \circ f^{-1}$. With this notation  $\hat H_{j_1j_2}=H_{j_1j_2} \circ g$. From \eqref{ind4}, \eqref{patch1} and Lemma \ref{B4} we have that
	\begin{equation}\label{gbound}
		\|g-{\sf{id}}\|_{C^1,t}\le cC_2\eta_k<1,
	\end{equation}
on $\phi_{j_2}^{-1}(B_{l-1}^{j_1} \cap B_{l-1}^{j_2} ).$
	Therefore combining \eqref{gbound}, \eqref{patch3} and Lemma \ref{iterateclosest} we obtain 
	\[ \|\hat H_{j_1j_2}-{\sf{id}}\|_{C^1,t}\le \tilde C(\eta_k+\eps) \le \eta_k( \tilde C+1),\]
	on $\phi_{j_2}^{-1}(B_{l-1}^{j_1} \cap B_{l-1}^{j_2} ),$ where $\tilde C$ is constant depending only on $C_2$ and $L$. Hence \eqref{ind3} follows provided we choose $D\ge  \tilde C+1.$
	
	Since the induction procedure is now proved, we can define the required map as $h:=\hat h_N$. The assumption $M \subset \cup_j \phi_j(B_{t}(0))$ and \eqref{ind5} grant that $h$ is defined on the whole $M$ and it of class $C^q$. Moreover \eqref{ind4},\eqref{modification} and \eqref{ind5}  imply \eqref{g} together with the bound
	\[\|\tilde H_{j}-{\sf{id}} \|_{C^1(B_{2t}(0)),t}\le C_2\eta_{N}\le \eps D^NC_2,	\]
	hence it is sufficient to take $C_3\ge D^NC_2$. Recall also that $\hat h_1\coloneqq h_{j_1}$ in $B_1^{j_1}$ for every $j_1 \in I_1$. Thus, up to proving the {\bf Claim}, the proof  is concluded.
	
	\textbf{Proof of the Claim}: Define the sets $\Omega_{r-1}:=\phi_j^{-1}(B^j_{r-1}\cap B^{\bar j}_{r-1})$,  $\Omega_{0}:=\phi_j^{-1}(B^j_{0}\cap B^{\bar j}_{0})$, which are both contained in the domain of $\phi_j^{-1}\circ \phi_{\bar j}$. Thus it is enough to prove that 
	$$\{x \in B_{2t}(0) \ | \ \sfd(x,\Omega_r)<\frac{t}{LN}\}\subset \Omega_{r-1}.$$
	Pick $x \in B_{2t}(0) \setminus \Omega_{r-1}$ and and set $\rho:=\sfd(x,\Omega_r)>0$. Fix $\eps>0$ arbitrary. There exists a point $y \in \Omega_r$ such that $|y-x|\le \rho+\eps.$  Moreover, since $y \in \Omega_r$, $x \in \Omega_{r-1}^c$ and $\Omega_r \subset \Omega_{r-1}$, then there must be at least one point $p$, lying in the segment joining $x$ and $y$, with $p\in \partial \Omega_{r-1}$. However, since the maps $\phi_j,\phi_{\bar j}$ are homeomorphisms, we must have that $\Omega_{r-1}\subset \subset \Omega_0$ and 
	\[\partial \Omega_{r-1}\subset \partial B_{(2-\frac{r-1}{N})t}(0)\cup \phi_j^{-1}\circ \phi_{\bar j}(\partial B_{(2-\frac{r-1}{N})t}(0)). \]
	Moreover $\Omega_r\subset B_{(2-\frac{r}{N})t}(0)\cap \phi_j^{-1}\circ \phi_{\bar j}(B_{(2-\frac{r}{N})t}(0))$. Hence if $p \in \partial B_{(2-\frac{r-1}{N})t}(0)$ it holds $|p-x|\ge t/N\ge 1/(LN)$. If instead $p\in \phi_j^{-1}\circ \phi_{\bar j}(\partial B_{(2-\frac{r-1}{N})t}(0)),$ since the map $\phi_{\bar j}^{-1}\circ \phi_{j}$ is $L$-Lipschitz by \eqref{patch1}, we must have that
	\[|x-y|\ge|x-p|\ge \frac{t}{LN}.\]
	Hence $\rho+\eps \ge \frac{t}{LN}$ and from the arbitrariness of $\eps$ we conclude.
\end{proof}

\end{document}